\documentclass[12pt,reqno]{amsart}
\usepackage{amsmath,amsthm,amssymb,enumitem,verbatim,stmaryrd,xcolor,microtype,graphicx,aliascnt}
\usepackage{url}
\usepackage{mathtools,bbm}
\usepackage{booktabs}

\usepackage[bookmarksdepth=2,linktoc=page,colorlinks,linkcolor={red!70!black},citecolor={red!80!black},urlcolor={blue!80!black},pdfpagemode=UseOutlines,pdfstartview={fitH}]{hyperref}

\usepackage{amsfonts}

\newlist{primenumerate}{enumerate}{1}
\setlist[primenumerate,1]{label={(\arabic*$'$)}}

\usepackage{etoolbox}
\makeatletter
\patchcmd{\@startsection}{\@afterindenttrue}{\@afterindentfalse}{}{}             
\patchcmd{\part}{\bfseries}{\bfseries\LARGE}{}{}
\patchcmd{\section}{\scshape}{\bfseries}{}{}\renewcommand{\@secnumfont}{\bfseries} 
\patchcmd{\@settitle}{\uppercasenonmath\@title}{\large}{}{}
\patchcmd{\@setauthors}{\MakeUppercase}{}{}{}
\makeatother

\usepackage{fancyhdr}

\usepackage[T1]{fontenc}
\usepackage[english]{babel} 
\usepackage[top=3.5cm,bottom=3.55cm,left=3.5cm,right=3.5cm]{geometry}

\usepackage[mathcal]{euscript}                               
\usepackage{cleveref}

 \usepackage{mathptmx}                                        
\usepackage[colorinlistoftodos,bordercolor=orange,backgroundcolor=orange!20,linecolor=orange,textsize=footnotesize,textwidth=23mm]{todonotes} \makeatletter \providecommand \@dotsep{5} \def\listtodoname{List of Todos} \def\listoftodos{\@starttoc{tdo}\listtodoname} \makeatother 

\allowdisplaybreaks 
\usepackage{tikz}
\usepackage{tikz-cd}
\usetikzlibrary{calc,arrows.meta,decorations.markings}

\tikzset{
  every picture/.style={>=Stealth, line cap=round, line join=round},
  point/.style={circle, fill=black, inner sep=1.2pt},
  chord/.style={very thick},
  chordSOne/.style={chord, red!70},
  chordSTwo/.style={chord, blue!70},
  chordSThree/.style={chord, teal!70!black},
  diameter/.style={chord, gray!70, dashed},
  circ/.style={gray!25, fill=gray!3},
  aux/.style={gray!60, dashed}
}

\newtheorem{theorem}{Theorem}[section]
\newtheorem{lemma}[theorem]{Lemma}
\newtheorem{proposition}[theorem]{Proposition}
\newtheorem{corollary}[theorem]{Corollary}
\theoremstyle{definition}
\newtheorem{definition}[theorem]{Definition}
\newtheorem{remark}[theorem]{Remark}
\newtheorem{example}[theorem]{Example}
\newtheorem{question}[theorem]{Question}
\newtheorem{conjecture}[theorem]{Conjecture}

\newcommand{\1}{\mathbf{1}}
\newcommand{\diag}{\operatorname{diag}}

\newcommand{\tr}{\operatorname{tr}}

\DeclareMathOperator{\upN}{N}   
\DeclareMathOperator{\upL}{L}   
\DeclareMathOperator{\upR}{R}   
\DeclareMathOperator{\Gr}{Gr}

\DeclareMathOperator{\Span}{Span}

\DeclareMathOperator{\supp}{supp}

\newcommand\C{{\mathbb C}}

\newcommand\N{{\mathbb N}}

\renewcommand\P{{\mathbb P}}

\newcommand\R{{\mathbb R}}

\newcommand\T{{\mathbb T}}

\renewcommand\max{\textup{max}}
\renewcommand{\min}{\textup{min}}
\renewcommand\geq{\geqslant}
\renewcommand\leq{\leqslant}

\renewcommand{\setminus}{\backslash}

\renewcommand\emptyset\varnothing

\newcommand{\Hd}{\mathcal{H}^{d}}
\newcommand{\Htwo}{\mathcal{H}^{2}}

\title{Lorentzian polynomials and matroids over triangular hyperfields\\[10pt] \normalsize Part 2: Analytic aspects}

\author{Matthew Baker}
\address{\rm Matthew Baker, Georgia Institute of Technology}
\email{mbaker@math.gatech.edu}

\author{June Huh}
\address{\rm June Huh, Princeton University and Korea Institute for Advanced Study}
\email{huh@princeton.edu}

\author{Mario Kummer}
\address{\rm Mario Kummer, Technische Universit\"at Dresden}
\email{mario.kummer@tu-dresden.de}

\author{Oliver Lorscheid}
\address{\rm Oliver Lorscheid, University of Groningen}
\email{o.lorscheid@rug.nl}

\hypersetup{pdftitle={Lorentzian polynomials and matroids over triangular hyperfields. Part 2},pdfauthor={Matthew Baker, June Huh, Mario Kummer and Oliver Lorscheid}}

\begin{document}

\begin{abstract}
Br\"and\'en and Huh showed in \cite{Branden-Huh20} that Lorentzian polynomials provide a unifying framework
for Hodge--Riemann relations in combinatorics. In particular, they proved that the support of every Lorentzian polynomial is an M-convex set, and conversely that every M-convex set supports a Lorentzian polynomial. 

In subsequent work, Baker--Huh--Kummer--Lorscheid showed in \cite{BHKL1} that for every $q>0$, the projectivized space $\P\upL_J$ of Lorentzian polynomials with support $J$ is homeomorphic to the thin Schubert cell $\Gr^{\rm w}_J(\T_q)$ of weak representations of $J$ over the generalized triangular hyperfield $\T_q$. Their proof relies crucially on the foundational containment
\[
\Gr^{\rm w}_J(\T_0)\subseteq \P\upL_J.
\]

In this paper, we study the quantitative relationship between Lorentzian polynomials and representations over triangular hyperfields. We prove that for every matroid $M$, there exists a constant $q>0$ depending on $M$ such that
\[
\Gr^{\rm w}_M(\T_q) \subseteq \P\upL_M \subseteq \Gr^{\rm w}_M(\T_2).
\]
Thus the (projectivized) space of Lorentzian polynomials with support $M$ is sandwiched between two thin Schubert cells, both of which are homeomorphic to $\P\upL_M$ itself.

More generally, for every M-convex set $J$, we prove a normalized sandwich theorem of the form
\[
\upN\Gr^{\rm w}_J(\T_q)\subseteq \P\upL_J \subseteq \upN\Gr^{\rm w}_J(\T_2)
\]
for some $q>0$ depending on $J$, where $\upN$ denotes the normalization operator.

We also study the extremal invariant
\[
q(M):=\sup\{q>0 \mid \Gr^{\rm w}_M(\T_q)\subseteq \P\upL_M\}.
\]

For $q(n):=q(U_{2,n})$, we prove that $q(4)=2$ and $q(5)=\log_2 3$, and we give matching upper and lower
bounds of order $1/n$, showing that $q(n)=\Theta(1/n)$. In particular, no universal positive
lower bound for $q(n)$ exists.

Finally, in an appendix, we discuss the tree-metric input underlying the foundational containment $\Gr^{\rm w}_J(\T_0)\subseteq \P\upL_J$.
\end{abstract}

\maketitle

\tableofcontents

\section{Introduction}

Lorentzian polynomials, introduced by Br\"and\'en and Huh in \cite{Branden-Huh20}, form a remarkable class of homogeneous polynomials with nonnegative coefficients satisfying strong Hodge--Riemann type inequalities. One of the basic structural results of \cite{Branden-Huh20} is that the support of every Lorentzian polynomial is an M-convex set, and conversely that every M-convex set supports a Lorentzian polynomial.

In \cite{BHKL1}, the authors established a close topological relationship between Lorentzian polynomials and representations over generalized triangular hyperfields. More precisely, if $J\subseteq \Delta_n^d$ is an M-convex set and $q>0$, then the projectivized space $\P\upL_J$ of Lorentzian polynomials with support $J$ is homeomorphic to the weak thin Schubert cell $\Gr^{\rm w}_J(\T_q)$. 

We will use the following normalization operator throughout the paper.  If
$\rho=(\rho_\alpha)_{\alpha\in J}$ is a projective coefficient vector supported
on $J\subseteq \Delta_n^d$, set
\[
\alpha!:=\alpha_1!\cdots \alpha_n!
\]
and define
\[
\upN(\rho)_\alpha:=\frac{\rho_\alpha}{\alpha!}.
\]
Equivalently, $\upN$ sends the homogeneous polynomial
$\sum_{\alpha\in J}\rho_\alpha x^\alpha$
to $\sum_{\alpha\in J}\rho_\alpha \frac{x^\alpha}{\alpha!}$.

The starting point for the present paper is the following pair of containments.
First, Br\"and\'en and Huh proved in \cite{Branden-Huh20} that
\[
\Gr^{\rm w}_J(\T_0)\subseteq \P\upL_J
\]
for every M-convex set $J$.  Second, a direct calculation shows that
\[
\P\upL_J\subseteq \upN\Gr^{\rm w}_J(\T_2)
\]
for every $J$; in the multi-affine case, $\alpha!=1$ for every
$\alpha\in J$, so $\upN$ acts trivially after projectivization and this
simplifies to
\[
\P\upL_M\subseteq \Gr^{\rm w}_M(\T_2)
\]
for every matroid $M$.

Since the weak representation spaces $\Gr_J^{\rm w}(\T_q)$ are monotone in $q$,
it is natural to ask whether a given Lorentzian stratum $\P\upL_J$ always sits
between two triangular-representation spaces with $q>0$.  The main result of this
paper shows that this is indeed the case after applying the natural normalization
operator.

\begin{theorem}[Normalized sandwich theorem]
\label{thm:intro-normalized-sandwich}
For every M-convex set $J\subseteq \Delta_n^d$, there exists $q>0$ (depending on $J$) such that
\[
\upN\Gr_J^{\rm w}(\T_q)
\subseteq
\P\upL_J
\subseteq
\upN\Gr_J^{\rm w}(\T_2).
\]
\end{theorem}

When $J$ is the set of bases of a matroid $M$, all monomials are squarefree, so
$\upN$ is the identity operator and 
Theorem~\ref{thm:intro-normalized-sandwich} specializes to the following matroid
sandwich theorem: for every matroid $M$, there exists $q>0$ (depending on $M$) such that
\[
\Gr_M^{\rm w}(\T_q)
\subseteq
\P\upL_M
\subseteq
\Gr_M^{\rm w}(\T_2).
\]

The necessity of the normalization operator already appears in the bivariate case.
Indeed, for $J=\Delta_2^d$ with $d\ge 2$, the inclusion
\[
\Gr_J^{\rm w}(\T_q)\subseteq \P\upL_J
\]
fails for every $q>0$. 
Thus the normalized statement is the natural general form of the sandwich theorem beyond the multi-affine setting.

\medskip

A central theme of the paper is that the entire problem reduces to the case where $J$ is the rank-$2$ uniform matroid $U_{2,n}$. In that case, weak $\T_q$-representations correspond to symmetric zero-diagonal matrices with positive off-diagonal entries satisfying certain 
``quartet inequalities'', and Lorentzianity becomes a statement about the signature of such matrices. Our key analytic input is the following theorem.

\begin{theorem}[Rank-$2$ uniform case]
\label{thm:intro-U2n}
Let $n\ge 4$, and define
\[
\varepsilon(n):=\log_2\!\left(1+\frac{1}{n-2}\right).
\]
If $A$ is a symmetric zero-diagonal matrix with positive off-diagonal entries satisfying the $\T_q$-inequalities for some $0<q\le \varepsilon(n)$, then $A$ is Lorentzian. 
\end{theorem}

The proof of Theorem~\ref{thm:intro-U2n} is metric in nature. After diagonal normalization, the $\T_q$-inequalities imply that the $1/q$-powers of the matrix entries define a metric on a finite set. A quantitative finite-metric embedding theorem, in the spirit of Schoenberg theory, then implies that an appropriate power of this metric is Euclidean. This produces a squared Euclidean distance matrix whose associated Cayley--Menger type matrix has Lorentzian signature. 

Once the case of $U_{2,n}$ has been established, the general matroid case follows by
reducing to rank $2$ via contractions and then simplifying rank-$2$ matroids to uniform ones.
The case of arbitrary $M$-convex sets is then reduced to the matroid case via a weighted generalization of the ``natural matroid'' construction.

\medskip

The sandwich theorem naturally leads to a new quantitative invariant.

\begin{definition}
For a matroid $M$, define
\[
q(M):=\sup\{q>0 \mid \Gr^{\rm w}_M(\T_q)\subseteq \P\upL_M\}.
\]
For the rank-$2$ uniform matroid $U_{2,n}$, we write
\[
q(n):=q(U_{2,n}).
\]
\end{definition}

This invariant is nontrivial already in small rank. We prove that
\[
q(4)=2
\qquad\text{and}\qquad
q(5)=\log_2 3.
\]
We also prove explicit upper bounds on $q(n)$, showing in particular that $q(n)\to 0$ as $n\to\infty$. In particular, there is no universal choice of $q>0$ valid for all matroids.

More precisely, we prove the following.

\begin{proposition}
\label{prop:intro-upper-bound}
For every $n\ge 4$,
\[
q(n)\le
\begin{cases}
2\log_2\!\left(\frac{n}{n-2}\right), & \text{if } n \text{ is even},\\[1ex]
\log_2\!\left(\frac{n+1}{n-3}\right), & \text{if } n \text{ is odd}.
\end{cases}
\]
In particular, $q(n)=O(1/n)$ as $n\to\infty$.
\end{proposition}

We also conjecture the following:

\begin{conjecture}
\label{conj:intro-sharp}
The upper bounds in Proposition~\ref{prop:intro-upper-bound} are sharp for every
$n\ge 4$.
\end{conjecture}

The first interesting test case for the conjecture is $n=5$. Our general lower bound gives only $q(5)\ge \log_2(4/3)$,
whereas the sharp value predicted by Conjecture~\ref{conj:intro-sharp} is $q(5)=\log_2 3$, which we prove in Section~\ref{sec:q5}.

\medskip

Finally, we include an appendix devoted to the foundational containment
\[
\Gr^{\rm w}_J(\T_0)\subseteq \P\upL_J.
\]
In the rank-$2$ uniform case this becomes the assertion that tree distance matrices are strictly Lorentzian. Since this fact plays a central role both in \cite{Branden-Huh20} and in the present paper, and since it admits proofs from several quite different viewpoints, we have chosen to collect several different arguments: via Schoenberg embedding theory, via split decompositions, via determinant formulas for tree distance matrices, via potential theory on metric graphs, and via tropical geometry.

\medskip

The paper is organized as follows. 
In Section~2 we review the necessary background on Lorentzian polynomials,
triangular hyperfields, weak representations, and the basic matrix criterion for
Lorentzian signature used throughout the paper. In Section~3 we prove the universal upper containment $\P\upL_M\subseteq \Gr^{\rm w}_M(\T_2)$, together with its normalized form for general M-convex sets. Section~4 contains the analytic heart of the paper, the proof of the sandwich theorem for $U_{2,n}$, together with a summary of the metric-embedding results which enter the argument. Section~5 deduces the general sandwich theorem for matroids and M-convex sets, and also discusses a further $\upN_t$-interpolation theorem. Section~6 studies upper bounds for $q(n)$ and formulates the sharpness conjecture. Section~7 proves the sharp value $q(5)=\log_2 3$. 
Section~8 discusses the strong-representation setting and records several examples
showing that strong representability behaves differently from weak
representability. Appendix~A contains the various proofs of the foundational containment mentioned above.

\subsection{Acknowledgements} 

We thank Donggyu Kim for a number of helpful discussions and for his careful proofreading of this manuscript. We also thank David Renshaw for his help in formalizing the proof of Theorem~\ref{thm:q5-four-point-ptolemaic} into Lean. 

\subsection{Statement on AI usage}
\label{sec:AI-usage} 

We acknowledge the use of AI during the preparation of this manuscript.
Most notably, the proof of Theorem~\ref{thm:q5-four-point-ptolemaic} is due to ChatGPT 5.5, with some conceptual simplifications due to Claude Opus 4.8 and the authors. 
With the help of David Renshaw, we were able to get Claude Code to auto-formalize Theorem~\ref{thm:q5-four-point-ptolemaic} and Corollary~\ref{cor:q5-sharp} in Lean.
ChatGPT 5.5 also came up with the idea of using the weighted natural matroid construction in Section~\ref{sec:general-sandwich}. 
Parts of this manuscript were initially drafted by ChatGPT 5.5 and then proofread, revised, and checked line-by-line by the authors.
The authors have reviewed and take full responsibility for all content in this paper.

\section{Background and conventions}

In this section we collect the definitions and background results needed in the sequel.

\subsection{Lorentzian matrices}

We begin with some linear-algebraic preliminaries.

\begin{definition}
A real symmetric matrix $A\in \R^{n\times n}$ is \emph{zero-diagonal} if $A_{ii}=0$
for all $i$. 
\end{definition}

\begin{remark}
For a homogeneous quadratic polynomial $f$, the Hessian matrix $H_f$ is zero-diagonal
if and only if $f$ is multi-affine.
\end{remark}

\begin{definition}
A real symmetric matrix $A\in \R^{n\times n}$ is \emph{Lorentzian} if
\begin{enumerate}
\item $A_{ij}\ge 0$ for all $i,j$, and
\item $A$ has at most one positive eigenvalue.
\end{enumerate}
It is \emph{strictly Lorentzian} if $A$ has nonnegative entries and inertia
$(1,n-1,0)$.
\end{definition}

\begin{definition}
A real symmetric matrix $A\in \R^{n\times n}$ is \emph{conditionally negative
semidefinite} (CND) if
\[
x^\top A x \le 0
\qquad\text{for all }x\in \R^n \text{ with } \sum_{i=1}^n x_i = 0.
\]
It is \emph{conditionally strictly negative definite} if, in addition,
\[
x^\top A x = 0 \text{ and } \sum_{i=1}^n x_i=0
\qquad\Longrightarrow\qquad
x=0.
\]
\end{definition}

We will use the following criterion for Lorentzianity.

\begin{theorem}[{\cite[Lemma 2.2]{RSW}}]
\label{thm:principal-minor-criterion}
Let $A\in \R^{n\times n}$ be a real symmetric matrix with nonnegative entries. Then the
following are equivalent:
\begin{enumerate}
\item $A$ is Lorentzian.
\item For every principal submatrix $N$ of $A$ of size $k\ge 1$,
\[
(-1)^k \det(N)\le 0.
\]
\end{enumerate}
\end{theorem}

\subsection{Lorentzian polynomials and $M$-convex sets}

We now recall some basic definitions from \cite{Branden-Huh20}.

Let $n\ge 1$ and $d\ge 2$. We denote by $\Hd_n$ the space of homogeneous polynomials
of degree $d$ in $x_1,\dots,x_n$ with nonnegative real coefficients.
Define
\[
\Delta_n^d := \bigl\{\alpha\in \N^n : \alpha_1+\cdots+\alpha_n = d\bigr\}
\]
and
\[
\binom{[n]}{d} := \{0,1\}^n \cap \Delta_n^d.
\]
For $\alpha=(\alpha_1,\dots,\alpha_n)\in \Delta_n^d$, we write
\[
x^\alpha := x_1^{\alpha_1}\cdots x_n^{\alpha_n},
\qquad
\alpha! := \alpha_1!\cdots \alpha_n!.
\]

We will usually write a polynomial $f\in \Hd_n$ in the normalized form
\[
f=\sum_{\alpha\in \Delta_n^d} c_\alpha \frac{x^\alpha}{\alpha!}.
\]
Its support is
\[
\supp(f):=\{\alpha\in \Delta_n^d : c_\alpha\neq 0\}.
\]
We say that $f$ is \emph{multi-affine} if $\supp(f)\subseteq \binom{[n]}{d}$.

For $f\in \Hd_n$, we write $\partial_i f:=\frac{\partial f}{\partial x_i}$, and more
generally, for $\alpha\in \Delta_n^k$ with $0\le k\le d-2$,
\[
\partial^\alpha f := \prod_{i=1}^n \partial_i^{\alpha_i} f.
\]
The \emph{Hessian} of $f$ is the symmetric matrix
\[
H_f := \bigl(\partial_i \partial_j f\bigr)_{1\le i,j\le n}.
\]

\begin{definition}
Let $\mathring{\upL}_n^2$ be the space of homogeneous quadratic polynomials
$f\in \Htwo_n$ whose Hessian $H_f$ is strictly Lorentzian.
A homogeneous polynomial $f\in \Hd_n$ is \emph{strictly Lorentzian} if
\[
\partial^\alpha f \in \mathring{\upL}_n^2
\qquad\text{for all }\alpha\in \Delta_n^{d-2}.
\]
A \emph{Lorentzian polynomial} is a limit of strictly Lorentzian polynomials.
We write $\upL_n^d$ and $\mathring{\upL}_n^d$ for the spaces of Lorentzian and strictly
Lorentzian polynomials of degree $d$, respectively.
\end{definition}

\begin{definition}
A nonempty subset $J\subseteq \Delta_n^d$ is \emph{$M$-convex} if for all
$\alpha,\beta\in J$ and every $i\in [n]$ with $\alpha_i<\beta_i$, there exists
$j\in [n]$ with $\alpha_j>\beta_j$ such that
\[
\alpha+e_i-e_j \in J
\qquad\text{and}\qquad
\beta-e_i+e_j \in J.
\]
\end{definition}

\begin{example}
A matroid $M$ of rank $r$ on $[n]$ determines an $M$-convex subset
\[
J(M):=\Bigl\{\sum_{i\in B} e_i : B \text{ is a basis of }M\Bigr\}\subseteq \Delta_n^r.
\]
Conversely, an $M$-convex set $J$ comes from a matroid if and only if
$J\subseteq \{0,1\}^n$.
For simplicity, we will often identify a matroid with its associated $M$-convex set.
\end{example}

An important link between Lorentzian polynomials and $M$-convexity is the following.

\begin{theorem}[{\cite[Proposition 2.19 and Theorem 2.23]{Branden-Huh20}}]
\label{thm:support-M-convex}
If $f\in \Hd_n$ is Lorentzian, then $\supp(f)$ is $M$-convex.
\end{theorem}

For an $M$-convex set $J\subseteq \Delta_n^d$, we write $\upL_J$ for the space of
Lorentzian polynomials with support $J$, and $\P\upL_J$ for its projectivization.
The next theorem provides a useful “limit-free” criterion for Lorentzianity.

\begin{theorem}[{\cite[Theorem 2.25]{Branden-Huh20}}]
\label{thm:BH-limit-free}
Let $f\in \Hd_n$. Then the following are equivalent:
\begin{enumerate}
\item $f$ is Lorentzian.
\item $\supp(f)$ is $M$-convex and $\partial^\alpha f\in\upL_n^2$ for all
$\alpha\in \Delta_n^{d-2}$.
\end{enumerate}
\end{theorem}

\begin{theorem}[{\cite[Theorem 3.10]{Branden-Huh20}}]
\label{thm:BH-canonical-Lorentzian}
If $J\subseteq \Delta_n^d$ is $M$-convex, then the corresponding exponential generating function
\[
f_J := \sum_{\alpha\in J} \frac{x^\alpha}{\alpha!}
\]
is Lorentzian.
\end{theorem}

In particular, every $M$-convex set supports a canonical Lorentzian polynomial.

\subsection{Tracts and generalized triangular hyperfields}

We now recall the tract-theoretic language used to define representations over triangular hyperfields.

A \emph{pointed monoid} is a commutative multiplicative monoid $F$ with unit $1$ and
a distinguished element $0$ such that $0\cdot a=0$ for all $a\in F$. Its unit group is
\[
F^\times := \{a\in F : ab=1 \text{ for some } b\in F\}.
\]
A \emph{pointed group} is a pointed monoid such that $F^\times = F\setminus\{0\}$. The \emph{ambient semiring} of a pointed group $F$ is the group semiring $F^+ := \N[F^\times]$. An \emph{ideal} of $F^+$ is an additive submonoid $I$ of $F^+$ that is closed under multiplication by $F^+$.

\begin{definition}
A \emph{tract} is a pointed group $F$ together with an ideal\footnote{In \cite{Baker-Bowler19}, the null set of a tract is not required to be closed under addition. However, since all tracts mentioned in the present paper possess this stronger property, it is convenient to use the definition above.} $N_F\subseteq F^+$, called
the \emph{null set}, such that for every $a\in F$ there is a unique $b\in F$ with
$a+b\in N_F$.
\end{definition}

The most basic example is an ordinary field, whose null set
consists of all formal sums that add to zero in the usual sense.

The triangular hyperfield $\T_1$ and its deformations $\T_q$ for $q\ge 0$, as defined below, were
introduced by Viro \cite{Viro10}.

\begin{definition}
The \emph{triangular hyperfield} $\T_1$ is the tract $\R_{\ge 0}$ with its usual
multiplication and with null set
\[
N_{\T_1}
=
\Bigl\{\sum_i a_i \;\Big|\; a_1,\dots,a_n \text{ are the side lengths of a (possibly degenerate) Euclidean }
n\text{-gon}\Bigr\}.
\]
It follows from \cite[Lemma 3.5]{BHKL1} that $N_{\T_1}$, as defined above, is in fact an ideal.
\end{definition}

In particular, $a+b+c\in N_{\T_1}$ if and only if $a,b,c$ are the side lengths of a
(possibly degenerate) Euclidean triangle. 

\begin{lemma}
\label{lem:Heron-criterion}
Let $a,b,c\in \R_{\ge 0}$. Then the following are equivalent:
\begin{enumerate}
\item $a+b+c\in N_{\T_1}$.
\item There exists a (possibly degenerate) Euclidean triangle with side lengths
$a,b,c$.
\item
$a^4+b^4+c^4 \le 2(a^2b^2+a^2c^2+b^2c^2).$
\end{enumerate}
\end{lemma}

\begin{proof}
It suffices to prove the equivalence of (2) and (3).
Without loss of generality, we can assume that $0\leq a\leq b\leq c$ and therefore $a\leq b+c$ and $b\leq a+c$. Thus $a$, $b$ and $c$ are the side lengths of a (possibly degenerate) Euclidean triangle if and only if $c\leq a+b$. This is, in turn, equivalent to
 \begin{multline*}
  2(a^2b^2+a^2c^2+b^2c^2) \ - \ (a^4+b^4+c^4) \ = \\ 
  \underbrace{(a+b+c)}_{\geq0} \,\cdot\, \underbrace{(-a+b+c)}_{\geq0} \,\cdot\, \underbrace{(a-b+c)}_{\geq0} \,\cdot\, (a+b-c) \ \geq \ 0,
 \end{multline*}
 as claimed.
\end{proof}

\begin{remark}
The equivalence of (2) and (3) in Lemma~\ref{lem:Heron-criterion} is
closely related to Heron's classical formula for the area of a triangle.
If a triangle has side lengths $a,b,c$ and semiperimeter
$s=\frac{a+b+c}{2}$, then Heron's formula says that its area $A$ satisfies
$A^2=s(s-a)(s-b)(s-c)$.
The proof of Lemma~\ref{lem:Heron-criterion} shows that (3) is equivalent to the statement that the right-hand side of Heron's formula is non-negative.
\end{remark}

\begin{definition}
Let $q>0$. The \emph{$q$-triangular hyperfield} $\T_q$ is the tract $\R_{\ge 0}$
whose null set consists of those formal sums $\sum_i a_i$ such that
$a_1^{1/q},\dots,a_n^{1/q}$ are the side lengths of a Euclidean $n$-gon.
\end{definition}

\begin{corollary}
\label{cor:Tq-Heron}
Let $q>0$. Then $a+b+c\in N_{\T_q}$ if and only if
\[
a^{4/q}+b^{4/q}+c^{4/q}
\le
2\bigl(a^{2/q}b^{2/q}+a^{2/q}c^{2/q}+b^{2/q}c^{2/q}\bigr).
\]
\end{corollary}

\begin{definition}
The \emph{tropical hyperfield} $\T_0$ is the tract with null set
\[
N_{\T_0}:=\bigcap_{q>0} N_{\T_q},
\]
and the \emph{degenerate triangular hyperfield} $\T_\infty$ is the tract with null set
\[
N_{\T_\infty}:=\bigcup_{q>0} N_{\T_q}.
\]
\end{definition}

We will use the following concrete description from \cite{BHKL1}.

\begin{lemma}[{\cite[Lemma 3.9]{BHKL1}}]
\label{lem:T0-Tinfty}
Let $a_1,\dots,a_n\in \R_{\ge 0}$. Then:
\begin{enumerate}
\item $\sum_i a_i \in N_{\T_0}$ if and only if the sum is identically zero or the
maximum among the $a_i$ is attained at least twice.
\item $\sum_i a_i \in N_{\T_\infty}$ if and only if the sum is identically zero, or the
maximum is attained at least twice, or at least three of the $a_i$ are nonzero.
\end{enumerate}
\end{lemma}

\subsection{Weak representations}

We now recall the notion of weak representation for $M$-convex sets, following
\cite{BHKL0}.

\begin{definition}
Let $J\subseteq \Delta_n^d$ be an $M$-convex set. Define
\[
\delta_J^- := \inf(J)\in \N^n,
\qquad
\delta_J^+ := \sup(J)\in \N^n,
\]
coordinatewise.
\end{definition}

\begin{definition}
Let $F$ be a tract and let $J\subseteq \Delta_n^d$ be an $M$-convex set.
A function
\[
\rho:\Delta_n^d \to F
\]
is a \emph{weak $F$-representation} of $J$ if:
\begin{enumerate}
\item $\supp(\rho)=J$, and
\item for every $\alpha\in \Delta_n^{d-2}$ and all $1\le i\le j\le k\le \ell\le n$
with
\[
\delta_J^- \le \alpha
\qquad\text{and}\qquad
\alpha + e_i+e_j+e_k+e_\ell \le \delta_J^+,
\]
the $3$-term Pl\"ucker relation
\begin{samepage}
\begin{multline*}
\rho(\alpha+e_j+e_k)\rho(\alpha+e_i+e_\ell)
\ - \
\rho(\alpha+e_i+e_k)\rho(\alpha+e_j+e_\ell)\\
\ + \
\rho(\alpha+e_i+e_j)\rho(\alpha+e_k+e_\ell)
\ \in \ N_F
\end{multline*}
\end{samepage}
holds.
\end{enumerate}
\end{definition}

\begin{definition}
Let $F$ be a tract and let $J\subseteq \Delta_n^d$ be an M-convex set. We write
$\upR_J^{\rm w}(F)$ for the set of weak $F$-representations of $J$. The multiplicative group $F^\times$ 
acts on $\upR_J^{\rm w}(F)$ by
\[
(t \cdot \rho)(\alpha) = t \cdot \rho(\alpha).
\]
We define the \emph{weak thin Schubert cell} $\Gr_J^{\rm w}(F)$ to be the quotient
\[
\Gr_J^{\rm w}(F):=\upR_J^{\rm w}(F)/F^\times .
\]
\end{definition}

Throughout this paper, we work primarily with weak representations. 
(We comment on the strong representation setting in Section~\ref{sec:strong-representations}.)
We will systematically identify a weak representation with its associated
exponential generating function $f_\rho$.

\begin{definition}
Let $J\subseteq \Delta_n^d$ be an $M$-convex set, let $F$ be a tract with underlying multiplicative group $\R_{\ge 0}$, and let
$\rho\in \upR_J^{\rm w}(F)$ be a weak representation of $J$ over $F$. We define
\[
f_\rho := \sum_{\alpha\in J} \rho(\alpha)\frac{x^\alpha}{\alpha!}.
\]
Note that when $J$ is a matroid, this is just the usual multi-affine generating polynomial of the
coefficient function $\rho$.
\end{definition}

\section{The universal upper containment}
\label{sec:universal-upper}

Let $F$ be a tract with underlying multiplicative monoid $\R_{\ge 0}$, and recall from the Introduction that the normalization operator $\upN$ sends the homogeneous polynomial
$\sum_{\alpha\in J}\rho_\alpha x^\alpha$ with coefficients in $F$
to $\sum_{\alpha\in J}\rho_\alpha \frac{x^\alpha}{\alpha!}$.
It induces a self-map on the projective space of homogeneous polynomials with support $J$ modulo the action of $F^\times$. 
The goal of this section is to prove the upper bound in the sandwich theorem.

\begin{theorem}[Universal upper containment]
\label{thm:universal-upper-containment}
Let $J\subseteq \Delta_n^d$ be an M-convex set. Then
\[
\upL_J \subseteq \upN \upR_J^{\rm w}(\T_2).
\]
Equivalently, after projectivizing we have
\[
\P\upL_J \subseteq \upN \Gr_J^{\rm w}(\T_2).
\]
In particular, if $J=M$ is a matroid, then $\upN$ acts trivially and
\[
\P\upL_M \subseteq \Gr_M^{\rm w}(\T_2).
\]
\end{theorem}

We first prove an \emph{a priori} weaker unnormalized containment; the normalized statement will then follow by applying the same argument to
$\upN^{-1}f$.  

\subsection{The degenerate cases $U_{2,2}$ and $U_{2,3}$}

The rank-$2$ uniform matroids on $2$ and $3$ elements are handled by the following lemma.

\begin{lemma}
\label{lem:U22-U23}
For $n=2,3$ and every $q\ge 0$, one has
\[
\upL_{U_{2,n}} = \upR_{U_{2,n}}^{\rm w}(\T_q),
\qquad
\P\upL_{U_{2,n}} = \Gr_{U_{2,n}}^{\rm w}(\T_q).
\]
\end{lemma}

\begin{proof}
For $n=2$, the matroid $U_{2,2}$ has a single basis, so there are no nontrivial
Pl\"ucker relations and the weak representation space consists of a single point up to
scaling. The corresponding quadratic form is $c\,x_1x_2$, whose Hessian
\[
\begin{pmatrix}
0 & c\\
c & 0
\end{pmatrix}
\]
has exactly one positive eigenvalue when $c>0$. Thus it is Lorentzian.

For $n=3$, there are again no nontrivial $3$-term Pl\"ucker relations. A quadratic form
with support $U_{2,3}$ has Hessian
\[
A=
\begin{pmatrix}
0 & a & b\\
a & 0 & c\\
b & c & 0
\end{pmatrix}
\qquad (a,b,c>0).
\]
All principal minors of $A$ satisfy the sign conditions of
Theorem~\ref{thm:principal-minor-criterion}: the $1\times 1$ minors vanish, the
$2\times 2$ principal minors are negative, and
\[
\det(A)=2abc>0.
\]
Hence $A$ is Lorentzian.
\end{proof}

\subsection{The case of $U_{2,4}$}

The first nontrivial case is $U_{2,4}$, where it turns out that satisfying the unique $3$-term Pl\"ucker relation over $\T_2$ is
exactly equivalent to Lorentzianity.

\begin{proposition}
\label{prop:U24}
We have
\[
\upL_{U_{2,4}} = \upR_{U_{2,4}}^{\rm w}(\T_2),
\qquad
\P\upL_{U_{2,4}} = \Gr_{U_{2,4}}^{\rm w}(\T_2).
\]
\end{proposition}

\begin{proof}
A quadratic form with non-negative coefficients and support $U_{2,4}$ corresponds to a zero-diagonal
symmetric matrix
\[
A=
\begin{pmatrix}
0 & a & b & c\\
a & 0 & d & e\\
b & d & 0 & f\\
c & e & f & 0
\end{pmatrix},
\qquad a,b,c,d,e,f>0.
\]
By Theorem~\ref{thm:principal-minor-criterion}, such a matrix is Lorentzian if and only if
all principal minors have the alternating sign pattern. The conditions for principal minors
of size $1,2,3$ are automatic, so the only nontrivial condition is
\[
\det(A)\le 0.
\]
A direct computation gives
\[
\det(A)
=
(a f)^2 + (b e)^2 + (c d)^2
-2\bigl(a f\, b e + a f\, c d + b e\, c d\bigr).
\]
Hence $\det(A)\le 0$ is equivalent to
\[
(a f)^2 + (b e)^2 + (c d)^2
\le
2\bigl(a f\, b e + a f\, c d + b e\, c d\bigr).
\]
By Corollary~\ref{cor:Tq-Heron} with $q=2$, this is exactly the condition that
\[
\sqrt{af},\ \sqrt{be},\ \sqrt{cd}
\]
are the side lengths of a (possibly degenerate) Euclidean triangle, i.e., exactly the weak
$\T_2$-Pl\"ucker relation for $U_{2,4}$.
\end{proof}

\subsection{The unnormalized upper containment}

We now prove the unnormalized upper containment.
The following lemma is standard; we provide a proof for the reader's convenience.

\begin{lemma}
\label{lem:diagonal-deletion}
Let $H$ be a real symmetric matrix with nonnegative entries and at most one
positive eigenvalue.  Let $D$ be a nonnegative diagonal matrix such that
$0\le D_{ii}\le H_{ii}$ for all $i$.  Then $H-D$ also has nonnegative
entries and at most one positive eigenvalue.
\end{lemma}

\begin{proof}
The assertion about entries is immediate.  Since $D\succeq 0$, we have
\[
        x^T(H-D)x \le x^T Hx
        \qquad\text{for all }x.
\]
If $H-D$ had at least two positive eigenvalues, then its quadratic form would
be positive definite on some two-dimensional subspace.  The displayed
inequality would then imply that the quadratic form associated to $H$ is also
positive definite on this same subspace, contradicting the fact that $H$ has
at most one positive eigenvalue.  Hence $H-D$ has at most one positive
eigenvalue.
\end{proof}

\begin{proposition}
\label{prop:unnormalized-upper-containment}
Let $J\subset \Delta_n^d$ be an $M$-convex set. Then
\[
\upL_J\subseteq \upR_J^{\rm w}(\mathbb T_2).
\]
Consequently, after projectivizing,
\[
\mathbb P \upL_J\subseteq \operatorname{Gr}_J^{\rm w}(\mathbb T_2).
\]
\end{proposition}

\begin{proof}
Let $f\in \upL_J$, and write
\[
f=\sum_{\beta\in J}\rho(\beta)\frac{x^\beta}{\beta!}.
\]
By Theorem~\ref{thm:BH-limit-free}, for every $\alpha\in\Delta_n^{d-2}$, the quadratic derivative
$q := \partial^\alpha f$ is Lorentzian. We show that the coefficient function $\rho$ satisfies all weak
$\mathbb T_2$-Plücker relations.

Write $q$ in the normalized quadratic basis as
\[
q
=
\sum_{i=1}^n a_i \frac{x_i^2}{2}
+
\sum_{1\le i<j\le n} b_{ij} x_i x_j,
\]
and let
\[
H=
\begin{pmatrix}
a_1 & b_{12} & \cdots & b_{1n}\\
b_{12} & a_2 & \cdots & b_{2n}\\
\vdots & \vdots & \ddots & \vdots\\
b_{1n} & b_{2n} & \cdots & a_n
\end{pmatrix}
\]
be the Hessian of $q$. Since $q$ is Lorentzian, the matrix $H$ is Lorentzian.

Fix indices $1\le i\le j\le k\le \ell\le n$. The corresponding weak $\T_2$-relation is
the assertion that the three nonnegative numbers
\[
T_1:=\rho_q(e_j+e_k)\rho_q(e_i+e_\ell),\qquad
T_2:=\rho_q(e_i+e_k)\rho_q(e_j+e_\ell),\qquad
T_3:=\rho_q(e_i+e_j)\rho_q(e_k+e_\ell)
\]
satisfy the $\T_2$-triangle condition, i.e., that
$\sqrt{T_1},\ \sqrt{T_2},\ \sqrt{T_3}$
are the side lengths of a Euclidean triangle.
We distinguish four different cases.

\smallskip

\emph{Case 1: $i,j,k,\ell$ are pairwise distinct.}
Then only off-diagonal coefficients occur, and the relevant quadratic is the multi-affine
part of the restriction of $q$ to the four variables $x_i,x_j,x_k,x_\ell$. 
The multi-affine part is obtained by deleting the square terms; by
Lemma~\ref{lem:diagonal-deletion}, this operation preserves Lorentzianity (see also \cite[Corollary 3.5]{Branden-Huh20}). Therefore 
the resulting $4$-variable multi-affine quadratic is Lorentzian.
By Proposition~\ref{prop:U24}, its coefficients satisfy the weak $\T_2$-relation.
Equivalently, the triple
\[
b_{jk}b_{i\ell},\qquad b_{ik}b_{j\ell},\qquad b_{ij}b_{k\ell}
\]
satisfies the required $\T_2$-condition.

\smallskip

\emph{Case 2: exactly one index is repeated.}

Suppose next that exactly one index is repeated.  Up to relabeling, the three
terms in the corresponding weak Plücker relation have the form
\[
bc,\qquad bc,\qquad ad .
\]
More concretely, these quantities arise from a quadratic Lorentzian polynomial
whose Hessian has a principal submatrix on three indices with entries $a,b,c,d$
in the corresponding positions.  Deleting the two diagonal entries not involved
in the repeated index, Lemma~\ref{lem:diagonal-deletion} shows that the matrix
\[
M=
\begin{pmatrix}
a & b & c\\
b & 0 & d\\
c & d & 0
\end{pmatrix}
\]
still has nonnegative entries and at most one positive eigenvalue.  Hence
$\det M\ge 0$.  Since
\[
\det M=d(2bc-ad),
\]
we obtain $ad\le 2bc$ if $d>0$, while the case $d=0$ is immediate.
Therefore $ad\le 4bc$, which is precisely the required weak $\T_2$-relation
for the three terms $bc,bc,ad$.

\smallskip

\emph{Case 3: two pairs of indices are repeated.}
Then the three terms are of the form
\[
b^2,\qquad b^2,\qquad ac,
\]
coming from a $2\times 2$ principal submatrix
\[
N=
\begin{pmatrix}
a & b\\
b & c
\end{pmatrix}
\]
of $H$. Since $H$ is Lorentzian, so is $N$, and
Theorem~\ref{thm:principal-minor-criterion} gives
\[
\det(N)=ac-b^2\le 0.
\]
Thus $ac\le b^2$, hence a fortiori $ac\le 4b^2$, which is exactly the required
$\T_2$-relation.

\smallskip

\emph{Case 4: at least three indices are equal.}
Then all three terms coincide, so the
weak $\T_2$-relation is automatic.

\smallskip

It follows that $\rho$ is a weak $\mathbb T_2$-representation of $J$ and thus
\[
\upL_J\subseteq \upR_J^{\rm w}(\mathbb T_2).\qedhere    
\]
\end{proof}

\subsection{The normalized upper containment}\label{sec:defnorm}

Recall that $\upN$ is the linear operator on homogeneous polynomials defined on monomials by
\[
\upN(x^\alpha)=\frac{x^\alpha}{\alpha!}.
\]
The following lemma shows that $\upN$ preserves weak $\T_q$-representations for every
$q\ge 0$.

\begin{lemma}
\label{lem:N-preserves-weak-reps}
Let $J\subseteq \Delta_n^d$ be an M-convex set and let $0\le q\le \infty$. If
\[
\rho:\Delta_n^d\to \R_{\ge 0}
\]
is a weak $\T_q$-representation of $J$, then so is
\[
\rho^\sharp(\beta):=\frac{\rho(\beta)}{\beta!}.
\]
Equivalently,
\[
\upN \upR_J^{\rm w}(\T_q)\subseteq \upR_J^{\rm w}(\T_q).
\]
\end{lemma}

\begin{proof}
It is enough to treat the case $0<q<\infty$; the endpoint cases
$q=0$ and $q=\infty$ follow from the corresponding definitions by the
same argument, with the triangle inequality interpreted in the tropical
limit.

Fix $\alpha\in \Delta_n^{d-2}$ and indices
$1\le i\le j\le k\le \ell\le n$ such that
\[
\delta_J^- \le \alpha
\qquad\text{and}\qquad
\alpha+e_i+e_j+e_k+e_\ell \le \delta_J^+.
\]
For $\rho$, the three Plücker products are
\[
T_1=\rho(\alpha+e_j+e_k)\rho(\alpha+e_i+e_\ell),
\]
\[
T_2=\rho(\alpha+e_i+e_k)\rho(\alpha+e_j+e_\ell),
\]
and
\[
T_3=\rho(\alpha+e_i+e_j)\rho(\alpha+e_k+e_\ell).
\]
Since $\rho$ is a weak $\T_q$-representation, the three numbers
\[
T_1^{1/q},\qquad T_2^{1/q},\qquad T_3^{1/q}
\]
satisfy the triangle inequality.

For $\rho^\sharp$, the corresponding Plücker products are
\[
T_r^\sharp=\frac{T_r}{D_r}\qquad (r=1,2,3),
\]
where
\[
D_1=(\alpha+e_j+e_k)!(\alpha+e_i+e_\ell)!,
\]
\[
D_2=(\alpha+e_i+e_k)!(\alpha+e_j+e_\ell)!,
\]
and
\[
D_3=(\alpha+e_i+e_j)!(\alpha+e_k+e_\ell)!.
\]

If $i,j,k,\ell$ are pairwise distinct, then the same four unit increments
$e_i,e_j,e_k,e_\ell$ occur once in each denominator product. Hence
\[
D_1=D_2=D_3,
\]
so the three Plücker products are all divided by the same positive scalar.
The weak $\T_q$-relation is therefore unchanged.

It remains to consider the case where some of $i,j,k,\ell$ coincide.  In
that case at least two of the three products $T_1,T_2,T_3$ are equal.  We
record the only nontrivial possibility explicitly.  Suppose, for example, that
$i=j<k<\ell$.  Then
\[
T_1=T_2=\rho(\alpha+e_i+e_k)\rho(\alpha+e_i+e_\ell),
\]
while
\[
T_3=\rho(\alpha+2e_i)\rho(\alpha+e_k+e_\ell).
\]
Thus the weak $\T_q$-relation for $\rho$ says precisely that
$T_3 \le 2^q T_1$,
since the two repeated side lengths are $T_1^{1/q}=T_2^{1/q}$.  

The denominators satisfy $D_1=D_2$ and $D_3\ge D_1.$
Indeed, all coordinates except the $i$-th contribute equally to $D_1$ and
$D_3$, while in the $i$-th coordinate the contribution to $D_3$ is
\[
(\alpha_i+2)!\,\alpha_i!,
\]
whereas the contribution to $D_1$ is
\[
(\alpha_i+1)!^2.
\]
The inequality
\[
(\alpha_i+1)!^2\le (\alpha_i+2)!\,\alpha_i!
\]
is the log-convexity of the factorial sequence. Hence
\[
T_3^\sharp=\frac{T_3}{D_3}
\le
\frac{T_3}{D_1}
\le
2^q\frac{T_1}{D_1}
=
2^q T_1^\sharp.
\]
Since $T_1^\sharp=T_2^\sharp$, this is exactly the weak $\T_q$-relation
for $\rho^\sharp$.

The cases $i<j=k<\ell$ and $i<j<k=\ell$ are identical, with the repeated
index playing the same role.  If two pairs coincide, say $i=j<k=\ell$, then
\[
T_1=T_2=\rho(\alpha+e_i+e_k)^2
\]
and
\[
T_3=\rho(\alpha+2e_i)\rho(\alpha+2e_k).
\]
Again $D_1=D_2$, and $D_3\ge D_1$, now by applying
\[
(m+1)!^2\le (m+2)!\,m!
\]
in the $i$- and $k$-coordinates.  The same argument gives
\[
T_3^\sharp\le 2^qT_1^\sharp.
\]
Finally, if three or four of the indices coincide, then all three Plücker
products and all three denominators are equal, so the relation is automatic.

Thus every weak $\T_q$-relation satisfied by $\rho$ is also satisfied by
$\rho^\sharp$, and therefore $\rho^\sharp$ is a weak
$\T_q$-representation of $J$.
\end{proof}

In view of Lemma~\ref{lem:N-preserves-weak-reps}, the following result strengthens the upper containment in
Proposition~\ref{prop:unnormalized-upper-containment}.

\begin{proposition}
\label{prop:normalized-upper}
Let $J\subseteq \Delta_n^d$ be an M-convex set. Then
\[
\upL_J \subseteq \upN \upR_J^{\rm w}(\T_2).
\]
Equivalently, after projectivizing,
\[
\mathbb P \upL_J \subseteq \upN \Gr_J^{\rm w}(\T_2).
\]
\end{proposition}

\begin{proof}
Let $f\in \upL_J$. We must show that $\upN^{-1}f$ is a weak $\T_2$-representation
of $J$.

Fix $\alpha\in \Delta_n^{d-2}$ and set
\[
g:=\partial^\alpha(\upN^{-1}f).
\]
Since weak $\T_2$-relations are preserved by multiplying a quadratic by a positive scalar
and by positive diagonal rescaling of the variables, it is enough to verify them for
\[
h:=\frac{1}{\alpha!}\,
g\!\left(\frac{x_1}{\alpha_1+1},\dots,\frac{x_n}{\alpha_n+1}\right).
\]

Write
\[
\partial^\alpha f
=
\sum_{i=1}^n c_{ii}\frac{x_i^2}{2}
+
\sum_{1\le i<j\le n} c_{ij}x_i x_j.
\]
Since $f$ is Lorentzian, the quadratic $\partial^\alpha f$ is Lorentzian by
Theorem~\ref{thm:BH-limit-free}. A direct computation shows that
\[
h
=
\sum_{i=1}^n
\frac{\alpha_i+2}{\alpha_i+1}\,c_{ii}\frac{x_i^2}{2}
+
\sum_{1\le i<j\le n} c_{ij}x_i x_j.
\]
Thus the off-diagonal coefficients are unchanged, while the diagonal coefficient in the
$i$-th variable is multiplied by
\[
\lambda_i:=\frac{\alpha_i+2}{\alpha_i+1}\le 2.
\]

We now verify the weak $\T_2$-relations for $h$.

\smallskip

\emph{Case 1: the four indices are pairwise distinct.}
Then only off-diagonal coefficients appear, so the relevant relation is exactly the same as
for $\partial^\alpha f$. As in the proof of Proposition~\ref{prop:unnormalized-upper-containment},
this follows from Proposition~\ref{prop:U24} applied to the multi-affine part of
the restriction to those four variables.

\smallskip

\emph{Case 2: exactly one index is repeated.}

Suppose that exactly one index is repeated.  Up to relabeling, the three terms
have the form
\[
bc,\qquad bc,\qquad \lambda_i ad,
\]
where
\[
\lambda_i=\frac{\alpha_i+2}{\alpha_i+1}\le 2.
\]
As in the proof of Proposition~\ref{prop:unnormalized-upper-containment},
we delete the two irrelevant diagonal entries from the corresponding
$3\times 3$ principal submatrix of the Hessian.  By
Lemma~\ref{lem:diagonal-deletion}, the resulting matrix
\[
M=
\begin{pmatrix}
a & b & c\\
b & 0 & d\\
c & d & 0
\end{pmatrix}
\]
still has at most one positive eigenvalue.  Thus $\det M=d(2bc-ad)\ge 0$,
and hence $ad\le 2bc$.  Therefore
\[
\lambda_i ad\le 2ad\le 4bc,
\]
which is the desired weak $\T_2$-relation.
\smallskip

\emph{Case 3: two pairs of indices are repeated.}
Then the three terms are
\[
c_{ik}^2,\qquad c_{ik}^2,\qquad \lambda_i\lambda_k c_{ii}c_{kk}.
\]
The corresponding $2\times 2$ principal submatrix is
\[
N=
\begin{pmatrix}
c_{ii} & c_{ik}\\
c_{ik} & c_{kk}
\end{pmatrix},
\]
and Lorentzianity gives
\[
\det(N)=c_{ii}c_{kk}-c_{ik}^2\le 0.
\]
Thus
\[
c_{ii}c_{kk}\le c_{ik}^2.
\]
Since $\lambda_i\lambda_k\le 4$, it follows that
\[
\lambda_i\lambda_k c_{ii}c_{kk}\le 4c_{ik}^2,
\]
which is exactly the weak $\T_2$-relation.

\smallskip

All remaining index configurations are tautological. Hence $h$, and therefore also
$g$, satisfies every weak $\T_2$-Pl\"ucker relation. It follows that $\upN^{-1}f$ is a
weak $\T_2$-representation of $J$, i.e.
\[
f\in \upN \upR_J^{\rm w}(\T_2).\qedhere    
\]
\end{proof}

\begin{proof}[Proof of Theorem~\ref{thm:universal-upper-containment}]
The normalized containment
\[
\upL_J\subseteq \upN \upR_J^{\rm w}(\mathbb T_2)
\]
is precisely Proposition~\ref{prop:normalized-upper}.  Projectivizing gives
\begin{equation*}
\mathbb P \upL_J\subseteq \upN\operatorname{Gr}^{\rm w}_J(\mathbb T_2).\qedhere    
\end{equation*}
\end{proof}

\section{The rank-\texorpdfstring{$2$}{2} uniform case}
\label{sec:U2n-core}

The purpose of this section is to prove the lower inclusion in the sandwich theorem for the
rank-$2$ uniform matroid $U_{2,n}$. This is the analytic heart of the paper; in the next
section, the general case will be reduced to it.
We begin by translating the problem into the language of matrices.

\subsection{Matrix reformulation}

A weak $\T_q$-representation of $U_{2,n}$ is determined by its coefficients
\[
a_{ij} := \rho(e_i+e_j)\qquad (1\le i<j\le n),
\]
which we package into a symmetric matrix with zero diagonal.

\begin{definition}
An $n\times n$ real matrix $A$ is \emph{admissible} if it is symmetric, has zero
diagonal, and has strictly positive off-diagonal entries.
\end{definition}

\begin{definition}
Let $q\ge 0$. We say that an admissible matrix $A$ \emph{satisfies the
$\T_q$-inequalities} if for every quartet $\{i,j,k,\ell\}\subseteq [n]$ with
$i<j<k<\ell$, the three numbers
\[
s_1 := (A_{ij}A_{k\ell})^{1/q},\qquad
s_2 := (A_{ik}A_{j\ell})^{1/q},\qquad
s_3 := (A_{i\ell}A_{jk})^{1/q}
\]
are the side lengths of a (possibly degenerate) Euclidean triangle.
\end{definition}

For $q=0$, this is interpreted in the tropical sense from
Lemma~\ref{lem:T0-Tinfty}: the maximum of the three products
\[
A_{ij}A_{k\ell},\qquad A_{ik}A_{j\ell},\qquad A_{i\ell}A_{jk}
\]
must be attained at least twice.

Under the identification
\[
f_A(x_1,\dots,x_n):=\sum_{1\le i<j\le n} A_{ij}x_ix_j,
\]
the support of $f_A$ is $U_{2,n}$, and $f_A$ is Lorentzian if and only if $A$ is
Lorentzian. Likewise, $f_A$ is a weak $\T_q$-representation of $U_{2,n}$ if and only
if $A$ satisfies the $\T_q$-inequalities. Thus the desired lower inclusion for
$U_{2,n}$ is equivalent to the following matrix statement.

\begin{theorem}
\label{thm:rank-two-uniform}
For $n\ge 4$, set
\[
\varepsilon(n):=\log_2\!\left(1+\frac{1}{n-2}\right).
\]
Let $A\in \R^{n\times n}$ be admissible and suppose that $A$ satisfies the
$\T_{\varepsilon(n)}$-inequalities. Then $A$ is Lorentzian.
\end{theorem}

By the monotonicity of the $\T_q$-conditions in $q$, the same conclusion holds
whenever $0<q\le \varepsilon(n)$.

The rest of the section is devoted to the proof of Theorem~\ref{thm:rank-two-uniform}.

\subsection{Diagonal normalization}

The first step is a simple but important invariance.

\begin{lemma}
\label{lem:diag-congruence-U2n}
Let $D=\diag(d_1,\dots,d_n)$ with $d_i>0$, and set
\[
A' := DAD.
\]
Then:
\begin{enumerate}
\item $A$ is admissible if and only if $A'$ is admissible.
\item $A$ satisfies the $\T_q$-inequalities if and only if $A'$ does.
\item $A$ and $A'$ have the same inertia.
\end{enumerate}
\end{lemma}

\begin{proof}
The first assertion is immediate.

For the second, fix a quartet $\{i,j,k,\ell\}$. Then
\[
(A'_{ij}A'_{k\ell})^{1/q}
=
(d_id_jd_kd_\ell)^{1/q}(A_{ij}A_{k\ell})^{1/q},
\]
and the same common positive factor multiplies the other two quantities appearing in the
$\T_q$-condition. Thus the three side lengths for $A'$ are obtained from those for $A$
by a common rescaling, so the triangle condition is unchanged.

The third assertion is Sylvester's law of inertia.
\end{proof}

\subsection{A brief Schoenberg interlude}

The proof of Theorem~\ref{thm:rank-two-uniform} passes through finite metric geometry. Since
this may be unfamiliar to some readers, we briefly summarize the ideas behind the argument.

Given a finite metric $d$ on a set $X$, a classical theorem of Schoenberg \cite{Schoenberg1938} says that
$d$ is Euclidean if and only if the matrix
\[
E=(d(i,j)^2)_{i,j\in X}
\]
is conditionally negative semidefinite; moreover, strict conditional negative definiteness
corresponds to affine independence of the embedding. 

In the present section we will not use Schoenberg's theorem directly, but rather a
quantitative finite-metric embedding theorem with a close connection to the $\T_q$-relations.

\begin{theorem}
\label{thm:snowflake-embedding}
Let $m\ge 2$, let $(X,d)$ be an $m$-point metric space, and set
\[
\eta(m):=\log_2\!\left(1+\frac{1}{m-1}\right).
\]
If $0\le p<\eta(m)$, then the metric space $(X,d^{p/2})$ admits an
isometric embedding into $\R^{m-1}$ as an affinely independent set.  
\end{theorem}

\begin{proof}
This is exactly Theorem~3.6 of \cite{FaverEtAl}, after translating their notation.
\end{proof}

\begin{remark}
In the statement of Theorem~\ref{thm:snowflake-embedding}, when
$p=0$, we use the convention
\[
d^{0}(x,y)=
\begin{cases}
0, & x=y,\\
1, & x\ne y.
\end{cases}
\]
\end{remark}

The next lemma explains how a squared Euclidean distance matrix gives rise to a
Lorentzian matrix of the form needed in our application.  Recall that the
\emph{Cayley--Menger matrix} associated to a squared Euclidean distance matrix $B$ is
obtained from $B$ by adjoining a top row and left column of ones, with a zero in
the upper-left corner.

\begin{lemma}
\label{lem:squared-edm-block}
Let $B$ be a real symmetric $m\times m$ matrix, and set
\[
        C=
        \begin{pmatrix}
        0&\1^\top\\
        \1&B
        \end{pmatrix}.
\]
Then $C$ has at most one positive eigenvalue if and only if $B$ is
conditionally negative semidefinite on $\1^\perp$, i.e.
\[
        \alpha^\top B\alpha\le 0
        \qquad\text{for all }\alpha\in\R^m\text{ with }\sum_i\alpha_i=0.
\]

In particular, if $v_1,\dots,v_m$ are points in a Euclidean space and
\[
        B=(\|v_i-v_j\|^2)_{1\le i,j\le m},
\]
then $C$ has at most one positive eigenvalue.  
If, moreover, $v_1,\dots,v_m$ are affinely independent, then $C$ is nonsingular and
has inertia $(1,m,0)$.
\end{lemma}

\begin{proof}
Let
\[
        Q(t,\alpha)
        =
        2t\sum_i\alpha_i+\alpha^\top B\alpha
\]
be the quadratic form associated to $C$, where $t\in\R$ and
$\alpha=(\alpha_1,\dots,\alpha_m)\in\R^m$.

First suppose that $B$ is conditionally negative semidefinite on $\1^\perp$.  If
$C$ had two positive eigenvalues, then $Q$ would be positive definite on some
two-dimensional subspace $U\subset \R\oplus\R^m$.  The linear functional $(t,\alpha)\longmapsto \sum_i\alpha_i$
has a nonzero kernel on $U$, so there is a nonzero vector $(t,\alpha)\in U$ with
$\sum_i\alpha_i=0$.  For this vector, $Q(t,\alpha)=\alpha^\top B\alpha\le0$,
contradicting the positive definiteness of $Q$ on $U$.  Thus $C$ has at most one
positive eigenvalue.

Conversely, suppose that $C$ has at most one positive eigenvalue.  We prove the
conditional negative semidefiniteness of $B$ by contradiction.  If there were some
$\alpha\in\1^\perp$ with $\alpha^\top B\alpha>0$, then the vector $u:=(0,\alpha)$
would satisfy $Q(u)>0$.  For $T>0$ set $w_T:=(T,\1)$.
Then
\[
        Q(w_T)=2mT+\1^\top B\1,
\]
which is positive for $T\gg0$, while
\[
        Q(w_T,u)=\1^\top B\alpha
\]
is independent of $T$.  Therefore the determinant of the Gram matrix of $Q$ on
$\operatorname{span}\{w_T,u\}$ is
\[
        Q(w_T)\,Q(u)-Q(w_T,u)^2
        =
        \bigl(2mT+\1^\top B\1\bigr)\alpha^\top B\alpha
        -
        (\1^\top B\alpha)^2,
\]
which is positive for $T\gg0$.  Hence $Q$ is positive definite on a two-dimensional
subspace, so $C$ has at least two positive eigenvalues, a contradiction.  Thus
$\alpha^\top B\alpha\le0$ for all $\alpha\in\1^\perp$.

Now assume $B=(\|v_i-v_j\|^2)$ for points $v_1,\dots,v_m$ in a Euclidean space.  If
$\sum_i\alpha_i=0$, then
\[
\begin{aligned}
        \alpha^\top B\alpha
        &=
        \sum_{i,j}\alpha_i\alpha_j\|v_i-v_j\|^2        \\
        &=
        -2\left\|\sum_i\alpha_i v_i\right\|^2
        \le0.
\end{aligned}
\]
Thus $B$ is conditionally negative semidefinite on $\1^\perp$, and the first part of
the lemma shows that $C$ has at most one positive eigenvalue.

It remains to prove the inertia statement under affine independence.  On the
$(m-1)$-dimensional subspace
\[
        W':=\left\{(0,\alpha):\sum_i\alpha_i=0\right\},
\]
the preceding formula gives
\[
        Q(0,\alpha)
        =
        -2\left\|\sum_i\alpha_i v_i\right\|^2.
\]
By affine independence, this form is negative definite on $W'$.  Hence $C$ has at
least $m-1$ negative eigenvalues.

We next show that $C$ is nonsingular.  Suppose
\[
        C\binom{t}{\alpha}=0.
\]
The first row gives $\sum_i\alpha_i=0$, and the remaining rows give
\[
        t\1+B\alpha=0.
\]
Multiplying on the left by $\alpha^\top$, we obtain
\[
        0
        =
        t\sum_i\alpha_i+\alpha^\top B\alpha
        =
        -2\left\|\sum_i\alpha_i v_i\right\|^2.
\]
Since $\sum_i\alpha_i=0$, affine independence implies $\alpha=0$.  Then $t\1=0$, so
$t=0$.  Thus $C$ is nonsingular.

Finally, $Q(1,0)=0$, so $C$ is not negative definite.  Since $C$ is nonsingular and
has at most one positive eigenvalue, it must have exactly one positive eigenvalue.  As
$C$ has size $m+1$, the remaining $m$ eigenvalues are negative.  Hence the inertia
is $(1,m,0)$.
\end{proof}

\subsection{Proof of the rank-\texorpdfstring{$2$}{2} uniform theorem}

\begin{proof}[Proof of Theorem~\ref{thm:rank-two-uniform}]
We  distinguish two cases.

\smallskip

\noindent\emph{Case 1: $0<q<\varepsilon(n)$.}
Let
\[
D=\diag(1,A_{12}^{-1},\dots,A_{1n}^{-1}),
\qquad
A':=DAD.
\]
By Lemma~\ref{lem:diag-congruence-U2n}, it suffices to prove that $A'$ is Lorentzian.
By construction,
\[
A'_{1i}=1 \qquad (i=2,\dots,n).
\]
Write
\[
A'=
\begin{pmatrix}
0 & \mathbf{1}^\top\\
\mathbf{1} & B
\end{pmatrix},
\]
where $B$ is the principal $(n-1)\times (n-1)$ block indexed by
$X:=\{2,\dots,n\}$.

Define
\[
d(i,j):=B_{ij}^{1/q}\qquad (i\neq j,\ i,j\in X),
\qquad
d(i,i):=0.
\]
We claim that $d$ is a metric on $X$. Indeed, let $i,j,k\in X$ be distinct.
Applying the $\T_q$-condition to the quartet $\{1,i,j,k\}$, and using
$A'_{1i}=A'_{1j}=A'_{1k}=1$, we see that the three numbers
\[
(A'_{1i}A'_{jk})^{1/q}=d(j,k),\qquad
(A'_{1j}A'_{ik})^{1/q}=d(i,k),\qquad
(A'_{1k}A'_{ij})^{1/q}=d(i,j)
\]
are the side lengths of a Euclidean triangle. Hence
\[
d(i,j)\le d(i,k)+d(k,j)
\]
for all distinct $i,j,k\in X$. Since $d(i,j)>0$ for $i\neq j$ and
$d(i,j)=d(j,i)$, it follows that $d$ is a metric.
Since $|X|=n-1$ and
\[
q<\log_2\!\left(1+\frac{1}{n-2}\right)=\eta(n-1),
\]
Theorem~\ref{thm:snowflake-embedding} applies to the metric space $(X,d)$
with $m=n-1$ and $p=q$. We obtain affinely independent vectors
\[
v_2,\dots,v_n\in \R^{\,n-2}
\]
such that
\[
\|v_i-v_j\|
=
d(i,j)^{q/2}
=
B_{ij}^{1/2}
\qquad (i,j\in X).
\]
Equivalently,
\[
B=(\|v_i-v_j\|^2)_{i,j\in X}.
\]
Applying Lemma~\ref{lem:squared-edm-block} with $m=n-1$ to the points
$v_2,\dots,v_n$, we conclude that
\[
A'=
\begin{pmatrix}
0 & \mathbf{1}^\top\\
\mathbf{1} & B
\end{pmatrix}
\]
has inertia $(1,n-1)$. In particular, $A'$ is Lorentzian.

\smallskip
\noindent\emph{Case 2: $q=\varepsilon(n)$.}
    For $0<\epsilon<1$ the matrix $A^{(\epsilon)}:=(A_{ij}^\epsilon)_{i,j}$ satisfies the $\T_p$-inequalities for $p=\epsilon\cdot\varepsilon(n)<\varepsilon(n)$. Thus $A^{(\epsilon)}$ is Lorentzian by Case 1. Because the set of Lorentzian polynomials is closed, it follows that $A=\lim_{\epsilon\to1}A^{(\epsilon)}$ is Lorentzian.
\end{proof}

\begin{remark}
\label{rem:quartets-through-1}
Only the quartets containing the distinguished index $1$ enter the proof.
In particular, to deduce Lorentzianity in Theorem~\ref{thm:rank-two-uniform}, one does not need
the $\T_q$-inequalities for quartets contained entirely in $\{2,\dots,n\}$.
\end{remark}

\subsection{Non-sharpness of the lower bound}
\label{subsec:snowflake-bound-not-sharp}

The lower bound supplied by Theorem~\ref{thm:snowflake-embedding} is sufficient for
the purposes of giving an explicit positive lower bound for $q(n)$, but it is not sharp.  

For example, when $m=4$, Theorem~\ref{thm:snowflake-embedding} gives Euclidean embeddability of
$(X,d^{q/2})$ for
\[
        0 < q<\eta(4)=\log_2(4/3).
\]

A theorem of Blumenthal (see \cite{Blumenthal1936} and \cite[Theorem~2.4]{MaeharaSurvey}) 
gives a much stronger statement: if $(X,d)$ is a four-point metric space, then $(X,d^{q/2})$ Euclidean for $0<q\le1$.
In particular:

\begin{corollary}
$q(5) \geq 1$.
\end{corollary}

The exponent $1$ in Blumenthal's theorem is optimal for arbitrary four-point metric spaces.  Indeed, let
$d$ be the path-distance metric on the four-cycle $C_4$, with cyclically adjacent
distances equal to $1$ and opposite distances equal to $2$:
\[
        d(1,2)=d(2,3)=d(3,4)=d(4,1)=1,
        \qquad
        d(1,3)=d(2,4)=2.
\]
For the zero-sum vector
\[
        c=(1,-1,1,-1),
\]
we have
\[
        \sum_{i,j=1}^4 c_i c_j d(i,j)^q
        =
        4\cdot 2^q-8.
\]
Thus, by Schoenberg's theorem, $(X,d^{q/2})$ fails to be Euclidean embeddable for every $q>1$.

If the metric space $(X,d)$ arises from an admissible $5 \times 5$ matrix satisfying the $\T_q$-quartet inequalities, as in our application, we get more information than just the fact that $d$ is a metric. Recall that, after diagonal normalization, only
the quartets through the distinguished index $1$ are used to prove that $d$ restricts to a metric on the remaining $n-1$ points.  
In Section~\ref{sec:q5}, we show that the full $\T_q$-quartet conditions give more: the associated four-point metric is
\emph{Ptolemaic}.  This extra structure will enable us to improve the lower bound on $q(5)$ from $1$ to the sharp value of $\log_2 3$.

\section{The general sandwich theorem}
\label{sec:general-sandwich}

In this section we prove the general sandwich theorem.  We work throughout with
the generalized normalization operators
\[
\upN_t(x^\alpha)=\frac{x^\alpha}{(\alpha!)^t}
\qquad (t\ge 0),
\]
where $\alpha!:=\alpha_1!\cdots \alpha_n!$.  Thus $\upN_0$ is the identity and
$\upN_1=\upN$ is the normalization operator used earlier.

The main result of the section is the following.

\begin{theorem}[$\upN_t$-sandwich theorem]
\label{thm:Nt-sandwich}
Let $J\subseteq \Delta_n^d$ be M-convex, and let $t>0$.  Then there exists a
constant $q_t(J)>0$ such that
\[
\upN_t\upR_J^{\rm w}(\T_q)\subseteq \upL_J \subseteq \upN_t\upR_J^{\rm w}(\T_p)
\]
for all $0<q\le q_t(J)$ and $p(t) \leq p \leq \infty$, where $p(t):=2\max\{1,t\}$.

In particular, for $t=1$ one has $p(1)=2$, and hence
\[
\upN\upR_J^{\rm w}(\T_q)\subseteq \upL_J
\subseteq
\upN\upR_J^{\rm w}(\T_2)
\]
for all sufficiently small $q>0$.
\end{theorem}

\begin{remark}
The family $\upN_t$ interpolates between the unnormalized and normalized
settings.  When $t=0$, the operator $\upN_t$ is the identity, while
$\upN_1=\upN$ is the normalization operator appearing in the main
sandwich theorem.  Thus the $\upN_t$-sandwich theorem may be viewed as a
one-parameter refinement of the normalized sandwich theorem, compatible with
the two limiting unnormalized containments
\[
\upR_J^{\rm w}(\mathbb T_0)\subseteq \upL_J
\qquad\text{and}\qquad
\upL_J\subseteq \upR_J^{\rm w}(\mathbb T_\infty).
\]
\end{remark}

\subsection{The bivariate case}
We prove Theorem~\ref{thm:Nt-sandwich} in several steps.  We begin with the bivariate case, which explains
why normalization is needed and gives a simple model for the general argument.

Let $J=\Delta_2^d$.  We write
\[
f_\rho(x,y)
=
\sum_{k=0}^d a_k
\frac{x^{d-k}}{(d-k)!}\frac{y^k}{k!},
\qquad a_k:=\rho(d-k,k).
\]

\begin{lemma}
\label{lem:bivariate-lorentzian}
Assume that $a_i \geq 0$ for all $i$. The polynomial $f_\rho$ is Lorentzian if and only if the sequence
$(a_0,\ldots,a_d)$ is log-concave and has no internal zeros.
\end{lemma}

\begin{proof}
Write $f_\rho=\sum_{k=0}^d b_kx^{d-k}y^k$ in the usual monomial basis.  Then
\[
b_k=\frac{a_k}{(d-k)!k!},
\qquad
\frac{b_k}{\binom{d}{k}}=\frac{a_k}{d!}.
\]
The bivariate Lorentzian criterion, see \cite[Example 2.26]{Branden-Huh20}, says that $f_\rho$ is Lorentzian if and only
if the sequence $\bigl(b_k/\binom{d}{k}\bigr)_{k=0}^d$ is log-concave and has no
internal zeros.  This is equivalent to the same condition for $(a_k)$.
\end{proof}

\begin{lemma}
\label{lem:bivariate-Tq-ineq}
Let $\rho:\Delta_2^d\to \R_{\ge 0}$ be a coefficient function, and write
\[
        \rho(d-k,k)=a_k
        \qquad(0\le k\le d).
\]
If $q>0$, then $\rho\in \upR_{\Delta_2^d}^{\rm w}(\T_q)$ if and only if
\[
        a_{k-1}a_{k+1}\le 2^q a_k^2
        \qquad (1\le k\le d-1).
\]
For $q=0$, one has $\rho\in \upR_{\Delta_2^d}^{\rm w}(\T_0)$ if and only if
\[
        a_{k-1}a_{k+1}\le a_k^2
        \qquad (1\le k\le d-1).
\]
\end{lemma}

\begin{proof}
For $1\le k\le d-1$, put
\[
        \alpha=(d-k-1,k-1)\in \Delta_2^{d-2}.
\]
The weak Pl\"ucker relation with repeated indices $(1,1,2,2)$ has three
terms
\[
        a_k^2,\qquad a_k^2,\qquad a_{k-1}a_{k+1}.
\]
Since there are only two variables, every weak three-term Pl\"ucker relation
for $\Delta_2^d$ is obtained in this way, for some $1\le k\le d-1$.

Assume first that $q>0$.  Membership in the null set of $\T_q$ is equivalent
to the assertion that the $q$-th roots of the three nonnegative terms satisfy
the triangle inequalities.  Since two of the terms are equal, the only possibly
nontrivial triangle inequality is
\[
        (a_{k-1}a_{k+1})^{1/q}
        \le
        2(a_k^2)^{1/q},
\]
which is equivalent to
\[
        a_{k-1}a_{k+1}\le 2^q a_k^2.
\]
Thus, for $q>0$, the weak $\T_q$-relations are equivalent to the displayed
inequalities for all $1\le k\le d-1$.

For $q=0$, membership in the null set of $\T_0$ means that the maximum of
the three terms is attained at least twice.  Since two of the terms are equal
to $a_k^2$, this is equivalent to
\[
        a_{k-1}a_{k+1}\le a_k^2.
\]
Therefore the weak $\T_0$-relations are equivalent to the displayed
inequalities for all $1\le k\le d-1$.
\end{proof}

\begin{proposition}[Bivariate normalized inclusion]
\label{prop:bivariate-Nt}
Let $J=\Delta_2^d$ with $d \geq 2$ and let $t>0$.  Set
\[
q_t(d):=
t \cdot \min_{1\le k\le d-1}
\log_2\left(\frac{(k+1)(d-k+1)}{k(d-k)}\right).
\]
Then
\[
\upN_t\upR_J^{\rm w}(\T_q)\subseteq \upL_J
\qquad\text{for all }0<q\le q_t(d).
\]
Moreover, for $t=0$ this conclusion fails for every $q>0$:
in that case
\[
\upR_J^{\rm w}(\T_q)\not\subseteq \upL_J.
\]
On the other hand,
\[
\upR_J^{\rm w}(\T_0)=\upL_J.
\]
\end{proposition}

\begin{proof}
Let $\rho\in \upR_J^{\rm w}(\T_q)$, and write
\[
f_\rho
=
\sum_{k=0}^d
a_k
\frac{x^{d-k}}{(d-k)!}\frac{y^k}{k!}.
\]
Then
\[
\upN_t f_\rho
=
\sum_{k=0}^d
\frac{a_k}{\bigl((d-k)!k!\bigr)^t}
\frac{x^{d-k}}{(d-k)!}\frac{y^k}{k!}.
\]
By Lemma~\ref{lem:bivariate-lorentzian}, it suffices to prove log-concavity of
\[
b_k:=\frac{a_k}{\bigl((d-k)!k!\bigr)^t}.
\]
For $1\le k\le d-1$,
\[
\frac{b_k^2}{b_{k-1}b_{k+1}}
=
\frac{a_k^2}{a_{k-1}a_{k+1}}
\left(
\frac{(k+1)(d-k+1)}{k(d-k)}
\right)^t.
\]
By Lemma~\ref{lem:bivariate-Tq-ineq},
\[
\frac{a_k^2}{a_{k-1}a_{k+1}}\ge 2^{-q}.
\]
Thus $b_k^2\ge b_{k-1}b_{k+1}$ for all $1\le k\le d-1$ whenever
\[
2^q\le
\left(
\frac{(k+1)(d-k+1)}{k(d-k)}
\right)^t
\]
for all $k$, which is exactly the condition $q\le q_t(d)$.  This proves the
first assertion.

The equality
\[
\upR_J^{\rm w}(\T_0)=\upL_J
\]
follows from the $q=0$ case of Lemma~\ref{lem:bivariate-Tq-ineq} together with
Lemma~\ref{lem:bivariate-lorentzian}.

It remains to show that the normalized containment can fail when $t=0$.  Let 
$q>0$ and choose $1\le k\le d-1$. Set
\[
a_j=1 \quad (j\ne k),
\qquad
2^{-q/2}<a_k<1.
\]
Then the weak $\T_q$-inequalities hold. Indeed, the only potentially nontrivial
inequality centered at $k$ is
\[
a_{k-1}a_{k+1}=1\le 2^q a_k^2,
\]
which holds by the choice of $a_k$.  The inequalities centered at $k-1$ and
$k+1$ are stronger than needed, since they have the form
\[
a_{k-2}a_k=a_k<1=a_{k-1}^2
\quad\text{or}\quad
a_ka_{k+2}=a_k<1=a_{k+1}^2
\]
when the indicated indices exist; all other inequalities are identities.
Thus $\rho\in \upR_J^{\rm w}(\T_q)$.  On the other hand,
\[
a_k^2<1=a_{k-1}a_{k+1},
\]
so the coefficient sequence is not log-concave.  By
Lemma~\ref{lem:bivariate-lorentzian}, $\rho\notin \upL_J$.  Since
$\upN_0$ is the identity, this proves that the asserted containment fails for
$t=0$.
\end{proof}

In particular, the minimum in Proposition~\ref{prop:bivariate-Nt} is positive.
For large $d$, it is attained near $k=d/2$, and
\[
q_t(d)
=
t\log_2\left(1+\frac4d+O(d^{-2})\right)
=
\frac{4t}{(\ln 2)d}+O(d^{-2}).
\]

\subsection{The matroid case}
\label{subsec:matroid-reduction}

We next prove the matroid version of the sandwich theorem.  Since matroid
supports are multi-affine, the operators $\upN_t$ act trivially on them.

\begin{theorem}[Matroid sandwich theorem]
\label{thm:matroid-sandwich}
Let $M$ be a matroid on $[n]$.  Then there exists $q_M>0$ such that
\[
\upR_M^{\rm w}(\T_q)\subseteq \upL_M \subseteq \upR_M^{\rm w}(\T_2)
\]
for all $0<q\le q_M$.

Equivalently, after projectivizing,
\[
\Gr_M^{\rm w}(\T_q)\subseteq \mathbb P\upL_M
\subseteq \Gr_M^{\rm w}(\T_2)
\]
for all sufficiently small $q>0$.
\end{theorem}

\begin{proof}
The upper containment is the multi-affine case of the general upper containment
proved in Proposition~\ref{prop:unnormalized-upper-containment}.

For the lower containment, let $\rho\in \upR_M^{\rm w}(\T_q)$.  By the
quadratic-derivative characterization of Lorentzian polynomials, it suffices to
show that every nonzero squarefree derivative of $f_\rho$ of degree $2$ is
Lorentzian.  Such a derivative corresponds to a rank-$2$ contraction of $M$.
Thus it is enough to treat loopless rank-$2$ matroids.

Let $N$ be a loopless rank-$2$ matroid with parallel classes
\[
V_1,\ldots,V_m.
\]
If $m\le 3$, every admissible rank-$2$ coefficient matrix is Lorentzian, so there
is nothing to prove.  Assume $m\ge 4$.  Choose representatives $i_a\in V_a$ and
write
\[
\sigma_{ab}:=\rho(i_ai_b)
\qquad (a\ne b).
\]
The weak Pl\"ucker relations with two parallel elements imply that for
$i,i'\in V_a$ and $j,k$ outside $V_a$ in distinct parallel classes,
\[
\rho(ij)\rho(i'k)=\rho(ik)\rho(i'j).
\]
Hence, after choosing positive scalars $t_i$, one has
\[
\rho(ij)=t_i t_j\sigma_{ab}
\qquad
(i\in V_a,\ j\in V_b,\ a\ne b).
\]
Consequently $f_\rho$ is obtained from the rank-$2$ uniform polynomial
\[
g_\sigma=\sum_{1\le a<b\le m}\sigma_{ab} y_ay_b
\]
by the linear substitution
\[
y_a=\sum_{i\in V_a} t_i x_i.
\]
The coefficient system $\sigma$ is a weak $\T_q$-representation of
$U_{2,m}$.  By Theorem~\ref{thm:rank-two-uniform}, if
\[
0<q\le \varepsilon(m):=\log_2\left(1+\frac1{m-2}\right),
\]
then $g_\sigma$ is Lorentzian.  Since Lorentzian polynomials are preserved under
linear substitutions with nonnegative coefficients, $f_\rho$ is Lorentzian.

Taking the minimum of these positive numbers over the finitely many rank-$2$
contractions of $M$ gives the desired constant $q_M>0$.
\end{proof}

\begin{remark}
Recall that 
\[
q(M):=\sup\{q>0 \mid \Gr^{\rm w}_M(\T_q)\subseteq \P\upL_M\}
\]
and $q(m) := q(U_{2,m})$.
The proof of Theorem~\ref{thm:matroid-sandwich} gives an explicit lower bound for $q(M)$ in terms of the largest rank-$2$
uniform minor of $M$: if  
\[
h(M):=\max\{m:\text{$M$ has a $U_{2,m}$-minor}\}
\]
then
\[
q(M)\ge q(h(M)).
\]

Indeed, $q(\cdot)$ is monotone under minors: if $N$ is a minor of $M$, then
\[
q(N)\ge q(M).
\]
For deletions, this follows by restricting to the corresponding coordinate
subspace.  For contractions, it follows by taking the corresponding squarefree
derivative.  In both cases, weak $\T_q$-representability and Lorentzianity
descend to the minor.

Since $U_{2,m}$ is a minor of $U_{2,m+1}$, the sequence $q(m)$ is
nonincreasing in $m$.  On the other hand, every nonzero squarefree quadratic
derivative in the proof of Theorem~\ref{thm:matroid-sandwich} corresponds to a rank-$2$ contraction of $M$.  After
simplifying that contraction, one obtains $U_{2,m}$ for some $m\le h(M)$.
Thus the worst possible rank-$2$ contraction is controlled by
$U_{2,h(M)}$, giving the displayed lower bound.

For example, if $M$ is binary then it has no $U_{2,4}$-minor; since $q(3)=\infty$ we conclude that $q(M) = \infty$.
And if $M$ is ternary then it has no $U_{2,5}$-minor; since $q(4)=2$ we conclude that $q(M) \geq 2$. 

By Theorem \ref{thm:universal-upper-containment} it follows that $\P\upL_M=\Gr^{\rm w}_M(\T_\infty)=\Gr^{\rm w}_M(\T_2)$ for binary $M$ and $\P\upL_M=\Gr^{\rm w}_M(\T_2)$ for ternary $M$.
\end{remark}

\subsection{Quadratic M-convex supports}
\label{subsec:general-from-matroids}

We now pass from matroids to arbitrary M-convex supports.  The proof is local:
we analyze quadratic derivatives and then use the quadratic-derivative
criterion for Lorentzian polynomials (Theorem~\ref{thm:BH-limit-free}).

Let $K\subseteq \Delta_n^2$ be M-convex.  Define
\[
V(K):=\{i:\text{there exists }j\ne i\text{ with }e_i+e_j\in K\},
\]
\[
E(K):=\{\{i,j\}:i\ne j,\ e_i+e_j\in K\},
\qquad
D(K):=\{i:2e_i\in K\}.
\]
By \cite[Lemma 4.5(2)]{BHKL1}, the graph
$(V(K),E(K))$ is complete multipartite.  We write its parts as
\[
V(K)=V_1\sqcup\cdots\sqcup V_r.
\]
(This is the generalization to M-convex sets of the decomposition of a
rank-$2$ matroid into parallel classes.)

Let
\[
S:=\{a\in [r]:D(K)\cap V_a\ne\varnothing\}.
\]
If $a\in S$, then $V_a$ is a singleton by \cite[Lemma 4.5(1)]{BHKL1}; we denote its unique element by
$i_a$.

\begin{lemma}
\label{lem:quadratic-simplification}
Let $c=(c_\beta)_{\beta\in K}$ be a weak $\T_q$-representation of $K$.
Then there exist positive real numbers $\omega_i$, normalized by the condition
\[
\sum_{i\in V_a}\omega_i=1
\qquad (1\le a\le r),
\]
and a weak $\T_q$-representation $\bar c$ of
\[
K' :=
\{e_a+e_b:a\ne b\text{ in }[r]\}\cup \{2e_a:a\in S\}
\subseteq \Delta_r^2
\]
such that
\[
c_{e_i+e_j}
=
\omega_i\omega_j\,\bar c_{e_a+e_b}
\qquad
(i\in V_a,\ j\in V_b,\ a\ne b\text{ in }[r]),
\]
and
\[
c_{2e_{i_a}}=\bar c_{2e_a}
\qquad (a\in S,\ V_a=\{i_a\}).
\]
\end{lemma}

\begin{remark}
    The representation $\bar c$ in Lemma \ref{lem:quadratic-simplification} should be thought of as the \emph{simplification} of $c$.
    Indeed, if $J=M$ is a matroid, the support of $\bar c$ is exactly the simplification of $M$ and Lemma \ref{lem:quadratic-simplification} says that the simplification operation preserves the property of being a $\T_q$-representation. The corresponding statement for Lorentzian polynomials is \cite[Theorem 4.7]{BHKL1}.
\end{remark}

\begin{proof}
Let
\[
f_c:=\sum_{\beta\in K} c_\beta \frac{x^\beta}{\beta!}.
\]
Since $c$ is a weak $\T_q$-representation, it is in particular a weak
$\T_\infty$-representation.  We may therefore apply \cite[Theorem 4.7]{BHKL1} to $f_c$.  It gives positive
weights $\omega_i$, normalized by
\[
\sum_{i\in V_a}\omega_i=1,
\]
and a unique polynomial
\[
g=\sum_{\kappa\in K'}\bar c_\kappa \frac{y^\kappa}{\kappa!}
\]
with support $K'$ such that
\[
f_c(x_1,\ldots,x_n)
=
g\left(\sum_{i\in V_1}\omega_i x_i,\ldots,
        \sum_{i\in V_r}\omega_i x_i\right).
\]
Comparing coefficients gives
\[
c_{e_i+e_j}
=
\omega_i\omega_j\,\bar c_{e_a+e_b}
\qquad
(i\in V_a,\ j\in V_b,\ a\ne b),
\]
and, for $a\in S$,
\[
c_{2e_{i_a}}=\bar c_{2e_a}.
\]

It remains only to check that $\bar c$ is a weak $\T_q$-representation, not merely
a weak $\T_\infty$-representation.  But every weak Pl\"ucker relation for $\bar c$
pulls back to the corresponding weak Pl\"ucker relation for $c$, with all three
terms multiplied by the same positive factor, namely the product of the relevant
weights $\omega_i$.  Since membership in $N_{\T_q}$ is invariant under multiplication
by a common positive scalar, the weak $\T_q$-relations for $c$ imply the weak
$\T_q$-relations for $\bar c$.
\end{proof}

The following result is a special case of \cite[Theorem~4.7(2)]{BHKL1}; we include a short proof for the reader's convenience.

\begin{lemma}
\label{lem:block-lorentzian-equivalence}
Let $B$ be the symmetric $r\times r$ matrix with entries
\[
B_{ab}=\bar c_{e_a+e_b}\quad(a\ne b),
\qquad
B_{aa}=\bar c_{2e_a}\quad(a\in S),
\qquad
B_{aa}=0\quad(a\notin S).
\]
Let $H$ be the symmetric matrix indexed by $V(K)$ with entries
\[
H_{ij}=\omega_i\omega_j B_{ab}
\qquad (i\in V_a,\ j\in V_b,\ a\ne b),
\]
\[
H_{i_ai_a}=B_{aa}\qquad(a\in S),
\]
and all other diagonal entries equal to zero.  Then $H$ is Lorentzian if and only
if $B$ is Lorentzian.
\end{lemma}

\begin{proof}
Define a linear map
\[
L:\mathbb R^{V(K)}\longrightarrow \mathbb R^r,
\qquad
(Lx)_a=\sum_{i\in V_a}\omega_i x_i.
\]
Since each $V_a$ is nonempty and the weights on $V_a$ are positive, $L$ is
surjective.  The quadratic form associated to $H$ is the pullback of the
quadratic form associated to $B$:
\[
x^\top Hx=(Lx)^\top B(Lx).
\]
Thus the positive index of $H$ equals the positive index of $B$; the kernel of
$L$ only contributes additional zero directions.  Since both matrices have
nonnegative entries, the Lorentzian condition is equivalent for $H$ and $B$.
\end{proof}

\subsection{Weighted copies}

The following construction is the key point in the proof of the lower
inclusion for $\upN_t$.
The construction replaces each index $a$ by a finite set of weighted copies and
produces a weak $\T_q$-representation on the resulting larger ground set.  

\begin{lemma}[Weighted-copy lift]
\label{lem:weighted-copy-lift}
With notation as in Lemma~\ref{lem:quadratic-simplification}, let
$\bar c$ be a weak $\T_q$-representation of
\[
K'=\{e_a+e_b:a\ne b\}\cup\{2e_a:a\in S\}\subseteq \Delta_r^2.
\]
For each $a\in S$, fix a number $\theta_a\in(0,1)$, and for
$a\notin S$ set $\theta_a=0$.

Then there exist a finite set $\widetilde E$, a surjective map
\[
\pi:\widetilde E\to [r],
\]
and a weak $\T_q$-representation $\widetilde c$ of the uniform matroid
$U_{2,\widetilde E}$ with the following property.

Let $A$ be the zero-diagonal symmetric matrix indexed by $\widetilde E$ whose
off-diagonal entries are
\[
A_{uv}=\widetilde c_{\{u,v\}}\qquad(u\ne v).
\]
Then there exist linearly independent vectors
\[
f_1,\ldots,f_r\in \mathbb R^{\widetilde E}
\]
with nonnegative entries and
\[
\supp(f_a)=\pi^{-1}(a)\qquad(a\in[r])
\]
such that the compressed matrix
\[
B=(B_{ab})_{1\le a,b\le r},
\qquad
B_{ab}:=f_a^\top A f_b,
\]
satisfies
\begin{equation}
\label{eq:weighted-copy-compression}
\begin{aligned}
B_{ab}&=\bar c_{e_a+e_b} &&\text{for } a\ne b,\\
B_{aa}&=\theta_a\,\bar c_{2e_a} &&\text{for } a\in S,\\
B_{aa}&=0 &&\text{for } a\notin S.
\end{aligned}
\end{equation}
Consequently, if $A$ is Lorentzian, then $B$ is Lorentzian.
\end{lemma}

\begin{proof}
For each $a\in [r]$, choose a finite set $E_a$ as follows.  If $a\in S$,
choose a finite probability vector
\[
\mu_a=(\mu_{a,u})_{u\in E_a}
\]
with all entries positive and
\[
1-\sum_{u\in E_a}\mu_{a,u}^2=\theta_a.
\]
This is possible because, for fixed $m$, the quantity
$1-\sum_{u=1}^m \mu_u^2$
ranges continuously over the interval $(0,1-1/m]$ as $\mu$ ranges over the
interior of the probability simplex.  Choosing $m$ sufficiently large makes
$1-1/m>\theta_a$, so the desired value $\theta_a$ is attained.

If $a\notin S$, take $E_a$ to consist of a single point $u$ and put
$\mu_{a,u}=1$.  Now set
\[
        \widetilde E:=\bigsqcup_{a=1}^r E_a,
\]
and let $\pi:\widetilde E\to [r]$
be the map sending every element of $E_a$ to $a$.

For notational brevity, write
\[
        \mu_u:=\mu_{\pi(u),u}.
\]
Define, for distinct $u,v\in \widetilde E$,
\[
\widetilde c_{uv}
=
\sqrt{\mu_u\mu_v}\,
\bar c_{e_{\pi(u)}+e_{\pi(v)}}.
\]
This is well-defined: if $\pi(u)=\pi(v)=a$ for distinct $u,v$, then
necessarily $a\in S$, and hence $2e_a\in K'$.

The weak $\T_q$-relations for $\widetilde c$ follow directly from those for
$\bar c$.  Indeed, in every quartet of distinct elements of $\widetilde E$, the
three Pl\"ucker products upstairs are the corresponding three Pl\"ucker
products downstairs multiplied by the same positive factor.  
(For instance, if the quartet consists of four distinct elements $u,v,w,z\in \widetilde E$, then the common
factor is $\sqrt{\mu_u\mu_v\mu_w\mu_z}$.)
Thus $\widetilde c$ is a weak $\T_q$-representation of the uniform matroid
$U_{2,\widetilde E}$.

Let $A$ be the zero-diagonal symmetric matrix indexed by $\widetilde E$ with
off-diagonal entries
\[
 A_{uv}=\widetilde c_{uv}\qquad(u\ne v).
\]

Set
\[
        f_a:=\sum_{u\in E_a}\sqrt{\mu_u}\,e_u
        \qquad(a\in [r]).
\]
Since all weights $\mu_u$ are positive, the vector $f_a$ has nonnegative
entries and support $E_a=\pi^{-1}(a)$.  In particular, the vectors
$f_1,\ldots,f_r$ have disjoint supports and are linearly independent.

If
$a\ne b$, then
\[
\begin{aligned}
f_a^\top A f_b
&=
\sum_{u\in E_a}
\sum_{v\in E_b}
\mu_u\mu_v\,\bar c_{e_a+e_b}  \\
&=
\left(\sum_{u\in E_a}\mu_u\right)
\left(\sum_{v\in E_b}\mu_v\right)
\bar c_{e_a+e_b} \\
&=
\bar c_{e_a+e_b}.
\end{aligned}
\]
If $a\in S$, then
\[
\begin{aligned}
f_a^\top A f_a
&=
\sum_{\substack{u,v\in E_a\\u\ne v}}
\mu_u\mu_v\,\bar c_{2e_a} \\
&=
\left(1-\sum_{u\in E_a}\mu_u^2\right)\bar c_{2e_a} \\
&=
\theta_a\bar c_{2e_a}.
\end{aligned}
\]
For $a\notin S$, the set $E_a$ has only one element, so the corresponding
diagonal entry is zero.

Finally, $B$ is the matrix of the restriction of the quadratic form associated
to $A$ to the span of the $f_a$, so the positive index of $B$ is at most the
positive index of $A$.  Thus Lorentzianity of $A$ implies Lorentzianity of $B$.
\end{proof}

\begin{remark}
The preceding weighted-copy construction can be viewed as a weighted refinement
of the usual natural matroid construction associated to a polymatroid (cf.~\cite{BCF23}).  In the
natural matroid construction, an element $a$ is replaced by a fiber of
$m_a:=|\pi^{-1}(a)|$ parallel copies, and one works with subsets obtained by choosing a prescribed
number of copies from each fiber.  The construction above reduces to this
situation when the weights on each fiber are all equal:
\[
        \mu_{a,u}=\frac{1}{m_a}
        \qquad
        (u\in \pi^{-1}(a)).
\]
In this uniform case, the corresponding diagonal factor is
\[
        1-\sum_{u\in \pi^{-1}(a)}\mu_{a,u}^2
        =
        1-\frac{1}{m_a}.
\]
The weighted-copy construction allows one to realize arbitrary prescribed factors
$\theta_a\in(0,1)$, rather than only the discrete values $1-1/m_a$.
\end{remark}

\subsection{The lower inclusion}

We now prove the lower containment in Theorem~\ref{thm:Nt-sandwich}.

\begin{proposition}[Lower $\upN_t$-containment]
\label{prop:lower-Nt}
Let $J\subseteq \Delta_n^d$ be M-convex, and let $t>0$.  Then there exists
$q_t(J)>0$ such that
\[
\upN_t\upR_J^{\rm w}(\T_q)\subseteq \upL_J
\qquad\text{for all }0<q\le q_t(J).
\]
\end{proposition}

\begin{proof}
Let $\rho\in \upR_J^{\rm w}(\T_q)$.  We must prove that $\upN_t f_\rho$ is
Lorentzian.  Its support is still $J$, hence it is M-convex.  By Theorem~\ref{thm:BH-limit-free}, 
it suffices to prove that every nonzero quadratic derivative of $\upN_t f_\rho$ is Lorentzian.

Fix $\gamma\in\Delta_n^{d-2}$ and put
\[
K_\gamma:=\{\beta\in\Delta_n^2:\gamma+\beta\in J\}.
\]
Let
\[
c_\beta:=\rho(\gamma+\beta)
\qquad(\beta\in K_\gamma).
\]
Then $c$ is a weak $\T_q$-representation of the quadratic M-convex set
$K_\gamma$.

The quadratic polynomial $h:=\partial^\gamma\upN_t f_\rho$
has the expansion
\[
h=
\sum_{\beta\in K_\gamma}
\frac{\rho(\gamma+\beta)}{((\gamma+\beta)!)^t}
\frac{x^\beta}{\beta!}.
\]
Let $H$ be its Hessian.  For $i\ne j$,
\[
H_{ij}
=
\frac{c_{e_i+e_j}}{((\gamma+e_i+e_j)!)^t}
=
\frac{c_{e_i+e_j}}
{(\gamma!)^t(\gamma_i+1)^t(\gamma_j+1)^t}
\]
and
\[
H_{ii}
=
\frac{c_{2e_i}}{((\gamma+2e_i)!)^t}
=
\frac{c_{2e_i}}
{(\gamma!)^t(\gamma_i+1)^t(\gamma_i+2)^t}.
\]
Multiplying $H$ by the positive scalar $(\gamma!)^t$ and applying the positive
diagonal congruence with diagonal entries $(\gamma_i+1)^t$, we obtain a matrix
$H'$ with entries
\[
H'_{ij}=c_{e_i+e_j}
\qquad(i\ne j)
\]
and
\[
H'_{ii}
=
\left(\frac{\gamma_i+1}{\gamma_i+2}\right)^t c_{2e_i}.
\]
It suffices to prove that $H'$ is Lorentzian.

Applying Lemma~\ref{lem:quadratic-simplification} to $K_\gamma$ and $c$, we obtain
a simplified representation $\bar c$ on $K'$.  By
Lemma~\ref{lem:block-lorentzian-equivalence}, it is enough to prove that the
$r\times r$ matrix $B$ with entries
\[
B_{ab}=\bar c_{e_a+e_b}\qquad(a\ne b),
\]
\[
B_{aa}
=
\left(\frac{\gamma_{i_a}+1}{\gamma_{i_a}+2}\right)^t
\bar c_{2e_a}
\qquad(a\in S),
\]
and $B_{aa}=0$ for $a\notin S$, is Lorentzian.

For $a\in S$, put
\[
\theta_a:=
\left(\frac{\gamma_{i_a}+1}{\gamma_{i_a}+2}\right)^t.
\]
Then $0<\theta_a<1$.  By Lemma~\ref{lem:weighted-copy-lift}, there is a finite weighted-copy lift $\tilde c$ on a uniform
rank-2 matroid $U_{2,\widetilde E_\gamma}$ whose compressed matrix is precisely
$B$.  

Set $M_\gamma:=|\widetilde E_\gamma|$.
If $M_\gamma\le 3$, every
admissible rank-$2$ coefficient matrix on $M_\gamma$ elements is Lorentzian.  If
$M_\gamma\ge 4$, Theorem~\ref{thm:rank-two-uniform} shows that the lifted
coefficient matrix is Lorentzian whenever
\[
0<q\le
\log_2\left(1+\frac1{M_\gamma-2}\right).
\]
Therefore, for this fixed $\gamma$, there exists $q_{t,\gamma}>0$ such that
$h=\partial^\gamma\upN_t f_\rho$ is Lorentzian for all
$0<q\le q_{t,\gamma}$.

There are only finitely many $\gamma\in \Delta_n^{d-2}$ for which
$K_\gamma\ne\emptyset$.  Hence
\[
q_t(J):=\min_{\substack{\gamma\in\Delta_n^{d-2}\\ K_\gamma\ne\emptyset}}
q_{t,\gamma}>0
\]
works for all nonzero quadratic derivatives simultaneously.  The remaining
quadratic derivatives vanish identically and impose no condition. Hence $\upN_t f_\rho$ is
Lorentzian.
\end{proof}

\begin{remark}
In the natural matroid construction, where an element $a$ is replaced by a fiber of $m_a$ parallel copies,
the factor appearing on the diagonal is
\[
1-\sum_{u\in \pi^{-1}(a)}\mu_{a,u}^2
=
1-\frac{1}{m_a}
=
\frac{m_a-1}{m_a}.
\]
For the ordinary normalization operator
$\upN=\upN_1$, the relevant diagonal factor is 
\[
\frac{\gamma_i+1}{\gamma_i+2},
\]
which is obtained by taking $m_a=\gamma_i+2$.

For the more general operator $\upN_t$, however, the required diagonal factor is
\[
\left(\frac{\gamma_i+1}{\gamma_i+2}\right)^t,
\]
which is typically not of the form $1-1/m$ for an integer $m$.  
\end{remark}

\subsection{The upper inclusion}

We now prove the upper containment in Theorem~\ref{thm:Nt-sandwich}.

\begin{lemma}
\label{lem:quadratic-ineqs}
Let $Q$ be a homogeneous quadratic polynomial with nonnegative
coefficients, and let $H$ be its Hessian.
Assume that $Q$ is Lorentzian.  Then:
\begin{enumerate}
\item If $i,j,k,\ell$ are distinct, the three products
\[
H_{ij}H_{k\ell},\qquad
H_{ik}H_{j\ell},\qquad
H_{i\ell}H_{jk}
\]
satisfy the weak $\T_2$-relation.

\item If $i,k,\ell$ are distinct, then
\[
H_{ii}H_{k\ell}\le 2H_{ik}H_{i\ell}.
\]

\item If $i\ne k$, then
\[
H_{ii}H_{kk}\le H_{ik}^2.
\]
\end{enumerate}
\end{lemma}

\begin{proof}
For (1), take the multi-affine part of the restriction of $Q$ to the variables
$i,j,k,\ell$.  The multi-affine part of a Lorentzian polynomial is Lorentzian,
and the rank-$2$ four-variable case is exactly Proposition~\ref{prop:U24}; hence
the three products satisfy the weak $\T_2$-relation.

For (3), apply the principal-minor criterion to the $2\times 2$ principal
submatrix on $\{i,k\}$.

For (2), consider the $3\times 3$ principal submatrix on $\{i,k,\ell\}$ and write
\[
a=H_{ii},\quad p=H_{kk},\quad q=H_{\ell\ell},
\quad
x=H_{ik},\quad y=H_{i\ell},\quad z=H_{k\ell}.
\]
Then the following matrix is Lorentzian:
\begin{equation*}
    \begin{pmatrix}
        a & x & y \\
        x & p & z \\
        y & z & q
    \end{pmatrix}
\end{equation*}
The $2\times 2$ principal minors give
\[
x^2\ge ap,
\qquad
y^2\ge aq.
\]
The $3\times 3$ principal-minor condition gives
\[
apq+2xyz-az^2-py^2-qx^2\ge 0.
\]
Using $py^2\ge apq$ and $qx^2\ge apq$, we obtain
\[
2xyz\ge az^2+apq.
\]
If $z=0$, the desired inequality is trivial.  If $z>0$, then
\[
2xy\ge az+\frac{apq}{z}\ge az,
\]
which proves $az\le 2xy$.
\end{proof}

\begin{proposition}[Upper $\upN_t$-containment]
\label{prop:upper-Nt}
Let $J\subseteq\Delta_n^d$ be M-convex and let $t>0$.  Then
\[
\upL_J\subseteq \upN_t\upR_J^{\rm w}(\T_{p(t)}),
\qquad
p(t)=\max\{2,1+t,2t\}.
\]
In particular, for $t=1$ one has
\[
\upL_J\subseteq \upN\upR_J^{\rm w}(\T_2).
\]
\end{proposition}

\begin{proof}
Let $f\in\upL_J$.  We must prove that $\rho:=\upN_t^{-1}f$ 
is a weak $\T_{p(t)}$-representation.  Fix $\alpha\in\Delta_n^{d-2}$ 
and consider the quadratic derivative $Q:=\partial^\alpha f$.
Let $H$ be its Hessian.  Since $f$ is Lorentzian, $Q$ is Lorentzian.
The corresponding quadratic derivative of $\upN_t^{-1}f$ differs from $Q$ by a
positive scalar, by a positive diagonal rescaling, and by multiplying each
diagonal coefficient $H_{ii}$ by
\[
\lambda_i
:=
\left(\frac{\alpha_i+2}{\alpha_i+1}\right)^t.
\]
In particular,
\[
1\le \lambda_i\le 2^t.
\]
Positive scalars and positive diagonal rescalings multiply all three terms in a
weak Pl\"ucker relation by a common positive factor, so it remains only to check
the effect of the diagonal factors $\lambda_i$.
There are three nontrivial cases.

First, suppose the four indices are distinct.  Then no diagonal coefficient
occurs.  By Lemma~\ref{lem:quadratic-ineqs}(1), the relevant three products
already satisfy the weak $\T_2$-relation.  Since $p(t)\ge 2$, they also satisfy
the weak $\T_{p(t)}$ relation.

Next, suppose exactly one index is repeated; say the repeated index is $i$ and
the other two indices are $k,\ell$.  The three products have the form
\[
H_{ik}H_{i\ell},
\qquad
H_{ik}H_{i\ell},
\qquad
\lambda_i H_{ii}H_{k\ell}.
\]
By Lemma~\ref{lem:quadratic-ineqs}(2),
\[
H_{ii}H_{k\ell}\le 2H_{ik}H_{i\ell}.
\]
Hence
\[
\lambda_i H_{ii}H_{k\ell}
\le 2^{1+t}H_{ik}H_{i\ell}.
\]
Thus the weak $\T_{p(t)}$-relation holds provided
$p(t)\ge 1+t$.

Finally, suppose the indices form two repeated pairs, say $i,i,k,k$.  The three
products are
\[
H_{ik}^2,
\qquad
H_{ik}^2,
\qquad
\lambda_i\lambda_k H_{ii}H_{kk}.
\]
By Lemma~\ref{lem:quadratic-ineqs}(3),
\[
H_{ii}H_{kk}\le H_{ik}^2,
\]
and therefore
\[
\lambda_i\lambda_kH_{ii}H_{kk}
\le 2^{2t}H_{ik}^2.
\]
Thus the weak $\T_{p(t)}$-relation holds provided
$p(t)\ge 2t$.

The remaining cases, where at least three of the four indices are equal, are
automatic: in each such case all three Pl\"ucker products coincide.  Therefore all
weak Pl\"ucker relations for $\rho$ hold over $\T_{p(t)}$, as claimed.
\end{proof}

\begin{proof}[Proof of Theorem~\ref{thm:Nt-sandwich}]
The lower containment is Proposition~\ref{prop:lower-Nt}.  The upper containment
is Proposition~\ref{prop:upper-Nt}.  Since $p(1)=\max\{2,2\}=2$,
the displayed $t=1$ sandwich follows immediately.
\end{proof}

\begin{remark}
\label{rem:Nt-preserves-Lorentzian}
The normalization operators $\upN_t$ preserve both Lorentzian polynomials and weak triangular representations
for every finite $t\ge 0$.  

The Lorentzianity statement follows from the symbol criterion \cite[Theorem~3.2]{Branden-Huh20}.  Indeed, the $\kappa$-symbol of $\upN_t$
factors as
\[
\operatorname{Sym}_{\kappa}(\upN_t)(w,u)
=
\prod_{j=1}^n
\left(
\sum_{a=0}^{\kappa_j}
\binom{\kappa_j}{a}
\frac{w_j^a u_j^{\kappa_j-a}}{(a!)^t}
\right).
\]
By the bivariate Lorentzian criterion, each factor is Lorentzian because the
sequence
\[
\left(\frac{1}{(a!)^t}\right)_{a=0}^{\kappa_j}
\]
is log-concave; equivalently, the factorial sequence is log-convex.  Products of
Lorentzian polynomials are Lorentzian, and hence the symbol criterion applies.

Similarly,
\[
        \upN_t\upR_J^{\rm w}(\T_q)\subseteq \upR_J^{\rm w}(\T_q)
        \qquad(q\in[0,\infty],\ t\ge 0).
\]
For $t=1$ this is Lemma~\ref{lem:N-preserves-weak-reps}.  The same proof applies
for arbitrary $t\ge 0$: the denominators in the three Pl\"ucker products are
replaced by their $t$-th powers, and the only needed comparison is still the
log-convexity inequality
\[
        (m+1)!^2\le (m+2)!\,m!,
\]
this time with both sides raised to the power $t$.
\end{remark}

\section{Upper bounds for \texorpdfstring{$q(n)$}{q(n)} and the nonexistence of a universal constant}
\label{sec:upper-bounds-qn}

In Section~\ref{sec:U2n-core}, we proved the lower bound
\[
q(n)\ge \varepsilon(n):=\log_2\!\left(1+\frac{1}{n-2}\right),
\]
showing in particular that $q(n)>0$ for every $n\ge 4$.
The purpose of the present section is to prove a complementary upper bound of the same
order of magnitude. In particular, we will see that $q(n)\to 0$ as $n\to\infty$, so there
is no universal constant $q>0$ such that every weak $\T_q$-representation of every
$U_{2,n}$ is Lorentzian.

The idea is to construct, for each $n$, an explicit admissible matrix which
satisfies the $\T_q$-inequalities but has two positive eigenvalues. 

\begin{proposition}
\label{prop:qn-upper}
For every $n\ge 4$,
\[
q(n)\le
\begin{cases}
2\log_2\!\left(\dfrac{n}{n-2}\right), & \text{if } n \text{ is even},\\[1.25ex]
\log_2\!\left(\dfrac{n+1}{n-3}\right), & \text{if } n \text{ is odd}.
\end{cases}
\]
In particular,
\[
q(n)=O(1/n)
\qquad\text{as } n\to\infty.
\]
\end{proposition}

\begin{proof}
Fix integers $a,b\ge 2$ with $a+b=n$ and fix a real number $\alpha>1$.

We partition $[n]$ into two blocks
\[
A:=\{1,\dots,a\},
\qquad
B:=\{a+1,\dots,n\},
\]
and consider the admissible matrix
\[
L=
\begin{pmatrix}
\alpha(J_a-I_a) & \mathbf 1_{a\times b}\\
\mathbf 1_{b\times a} & \alpha(J_b-I_b)
\end{pmatrix}.
\]
Thus entries within each block are equal to $\alpha$, while entries across the
two blocks are equal to $1$.  We will determine exactly when $L$ satisfies the
$\T_q$-inequalities and exactly when $L$ is Lorentzian.

\smallskip

\emph{Step 1: the $\T_q$-inequalities.}
Consider a quartet of indices. 
If a quartet contains two indices from $A$ and two from $B$, then, up to reordering,
the three products appearing in the $\T_q$-condition are $\alpha^2$, $1$, and $1$.
Since $\alpha>1$, the only nontrivial triangle inequality is
$(\alpha^2)^{1/q}\le 2$,
or equivalently $\alpha\le 2^{q/2}$.
If the quartet has block type $(3,1)$ or $(4,0)$, then all three products are
equal, so the $\T_q$-condition is automatic.

Therefore $L$ satisfies the $\T_q$-inequalities if and only if
\begin{equation}
\label{eq:qn-upper-Tq-condition}
\alpha\le 2^{q/2}.
\end{equation}

\smallskip

\emph{Step 2: spectral decomposition.}
Consider the decomposition
\[
\R^n = \Span\{1_A,1_B\} \oplus W,
\]
where
\[
W:=
\left\{
(x_A,x_B)\in \R^a\oplus\R^b :
\sum_{i\in A}x_i=0,\ \sum_{j\in B}x_j=0
\right\}.
\]
On the subspace $W$, the all-ones matrices $J_a$ and $J_b$ vanish, so the
diagonal blocks act as $-\alpha I$, while the off-diagonal all-ones blocks also
vanish because they depend only on the blockwise sums.  Thus $L$ acts as
multiplication by $-\alpha$ on $W$.  In particular, $L$ has eigenvalue
$-\alpha$ with multiplicity
\[
(a-1)+(b-1)=n-2.
\]

It remains to analyze the action of $L$ on $\Span\{1_A,1_B\}$.  With respect to
the basis $\{1_A,1_B\}$, this action is given by
\[
M=
\begin{pmatrix}
\alpha(a-1) & b\\
a & \alpha(b-1)
\end{pmatrix}.
\]
Since
\[
\tr(M)=\alpha(a+b-2)=\alpha(n-2)>0,
\]
the matrix $M$ has two positive eigenvalues if and only if $\det(M)>0$.
A direct calculation gives
\[
\det(M)=\alpha^2(a-1)(b-1)-ab.
\]
Thus $L$ fails to be Lorentzian precisely when
\[
\alpha^2(a-1)(b-1)>ab,
\]
or equivalently,
\begin{equation}
\label{eq:qn-upper-nonlor}
\alpha>r(a,b):=
\sqrt{\frac{ab}{(a-1)(b-1)}}.
\end{equation}

\smallskip

\emph{Step 3: optimize the obstruction.}
Combining \eqref{eq:qn-upper-Tq-condition} and
\eqref{eq:qn-upper-nonlor}, we see that whenever
\[
r(a,b)<\alpha\le 2^{q/2},
\]
the matrix $L$ satisfies the $\T_q$-inequalities but is not Lorentzian.  In
particular, if $q>2\log_2 r(a,b)$
then we may choose $\alpha$ with $r(a,b)<\alpha\le 2^{q/2}$
and obtain a weak $\T_q$-representation of $U_{2,n}$ which is not Lorentzian.
Therefore $q(n)\le 2\log_2 r(a,b)$ for every decomposition $a+b=n$ with $a,b\ge 2$.
To obtain the best possible upper bound, we minimize $r(a,b)$ over all such
decompositions.

Since
\[
r(a,b)^2=\frac{ab}{(a-1)(b-1)}
=\left(1-\frac{n-1}{ab}\right)^{-1},
\]
the quantity $r(a,b)$ is minimized when $ab$ is maximized, i.e. when the
partition is as balanced as possible.

If $n=2m$ is even, the minimum occurs at $a=b=m$, giving
\[
r_{\min}(n)=\frac{m}{m-1}=\frac{n}{n-2}.
\]
Therefore
\[
q(n)\le 2\log_2\!\left(\frac{n}{n-2}\right).
\]

If $n=2m+1$ is odd, the minimum occurs at $(a,b)=(m,m+1)$, giving
\[
r_{\min}(n)=\sqrt{\frac{m(m+1)}{(m-1)m}}
=\sqrt{\frac{m+1}{m-1}}
=\sqrt{\frac{n+1}{n-3}}.
\]
Therefore
\[
q(n)\le \log_2\!\left(\frac{n+1}{n-3}\right).
\]

This proves the stated upper bound.

\smallskip

\emph{Step 4: asymptotics.}
As $n\to\infty$,
\[
2\log_2\!\left(\frac{n}{n-2}\right)
=
2\log_2\!\left(1+\frac{2}{n-2}\right)
=
\frac{4}{(\ln 2)\,n}+O(n^{-2}),
\]
and similarly
\[
\log_2\!\left(\frac{n+1}{n-3}\right)
=
\log_2\!\left(1+\frac{4}{n-3}\right)
=
\frac{4}{(\ln 2)\,n}+O(n^{-2}).
\]
Hence $q(n)=O(1/n)$.
\end{proof}

The preceding proposition has two immediate consequences.

\begin{corollary}
\label{cor:no-universal-q}
There is no universal constant $q>0$ such that every weak $\T_q$-representation of
every $U_{2,n}$ is Lorentzian.
\end{corollary}

\begin{proof}
By Proposition~\ref{prop:qn-upper}, we have $q(n)\to 0$ as $n\to\infty$.
\end{proof}

\begin{remark}
\label{rem:compare-upper-lower}
Proposition~\ref{prop:qn-upper} should be compared with the lower bound from
Theorem~\ref{thm:rank-two-uniform}:
\[
\log_2\!\left(\frac{n-1}{n-2}\right)
\le q(n)\le
\begin{cases}
2\log_2\!\left(\dfrac{n}{n-2}\right), & n \text{ even},\\[1.25ex]
\log_2\!\left(\dfrac{n+1}{n-3}\right), & n \text{ odd}.
\end{cases}
\]

In particular, the upper and lower bounds match up to absolute constant factors:
\[
q(n)=\Theta(1/n).
\]
\end{remark}

For $n=4$, the upper bound in Proposition~\ref{prop:qn-upper} gives
$q(4)\le 2$, and Proposition~\ref{prop:U24} shows that this is sharp:
\[
q(4)=2.
\]

We believe that the upper bound is the right answer for all $n\ge 4$.

\begin{conjecture}
\label{conj:qn-sharp}
For every $n\ge 4$,
\[
q(n)=
\begin{cases}
2\log_2\!\left(\dfrac{n}{n-2}\right), & \text{if } n \text{ is even},\\[1.25ex]
\log_2\!\left(\dfrac{n+1}{n-3}\right), & \text{if } n \text{ is odd}.
\end{cases}
\]
\end{conjecture}

\begin{remark}
\label{rem:section7-motivation}
For $n=5$, Proposition~\ref{prop:qn-upper} yields $q(5)\le \log_2 3$,
while Theorem~\ref{thm:rank-two-uniform} gives $q(5)\ge \log_2(4/3)$.
In the next section, we prove that the upper bound is sharp when $n=5$:
\[
q(5)=\log_2 3,
\]
providing a non-trivial piece of evidence for Conjecture~\ref{conj:qn-sharp}.
\end{remark}

\section{A sharp lower bound for $q(5)$}
\label{sec:q5}

In Section~\ref{sec:U2n-core}, we implicitly used the following metric principle:
if every $(n-1)$-point metric space $(X,d)$ has the property that $(X,d^{q/2})$ can be embedded in some Euclidean space,
then $q(n)\ge q$.  This is enough to recover the Blumenthal lower bound
$q(5)\ge1$, but it does not utilize all of the $\T_q$-quartet inequalities.
After diagonal normalization, the quartets through the distinguished index force a
metric on the remaining $n-1$ points; the remaining quartets force this metric to be
Ptolemaic. 

In this section, we exploit this extra Ptolemaic structure in the first open case $n=5$.
Our main result is that if $(X,d)$ is a four-point Ptolemaic metric, then
$(X,d^{q/2})$ embeds isometrically in Euclidean space for all $0<q\le\log_2 3$ (Theorem~\ref{thm:q5-four-point-ptolemaic}).
From this, we will deduce that $q(5)=\log_2(3)$ (Corollary~\ref{cor:q5-sharp}).\footnote{With the help of David Renshaw and Claude Code, we have formalized the proofs of Theorem~\ref{thm:q5-four-point-ptolemaic} and Corollary~\ref{cor:q5-sharp} in Lean. The Lean code can be found at \url{https://github.com/icarm/FourPointPtolemaic.git}}

\subsection{Ptolemaic metrics and Schoenberg theory}

\begin{definition}
A finite metric space $(X,d)$ is \emph{Ptolemaic}\footnote{The name refers to the classical Ptolemy inequality in Euclidean geometry. Ptolemaic metric spaces have a substantial literature; for example, Schoenberg proved that a real normed vector space is Ptolemaic if and only if it is an inner product space~\cite{Schoenberg1952Ptolemy}.} if, for every four points
$x,y,z,w\in X$,
\[
        d(x,y)d(z,w)
        \le
        d(x,z)d(y,w)+d(x,w)d(y,z).
\]
\end{definition}

\begin{definition}
A finite metric space $(X,d)$ has $q$-\emph{negative type} if, for every
choice of real numbers $(c_x)_{x\in X}$ with $\sum_{x\in X}c_x=0$, one has
\[
        \sum_{x,y\in X} c_xc_y d(x,y)^q \le 0.
\]

We will use the following classical result, cf.~\cite{Schoenberg1938}.
For the reader's convenience, we give a self-contained proof of the precise version of the theorem that we will use.

\begin{lemma}[Schoenberg's criterion]
\label{lem:q5-schoenberg}
Let $(X,d)$ be a finite metric space and let $q>0$. For a base point
$p\in X$, define the \emph{Schoenberg matrix based at $p$} by
\[
        G^{(p)}_{xy}
        :=
        \frac{d(x,p)^q+d(y,p)^q-d(x,y)^q}{2},
        \qquad x,y\in X\setminus\{p\}.
\]
Then the following are equivalent:
\begin{enumerate}
\item $(X,d)$ has $q$-negative type, i.e.
\[
        \sum_{x,y\in X} c_xc_y d(x,y)^q \le 0
\]
for every choice of real numbers $(c_x)_{x\in X}$ satisfying
$\sum_{x\in X} c_x=0$.
\item For some base point $p\in X$, the matrix $G^{(p)}$ is positive
semidefinite.
\item For every base point $p\in X$, the matrix $G^{(p)}$ is positive
semidefinite.
\item The snowflaked\footnote{In metric geometry, replacing a metric $d$ by $d^\alpha$ for $0<\alpha<1$ is called taking a \emph{snowflake} of the metric space; the terminology alludes to the classical Koch snowflake curve.}
metric space $(X,d^{q/2})$ embeds isometrically in Euclidean space.
\end{enumerate}
Moreover, the following stronger conditions are equivalent:
\begin{primenumerate}
\item The matrix $E=(d(x,y)^q)_{x,y\in X}$ is conditionally strictly
negative definite, i.e.
\[
        \sum_{x,y\in X} c_xc_y d(x,y)^q < 0
\]
for every nonzero vector $(c_x)_{x\in X}$ satisfying $\sum_{x\in X}c_x=0$.
\item For some base point $p\in X$, the matrix $G^{(p)}$ is positive
definite.
\item For every base point $p\in X$, the matrix $G^{(p)}$ is positive
definite.
\item The snowflaked metric space $(X,d^{q/2})$ admits an isometric
embedding into Euclidean space whose image is affinely independent.
\end{primenumerate}
\end{lemma}

\begin{proof}
We first prove the semidefinite equivalences. Fix a base point $p\in X$ and
put $Y=X\setminus\{p\}$. For a vector $a=(a_x)_{x\in Y}$, define coefficients
$(c_z)_{z\in X}$ by
\[
        c_x=a_x \quad (x\in Y),
        \qquad
        c_p=-\sum_{x\in Y} a_x .
\]
Then $\sum_{z\in X}c_z=0$. Conversely, every vector $(c_z)_{z\in X}$ with
$\sum_z c_z=0$ arises uniquely in this way, by taking $a_x=c_x$ for
$x\in Y$.

A direct expansion gives
\[
\begin{aligned}
        \sum_{u,v\in X} c_uc_v d(u,v)^q
        &=
        \sum_{x,y\in Y} a_xa_y d(x,y)^q
        +2c_p\sum_{x\in Y} a_x d(x,p)^q  \\
        &=
        \sum_{x,y\in Y} a_xa_y d(x,y)^q
        -2\sum_{x,y\in Y} a_xa_y d(x,p)^q  \\
        &=
        \sum_{x,y\in Y} a_xa_y
        \bigl(d(x,y)^q-d(x,p)^q-d(y,p)^q\bigr)  \\
        &=
        -2 a^T G^{(p)} a .
\end{aligned}
\]
In the second line we used $c_p=-\sum_{y\in Y}a_y$, and in the third line
we symmetrized the term involving $d(x,p)^q$. Therefore, for any fixed base point $p$, the $q$-negative
type inequality is equivalent to
\[
        a^T G^{(p)}a\ge 0
        \qquad\text{for all } a\in\mathbb{R}^{Y}.
\]
In particular, $(1), (2)$, and $(3)$ are all equivalent.

Next suppose $G^{(p)}$ is positive semidefinite for some base point $p$.
Realize it as a Gram matrix:
\[
        G^{(p)}_{xy}=\langle v_x,v_y\rangle
        \qquad (x,y\in Y)
\]
for vectors $(v_x)_{x\in Y}$ in a Euclidean space, and set $v_p=0$. Then,
for $x,y\in Y$,
\[
\begin{aligned}
        \|v_x-v_y\|^2
        &=
        G^{(p)}_{xx}+G^{(p)}_{yy}-2G^{(p)}_{xy}  \\
        &=
        d(x,p)^q+d(y,p)^q
        -\bigl(d(x,p)^q+d(y,p)^q-d(x,y)^q\bigr)  \\
        &=
        d(x,y)^q.
\end{aligned}
\]
Also,
\[
        \|v_x-v_p\|^2=\|v_x\|^2=G^{(p)}_{xx}=d(x,p)^q .
\]
Hence $\|v_x-v_y\|=d(x,y)^{q/2}$ for all $x,y\in X$, so
$(X,d^{q/2})$ embeds isometrically in Euclidean space. This proves
$(2)\Longrightarrow (4)$.

Conversely, suppose $(X,d^{q/2})$ embeds isometrically in Euclidean space.
Choose an isometric embedding $x\mapsto v_x$, and translate it so that
$v_p=0$. Then, for $x,y\in Y$,
\[
        \langle v_x,v_y\rangle
        =
        \frac{\|v_x\|^2+\|v_y\|^2-\|v_x-v_y\|^2}{2}
        =
        \frac{d(x,p)^q+d(y,p)^q-d(x,y)^q}{2}
        =
        G^{(p)}_{xy}.
\]
Thus $G^{(p)}$ is a Gram matrix, and is therefore positive semidefinite.
This proves $(4)\Longrightarrow (2)$.

It remains to prove the strict assertions. The identity
\[
        \sum_{u,v\in X} c_uc_v d(u,v)^q
        =
        -2a^T G^{(p)}a
\]
above identifies nonzero vectors $c\in\mathbb{R}^X$ with $\sum_x c_x=0$
with nonzero vectors $a\in\mathbb{R}^{Y}$. Therefore $E=(d(x,y)^q)$ is
conditionally strictly negative definite if and only if $G^{(p)}$ is positive
definite. This proves $(1^\prime)\Longleftrightarrow (2^\prime)$ for any chosen base point
$p$, and hence also the equivalence with $(3^\prime)$.

Finally, if $G^{(p)}$ is positive definite, then in the Gram realization above
the vectors $(v_x)_{x\in Y}$ are linearly independent. Since $v_p=0$, this is
equivalent to the affine independence of the full set
$\{v_x:x\in X\}$. Hence $(2^\prime)\Longrightarrow (4^\prime)$.

Conversely, suppose $(X,d^{q/2})$ admits an isometric Euclidean embedding
whose image is affinely independent. Translating the embedding so that
$v_p=0$, the vectors $(v_x)_{x\in Y}$ are linearly independent. Since
$G^{(p)}$ is their Gram matrix, it is positive definite. Thus
$(4^\prime)\Longrightarrow (2^\prime)$, completing the proof of the strict equivalences.
\end{proof}

\begin{remark}[Four-point interpretation of the Schoenberg determinant]
\label{rem:four-point-schoenberg-determinant}
In the special case of four points, Lemma~\ref{lem:q5-schoenberg}
has the following elementary geometric interpretation. Let $d$ be a
four-point metric on $\{1,2,3,4\}$, and let $q>0$ be such that
$\widetilde d:=d^{q/2}$ is again a metric. 
Form two comparison triangles $134$ and $234$ in Euclidean $3$-space, using the side
lengths prescribed by $\widetilde d$. Glue these two triangles along their
common side $34$. If one rotates the second triangle about the line $34$,
the possible values of the remaining distance $|1-2|$ fill an interval
\[
        [\ell_-,\ell_+].
\]
Then the following conditions are equivalent:
\begin{enumerate}
\item The four-point metric space
\[
        \bigl(\{1,2,3,4\},\widetilde d\bigr)
\]
embeds isometrically in Euclidean space.
\item The prescribed fourth distance satisfies
\[
        \ell_- \le \widetilde d(1,2) \le \ell_+ .
\]
\item The $3\times 3$ Schoenberg matrix based at $4$,
\[
        \widetilde G^{(4)}_{ij}
        =
        \frac{\widetilde d(i,4)^2+\widetilde d(j,4)^2-\widetilde d(i,j)^2}{2},
        \qquad i,j\in\{1,2,3\},
\]
has nonnegative determinant.
\end{enumerate}

Indeed, the $2\times2$ principal minors of $\widetilde G^{(4)}$ encode
the existence of the two Euclidean comparison triangles $134$ and $234$.
Once these triangles exist, the only remaining obstruction to realizing the
four prescribed distances simultaneously in Euclidean space is whether the
prescribed value of $\widetilde d(1,2)$ lies in the interval obtained by
rotating one comparison triangle around the common side $34$. Equivalently,
in this four-point case the nonnegativity of the determinant of
$\widetilde G^{(4)}$ is precisely the final condition for positive
semidefiniteness of the Schoenberg matrix.
\end{remark}

For $m\ge3$, let
\[
        P(m):=\sup\{q>0:\text{ every $m$-point Ptolemaic metric has
        $q$-negative type}\}.
\]
\end{definition}

The following proposition makes precise the connection between $q(n)$ (the $q$-invariant of the matroid $U_{2,n}$) and the notion of Ptolemaic
negative type.

\begin{proposition}
\label{prop:qn-Pn}
For every $n\ge4$,
\[
        q(n)=P(n-1).
\]
\end{proposition}

\begin{proof}
Let $A$ be an admissible $n\times n$ matrix satisfying the
$\T_q$-inequalities.  By Lemma~\ref{lem:diag-congruence-U2n}, after positive
diagonal congruence we may assume
\[
        A_{1i}=1\qquad (i=2,\dots,n).
\]
Write
\[
        A=
        \begin{pmatrix}
        0&\1^\top\\
        \1&B
        \end{pmatrix},
        \qquad
        X:=\{2,\dots,n\}.
\]
Define
\[
        d(i,j):=B_{ij}^{1/q}\qquad (i\ne j,\ i,j\in X),
        \qquad
        d(i,i):=0.
\]
The quartets $\{1,i,j,k\}$ say precisely that $d(i,j),d(i,k),d(j,k)$ are the
side lengths of a Euclidean triangle.  Hence $d$ is a metric on $X$.

The quartets contained in $X$ say that, for every four distinct
$i,j,k,\ell\in X$, the three numbers
\[
        d(i,j)d(k,\ell),\qquad
        d(i,k)d(j,\ell),\qquad
        d(i,\ell)d(j,k)
\]
are the side lengths of a Euclidean triangle.  Equivalently, all Ptolemy inequalities
on $X$ hold.  Thus $d$ is Ptolemaic.

By Lemma~\ref{lem:squared-edm-block}, $A$ is Lorentzian if and only if $d$ has $q$-negative type.  It follows that every
$q<P(n-1)$ satisfies the defining property of $q(n)$, so $q(n)\ge P(n-1)$.

Conversely, let $d$ be any $(n-1)$-point Ptolemaic metric and form
\[
        A=
        \begin{pmatrix}
        0&\1^\top\\
        \1&(d(i,j)^q)
        \end{pmatrix}.
\]
The same discussion shows that $A$ satisfies the $\T_q$-inequalities.  If
$q<q(n)$, then $A$ is Lorentzian, hence $d$ has $q$-negative type by the
equivalence above. Thus $P(n-1)\ge q(n)$.
\end{proof}

We will also use the following well-known fact, included for completeness.

\begin{lemma}[Line metrics]
\label{lem:q5-line-metrics}
Every finite subset of $\R$, with its usual metric, has $q$-negative type for
$0<q\le2$.
\end{lemma}

\begin{proof}
The case $q=2$ follows from the elementary identity
\[
\sum_{i,j} c_i c_j |x_i-x_j|^2
=
-2\left(\sum_i c_i x_i\right)^2
\le0
\]
whenever $\sum_i c_i=0$.

For $0<q<2$, one can use the Schoenberg integral representation
\[
        |t|^q
        =
        c_q\int_0^\infty \bigl(1-\cos(st)\bigr)\frac{ds}{s^{1+q}},
\]
where $c_q>0$; see \cite{Schoenberg38}.  If $\sum_i c_i0$, then
\[
\begin{aligned}
        \sum_{i,j}c_ic_j|x_i-x_j|^q
        &=
        c_q\int_0^\infty
        \sum_{i,j}c_ic_j\bigl(1-\cos(s(x_i-x_j))\bigr)\frac{ds}{s^{1+q}}  \\
        &=
        -c_q\int_0^\infty
        \left|\sum_i c_i e^{isx_i}\right|^2
        \frac{ds}{s^{1+q}}
        \le0.
\end{aligned}
\]
\end{proof}

\begin{remark}
Alternatively, one can deduce the case $q<2$ from the $q=2$ case using the star-shapedness in log-coordinates of the space of Lorentzian polynomials, cf.~\cite[Proposition~3.25]{Branden-Huh20} and \cite[Theorem~4.4]{BHKL1}.
\end{remark}

\subsection{Metric inversion}

Let $(X,d)$ be a finite metric space and fix $p\in X$.  
Define the inverted distance $\widehat d$ by
\[
        \widehat d(i,p)=\frac{1}{d(i,p)}\qquad (i\ne p),
        \qquad
        \widehat d(i,j)=\frac{d(i,j)}{d(i,p)d(j,p)}\qquad (i,j\ne p).
\]

\begin{lemma}
\label{lem:q5-inversion-euclidean}
With notation as above, $(X,d^{q/2})$ embeds isometrically in Euclidean space if
and only if $(X,\widehat d^{q/2})$ does.
\end{lemma}

\begin{proof}
Translate a Euclidean realization of $d^{q/2}$ so that $p$ is at the origin.
Euclidean inversion
\[
        I(x)=\frac{x}{\|x\|^2}
\]
satisfies
\[
        \|I(x_i)-I(x_j)\|=\frac{\|x_i-x_j\|}{\|x_i\|\,\|x_j\|}.
\]
This realizes the inverted snowflake.  Since Euclidean inversion is its own inverse,
the converse follows in the same way.
\end{proof}

Ptolemy's inequality is closely tied to metric inversions (indeed, there is a beautiful proof of the classical Ptolemy theorem using inversion in a circle, see e.g. \cite{BogomolnyPtolemyInversion}).
The specific connection we need is the following.

\begin{lemma}
\label{lem:q5-ptolemy-inversion}
If $(X,d)$ is Ptolemaic, then the inverted distance $\widehat d$ is a metric.
\end{lemma}

\begin{proof}
The triangle inequalities involving $p$ are equivalent to the original triangle
inequalities for $d$.  For $i,j,k\ne p$, the inequality
\[
        \widehat d(i,j)\le \widehat d(i,k)+\widehat d(k,j)
\]
becomes, after multiplying by $d(i,p)d(j,p)d(k,p)$,
\[
        d(i,j)d(k,p)\le d(i,k)d(j,p)+d(k,j)d(i,p),
\]
which is Ptolemy's inequality for $i,j,k,p$.
\end{proof}

\subsection{The star inequality}

The following lemma provides the mechanism through which the constant $\log_2 3$ will enter our proof of the sharp lower bound $q(5) \geq \log_2(3)$.

\begin{lemma}[Star inequality]
\label{lem:q5-star}
Let $1\le q\le \log_2 3$.  Then for all $a,b\ge0$,
\[
        (a+b)^q\le a^q+b^q+(ab)^{q/2}.
\]
\end{lemma}

\begin{proof}
The cases $a=0$ or $b=0$ are immediate, so assume $a,b>0$.
Put
\[
        A:=a^{q/2},\qquad B:=b^{q/2},\qquad r:=\frac{2}{q}.
\]
Then $a=A^r$, $b=B^r$, and the desired inequality becomes
\[
        (A^r+B^r)^{2/r}\le A^2+AB+B^2.
\]
Equivalently, after raising both sides to the power $1/q=r/2$, it is enough to
prove
\[
        A^r+B^r\le (A^2+AB+B^2)^{r/2}.
\]
By homogeneity, set $B=1$.  Writing $c:=A$ and $m:=r/2=1/q$, the inequality is
\[
        \varphi(c):=(c^2+c+1)^m-c^{2m}-1\ge0.
\]
The hypothesis $q\le\log_2 3$ is equivalent to
\[
        m\ge m_0:=\log_3 2.
\]
We have the symmetry
\[
        \varphi(c)=c^{2m}\varphi(1/c),
\]
so it suffices to prove the inequality for $c\ge1$.  At the symmetric point,
\[
        \varphi(1)=3^m-2\ge0.
\]

We show that $\varphi$ is nondecreasing on $[1,\infty)$.  Differentiating gives
\[
        \varphi'(c)
        =
        m\left((2c+1)(c^2+c+1)^{m-1}-2c^{2m-1}\right).
\]
For $c\ge1$, set $v:=1/c\in(0,1]$.  After dividing by the positive factor
$2mc^{2m-1}$, the inequality $\varphi'(c)\ge0$ is equivalent to
\[
        \rho_m(v):=
        \left(1+\frac v2\right)(1+v+v^2)^{m-1}\ge1.
\]
For fixed $v$, $\rho_m(v)$ is increasing in $m$, so it suffices to prove this
for $m=m_0$.  Define
\[
        \Theta(v)
        =
        \log\left(1+\frac v2\right)
        -(1-m_0)\log(1+v+v^2).
\]
Then $\Theta(0)=\Theta(1)=0$, and
\[
        \Theta'(v)>0
        \quad\Longleftrightarrow\quad
        \frac{1+v+v^2}{(2+v)(1+2v)}>1-m_0.
\]
The function
\[
        \Xi(v):=\frac{1+v+v^2}{(2+v)(1+2v)}
\]
satisfies
\[
        \Xi'(v)=\frac{3(v^2-1)}{(2+v)^2(1+2v)^2},
\]
so it is decreasing on $[0,1]$.  Since
\[
        \Xi(0)=\frac12,\qquad \Xi(1)=\frac13,
        \qquad
        \frac13<1-m_0<\frac12,
\]
the derivative $\Theta'$ changes sign exactly once.  Thus $\Theta$ increases and
then decreases, and because both endpoint values are zero, $\Theta\ge0$ on
$[0,1]$.  Hence $\varphi'\ge0$ on $[1,\infty)$, and
\[
        \varphi(c)\ge\varphi(1)\ge0.
\]
The symmetry gives the result for $0<c\le1$.
\end{proof}

\subsection{The four-point Ptolemaic theorem}

In this section we prove that a four-point Ptolemaic metric space has $q$-negative type for all $0<q\le \log_2 3$.
By Lemma~\ref{lem:q5-schoenberg}, an equivalent statement is that if $(X,d)$ is a four-point Ptolemaic metric, then
$(X,d^{q/2})$ embeds isometrically in Euclidean space for $0<q\le\log_2 3$.
This result is sharp, as the following example shows:

\begin{example}
\label{ex:q5-rhombus}
Let $d$ be the four-point metric with
\[
        d(1,2)=2,\qquad d(3,4)=1,
        \qquad
        d(1,3)=d(1,4)=d(2,3)=d(2,4)=1.
\]
This metric is Ptolemaic, since the only nontrivial Ptolemy inequality is
\[
        2\cdot 1=1\cdot1+1\cdot1.
\]
At $q=\log_2 3$, the snowflaked metric $d^{q/2}$ has
\[
        d(1,2)^{q/2}=\sqrt3,\qquad d(3,4)^{q/2}=1,
\]
and all four cross-distances equal to $1$. Thus it is realized by a Euclidean
rhombus with side length $1$ and diagonals $\sqrt3$ and $1$. Equivalently,
the matrix $(d(i,j)^q)$ is conditionally negative semidefinite at $q=\log_2 3$.

The constant $\log_2 3$ is sharp for this example. Indeed, take
\[
        w=(1,1,-1,-1),
        \qquad \sum_i w_i=0.
\]
Then
\[
\begin{aligned}
        \sum_{i,j} w_iw_j d(i,j)^q
        &=
        2\bigl(d(1,2)^q+d(3,4)^q
        -d(1,3)^q-d(1,4)^q-d(2,3)^q-d(2,4)^q\bigr) \\
        &=
        2(2^q+1-4)
        =
        2(2^q-3).
\end{aligned}
\]
For $q>\log_2 3$, this quantity is positive, so $(d(i,j)^q)$ is not conditionally
negative semidefinite. Hence $d$ does not have $q$-negative type for any
$q>\log_2 3$.
\end{example}

We require several preliminary lemmas. 

\begin{lemma}[Star metrics]
\label{lem:q5-star-metrics}
Let $1\le q\le \log_2 3$.  Every four-point star metric with center $o$ and
three leaves has $q$-negative type.
\end{lemma}

\begin{proof}
Let the leaf lengths be $r_1,r_2,r_3$, and use $o$ as base point.  After
conjugating the Schoenberg matrix by the diagonal matrix with entries $r_i^{-q/2}$,
the diagonal entries are $1$, and the off-diagonal entries are $-\eta_{ij}/2$, where
\[
        \eta_{ij}=\frac{(r_i+r_j)^q-r_i^q-r_j^q}{(r_ir_j)^{q/2}}.
\]
By Lemma~\ref{lem:q5-star}, $0\le \eta_{ij}\le1$.  The determinant is
\[
        1-\frac{\eta_{12}^2+\eta_{13}^2+\eta_{23}^2+\eta_{12}\eta_{13}\eta_{23}}{4},
\]
which is nonnegative for $0\le\eta_{ij}\le1$.  The principal $2\times2$ minors
are also nonnegative, so the Schoenberg matrix is positive semidefinite.  By
Lemma~\ref{lem:q5-schoenberg}, the metric has $q$-negative type.
\end{proof}

The following lemma provides a useful concavity result for determinants of symmetric $3 \times 3$ matrices.

\begin{lemma}
\label{lem:q5-one-entry-concavity}
Let
\[
        G=
        \begin{pmatrix}
        A & x & u\\
        x & B & v\\
        u & v & C
        \end{pmatrix}
\]
be a symmetric $3\times3$ matrix.  If $A,B,C,u,v$ are fixed and only $x$ is
allowed to vary, then $\det G$ is a quadratic polynomial in $x$ with coefficient
$-C$.  In particular, if $C\ge0$, then $\det G$ is concave as a function of
$x$.
\end{lemma}

\begin{proof}
Expanding the determinant gives
\[
        \det G=ABC+2xuv-Av^2-Bu^2-Cx^2.\qedhere
\]
\end{proof}

The following lemma concerns the geometry of a triangle $PUV$ with a tail $AP$ attached at $P$.  

\begin{figure}[ht]
\centering
\begin{tikzpicture}[scale=1.05]
  \coordinate (P) at (0,0);
  \coordinate (A) at (-2.35,0);
  \coordinate (U) at (2.05,1.25);
  \coordinate (V) at (2.55,-1.05);

  \draw[thick] (A) -- (P);
  \draw[thick] (P) -- (U);
  \draw[thick] (P) -- (V);
  \draw[thick] (U) -- (V);
  \draw[dashed] (A) to[bend left=9] (U);
  \draw[dashed] (A) to[bend right=9] (V);

  \fill (A) circle (1.7pt) node[left] {$A$};
  \fill (P) circle (1.7pt) node[below left] {$P$};
  \fill (U) circle (1.7pt) node[above right] {$U$};
  \fill (V) circle (1.7pt) node[below right] {$V$};

  \node[below, fill=white, inner sep=1pt] at ($(A)!0.5!(P)$) {$y$};
  \node[above left, fill=white, inner sep=1pt] at ($(P)!0.55!(U)$) {$r$};
  \node[below left, fill=white, inner sep=1pt] at ($(P)!0.55!(V)$) {$z$};
  \node[right, fill=white, inner sep=1pt] at ($(U)!0.5!(V)$) {$h$};
  \node[above, xshift=-14pt, yshift=8pt, fill=white, inner sep=1pt] at ($(A)!0.55!(U)$) {$y+r$};
  \node[below, xshift=-14pt, yshift=-8pt, fill=white, inner sep=1pt] at ($(A)!0.55!(V)$) {$y+z$};
\end{tikzpicture}
\caption{The attached-ray extension: a segment $AP$ is attached to the triangle $PUV$ at $P$.  The dashed curves indicate the additive distances from $A$ to $U$ and $V$.}
\label{fig:attached-ray}
\end{figure}
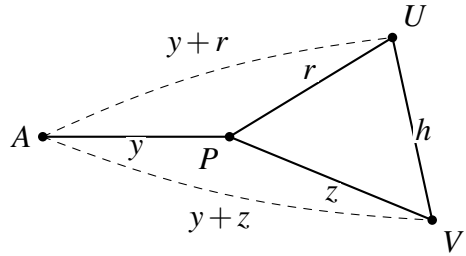

\begin{lemma}[Attached-ray extension]
\label{lem:q5-attached-ray}
Let $1\le q\le\log_2 3$.  Suppose a four-point metric $d$ on $\{A,P,U,V\}$
satisfies
\[
        d(A,P)=y,\qquad
        d(P,U)=r,\qquad
        d(P,V)=z,\qquad
        d(U,V)=h,
\]
and
\[
        d(A,U)=y+r,\qquad d(A,V)=y+z.
\]
Then the metric has $q$-negative type.
\end{lemma}

\begin{proof}
The triangle inequalities for $P,U,V$ give
\[
        |r-z|\le h\le r+z.
\]
Use $A$ as base point and form the Schoenberg matrix $G$ of the snowflaked metric
$d^{q/2}$ on the remaining points $P,U,V$.  
(Note that $d^{q/2}$ is a metric because $d$ is a metric and $q/2<1$.)
The $2\times2$ principal minors
are nonnegative because the corresponding three-point restrictions of $d^{q/2}$ are
metric triangles, hence Euclidean. 
Thus it remains to prove $\det G\ge0$.

The variable $H:=h^q$ appears only through the off-diagonal entry $G_{UV}$.  By
Lemma~\ref{lem:q5-one-entry-concavity}, $\det G$ is concave as a function of $H$.
Since $h\mapsto h^q$ is increasing, $H$ ranges over an interval.  It is enough to
check the endpoints of this interval.

If $h=|r-z|$, the triangle $P,U,V$ is degenerate and the whole four-point metric
is a line metric.  Hence it has $q$-negative type by Lemma~\ref{lem:q5-line-metrics}.
If $h=r+z$, then $P$ lies between $U$ and $V$, and the whole metric is a
three-leaf star with center $P$, with leaves $A,U,V$.  This endpoint is handled by
Lemma~\ref{lem:q5-star-metrics}.  Thus $\det G\ge0$ throughout the interval, so
$G\succeq0$.  By Lemma~\ref{lem:q5-schoenberg}, the metric has $q$-negative type.
\end{proof}

The following Lemma investigates the effect of inserting a point $P$ along one edge $AB$ of a triangle $ABO$, assuming the distance from $O$ to $P$ satisfies a Ptolemy-style upper bound.

\begin{figure}[ht]
\centering
\begin{tikzpicture}[scale=1.05]
  \coordinate (A) at (0,0);
  \coordinate (P) at (2.0,0);
  \coordinate (B) at (4.1,0);
  \coordinate (O) at (2.0,1.65);

  \draw[thick] (A) -- (P) -- (B);
  \draw[thick] (O) -- (A);
  \draw[thick] (O) -- (P);
  \draw[thick] (O) -- (B);

  \fill (A) circle (1.7pt) node[below] {$A$};
  \fill (P) circle (1.7pt) node[below] {$P$};
  \fill (B) circle (1.7pt) node[below] {$B$};
  \fill (O) circle (1.7pt) node[above] {$O$};

  \node[below] at ($(A)!0.5!(P)$) {$y$};
  \node[below] at ($(P)!0.5!(B)$) {$z$};
  \node[left] at ($(O)!0.5!(A)$) {$a$};
  \node[right] at ($(O)!0.5!(P)$) {$r$};
  \node[right] at ($(O)!0.5!(B)$) {$b$};
\end{tikzpicture}
\caption{Geodesic insertion: $P$ lies on the segment $AB$, and $r=d(O,P)$ is constrained by $(y+z)r\le za+yb$.}
\label{fig:geodesic-insertion}
\end{figure}
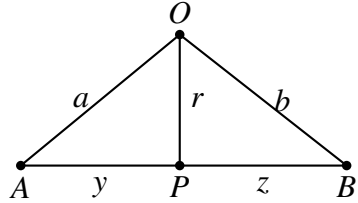

\begin{lemma}[Geodesic insertion]
\label{lem:q5-radial}
Let $1\le q\le\log_2 3$.  Let $\{O,A,P,B\}$ be a four-point metric such that
$P$ lies on a geodesic from $A$ to $B$.  Write
\[
        d(A,P)=y,\qquad d(P,B)=z,\qquad d(A,B)=y+z,
\]
and
\[
        d(O,A)=a,\qquad d(O,P)=r,\qquad d(O,B)=b.
\]
Assume
\[
        (y+z)r\le za+yb.
        \tag{$*$}
        \label{eq:q5-radial-ptolemy}
\]
Then the metric has $q$-negative type.
\end{lemma}

\begin{proof}
By scaling all distances by the positive factor $b^{-1}$, which does not affect
$q$-negative type, we may assume that $b=1$.  We retain the same notation for
the rescaled distances.  Set
\[
        s:=y+z,
        \qquad
        r_*:=\frac{za+y}{y+z}.
\]
Since $\{O,A,P,B\}$ is a four-point metric space, the four points are distinct;
in particular, $y,z>0$.

Regard $r$ as variable while keeping $a,y,z$ fixed.  The closure of the set of
positive values of $r$ for which the prescribed distances define a metric and
satisfy \eqref{eq:q5-radial-ptolemy} is the nonempty interval
$r_0\le r\le r_1$, where
\[
        r_0:=\max\{|a-y|,\ |1-z|\},
        \qquad
        r_1:=\min\{a+y,\ 1+z,\ r_*\}.
\]
Indeed, the lower bounds and the first two upper bounds are exactly the triangle
inequalities for $OAP$ and $OBP$, while
\eqref{eq:q5-radial-ptolemy} is equivalent to $r\le r_*$.  The triangle
inequalities for $OAB$ are independent of $r$ and already hold.

Use $A$ as basepoint, and let $G(r)$ be the Schoenberg matrix of $d^{q/2}$ on
$\{P,O,B\}$.  Its $1\times1$ and $2\times2$ principal minors are nonnegative:
since $0<q/2\le1$, snowflaking preserves the triangle inequalities, and every
metric triangle is Euclidean.  It therefore remains to prove that
$\det G(r)\ge0$.

The quantity $R:=r^q$ occurs only in the symmetric pair of entries
\[
        [G(r)]_{PO}=[G(r)]_{OP}
        =\frac12\bigl(y^q+a^q-R\bigr).
\]
By Lemma~\ref{lem:q5-one-entry-concavity}, $\det G(r)$ is a concave
quadratic polynomial in $R$, with quadratic coefficient
\[
        -\frac14 d(A,B)^q=-\frac14 s^q.
\]
Since $r\mapsto r^q$ maps $[r_0,r_1]$ onto $[r_0^q,r_1^q]$, it is enough
to check the endpoints $r=r_0$ and $r=r_1$.

We first consider the lower endpoint.  If $r_0=0$, then $a=y$ and $z=1$.

The formula for $G(r)$ extends continuously to $r=0$.  Since then
$a=y$ and
\[
        d(P,B)=z=1=d(O,B),
\]
the rows of $G(0)$ indexed by $P$ and $O$ are identical.  Thus
$\det G(0)=0$.

Assume now that $r_0>0$.  If $r_0=|a-y|$, then the triangle $OAP$ is
degenerate.  If $a=r_0+y$,
then Lemma~\ref{lem:q5-attached-ray} applies with
\[
        (A,P,U,V)=(A,P,O,B).
\]
If instead $y=a+r_0$,
then the triangle inequality for $OBP$ gives $1\le r_0+z$, and hence
\[
        s=y+z=a+(r_0+z)\ge a+1.
\]
The triangle inequality for $OAB$ gives $s\le a+1$.  Thus
\[
        s=a+1,
        \qquad
        r_0+z=1.
\]
The four points are therefore realized on a line in the order $A,\ O,\ P,\ B$,
with consecutive distances $a,r_0,z$.  This endpoint is handled by
Lemma~\ref{lem:q5-line-metrics}.

This settles the lower endpoint whenever $r_0=|a-y|$.  We may therefore assume
that $r_0\ne|a-y|$.  Then $r_0=|1-z|$,
so the triangle $OBP$ is degenerate.  If
$1=r_0+z$, then Lemma~\ref{lem:q5-attached-ray} applies with
\[
        (A,P,U,V)=(B,P,O,A).
\]
If instead $z=1+r_0$,
then the triangle inequality for $OAP$ gives $a\le r_0+y$, and hence
\[
        s=y+z=1+(r_0+y)\ge a+1.
\]
Again $s\le a+1$ by the triangle inequality for $OAB$, so
\[
        s=a+1,
        \qquad
        a=r_0+y.
\]
The four points are therefore realized on a line in the order $A,\ P,\ O,\ B$,
with consecutive distances $y,r_0,1$.  This endpoint is again handled by
Lemma~\ref{lem:q5-line-metrics}.

We now consider the upper endpoint.  If
$r_1=a+y$,
then \eqref{eq:q5-radial-ptolemy} gives
\[
        (y+z)(a+y)\le za+y.
\]
After canceling $za$, this becomes
\[
        y(a+y+z)\le y.
\]
Since $y>0$, we obtain $a+s\le1$.  The reverse inequality $1\le a+s$
is the triangle inequality for $OAB$.  Hence $a+s=1$, and the four points form
a line metric in the order $O,\ A,\ P,\ B$
with consecutive distances $a,y,z$.

If the preceding case does not apply and
$r_1=1+z$, then \eqref{eq:q5-radial-ptolemy} gives
\[
        (y+z)(1+z)\le za+y.
\]
After canceling $y$, this becomes
\[
        z(1+y+z)\le za.
\]
Since $z>0$, we obtain $1+s\le a$.  The reverse inequality $a\le1+s$
is the triangle inequality for $OAB$.  Hence $a=1+s$, and the four points form
a line metric in the order $O,\ B,\ P,\ A$
with consecutive distances $1,z,y$.  Thus every upper endpoint arising from
a metric bound is handled by Lemma~\ref{lem:q5-line-metrics}.

It remains to treat the case $r_1=r_*$,
with neither metric upper bound handled above.  Then
\eqref{eq:q5-radial-ptolemy} is an equality:
\[
        sr=za+y.
\]
Define the inverted distance $\widehat d$ at $A$.  We first verify that
$\widehat d$ is a metric.  Every triangle containing $A$ satisfies the triangle
inequalities, since after clearing positive denominators these are exactly the
corresponding triangle inequalities for $d$.  For the remaining triangle on
$\{P,B,O\}$, the equality above gives
\[
\begin{aligned}
        \widehat d(P,O)
        &=\frac{r}{ay}
          =\frac{1}{as}+\frac{z}{sy} \\
        &=\widehat d(O,B)+\widehat d(B,P).
\end{aligned}
\]
Thus $\widehat d$ is a metric.  Moreover, the geodesic relation $s=y+z$ gives
\[
        \widehat d(A,P)
        =\frac1y
        =\frac1s+\frac{z}{sy}
        =\widehat d(A,B)+\widehat d(B,P).
\]
Consequently Lemma~\ref{lem:q5-attached-ray} applies to $\widehat d$, with the
lemma's labels $(A,P,U,V)$ replaced by $(P,B,A,O)$.
Hence $\widehat d$ has $q$-negative type.  By
Lemmas~\ref{lem:q5-inversion-euclidean} and~\ref{lem:q5-schoenberg}, the
original metric $d$ has $q$-negative type as well.

We have proved that $\det G(r)$ is nonnegative at the two endpoints
$R=r_0^q$ and $R=r_1^q$.
By concavity, $\det G(r)\ge0$ throughout that interval, in particular
at the original value of $r$.  Together with the nonnegative principal minors of
orders $1$ and $2$, this shows that $G(r)$ is positive semidefinite.  The
conclusion now follows from Lemma~\ref{lem:q5-schoenberg}.
\end{proof}

With these preliminaries out of the way, we can now prove the main result of this section.

\begin{theorem}[Four-point Ptolemaic snowflake theorem]
\label{thm:q5-four-point-ptolemaic}
Every four-point Ptolemaic metric has $q$-negative type for
\[
        0<q\le \log_2 3.
\]
Equivalently, if $(X,d)$ is a four-point Ptolemaic metric, then
$(X,d^{q/2})$ embeds isometrically in Euclidean space for
$0<q\le\log_2 3$.
\end{theorem}

\begin{proof}
For $0<q\le1$, this is Blumenthal's theorem. So we may assume that $1\le q\le\log_2 3$.
We write $X=\{1,2,3,4\}$ and use $4$ as a base point. 
Put
\[
        \rho_i=d(i,4)\qquad (i=1,2,3),
        \qquad
        \delta_{ij}=d(i,j)\qquad (1\le i<j\le3).
\]

Since $d$ is a metric, $\rho_i>0$ for $i=1,2,3$.  By
Lemma~\ref{lem:q5-schoenberg}, it is enough to show that the normalized
Schoenberg matrix $K$ based at $4$, with diagonal entries $1$ and
off-diagonal entries
\[
        K_{ij}=\frac{\rho_i^q+\rho_j^q-\delta_{ij}^q}
        {2\rho_i^{q/2}\rho_j^{q/2}}\qquad(1\le i<j\le3),
\]
is positive semidefinite.  
Its $2\times2$ principal minors are nonnegative: since
$0<q/2\le1$, the snowflaked distance $d^{q/2}$ is again a metric, and every
three-point metric embeds isometrically in the Euclidean plane.  It therefore
suffices to prove that $\det K\ge0$.
 
We obtain this by pushing a single distance to the boundary of its feasible
range.  Hold $\rho_1,\rho_2,\rho_3,\delta_{13},\delta_{23}$ fixed and regard
$\det K$ as a function of $\delta_{12}$ alone.  
Set $W:=\delta_{12}^q$.
Only the symmetric pair of entries $K_{12}=K_{21}$ depends on $W$, and
\[
        K_{12}
        =
        \frac{\rho_1^q+\rho_2^q-W}
        {2\rho_1^{q/2}\rho_2^{q/2}}.
\]
Thus $K_{12}$ is an affine function of $W$.  By
Lemma~\ref{lem:q5-one-entry-concavity}, applied with diagonal entry $C=1\ge0$,
$\det K$ is a concave quadratic polynomial in $K_{12}$, and hence a concave
quadratic polynomial in $W$.
 
With the other distances fixed, the positive values of $\delta_{12}$ for
which the prescribed distances define a four-point Ptolemaic metric are
exactly those satisfying the triangle inequalities for $\{1,2,4\}$ and
$\{1,2,3\}$ and the Ptolemy inequalities for $\{1,2,3,4\}$.  Explicitly, this set is
$(0,\infty)\cap[\delta^-,\delta^+]$,
where
\[
        \delta^-=\max\!\left\{\,|\rho_1-\rho_2|,\quad
              |\delta_{13}-\delta_{23}|,\quad
              \frac{|\rho_2\delta_{13}-\rho_1\delta_{23}|}{\rho_3}\,\right\}
\]
and
\[
        \delta^+=\min\!\left\{\,\rho_1+\rho_2,\quad
              \delta_{13}+\delta_{23},\quad
              \frac{\rho_2\delta_{13}+\rho_1\delta_{23}}{\rho_3}\,\right\}.
\]
The closed interval $[\delta^-,\delta^+]$ is nonempty, since it contains the
original positive value of $\delta_{12}$, and the map $\delta_{12}\mapsto W=\delta_{12}^q$ sends
$[\delta^-,\delta^+]$ onto $[(\delta^-)^q,(\delta^+)^q]$.
Since $\det K$ is concave as a function of $W$, its minimum over this interval
is attained at an endpoint.  Hence it is enough to prove $\det K\ge0$ when
$\delta_{12}=\delta^-$ and when $\delta_{12}=\delta^+$.

There is one possible auxiliary endpoint which is not itself a four-point
metric.  If $\delta^-=0$, then
\[
        \rho_1=\rho_2,
        \qquad
        \delta_{13}=\delta_{23}.
\]
At $\delta_{12}=0$ we therefore have
\[
        K_{12}=1,
        \qquad
        K_{13}=K_{23}.
\]
The first two rows of $K$ coincide, and hence $\det K=0$ in this case.

Every endpoint not already handled is positive.  At such an endpoint, the
modified distances define a genuine four-point Ptolemaic metric, and at least
one of the defining triangle or Ptolemy inequalities holds with equality.
We now fix such an endpoint and distinguish two cases.  

\smallskip
\noindent\emph{Case 1: A triangle inequality is tight.}
 
There are four possibilities, all handled by Lemma~\ref{lem:q5-radial}; in each
the required hypothesis~\eqref{eq:q5-radial-ptolemy} is one of the Ptolemy
inequalities of the quadruple.
\begin{itemize}
\item $\delta_{12}=\rho_1+\rho_2$: then $d(1,2)=d(1,4)+d(2,4)$, so $4$ lies on a
        geodesic from $1$ to $2$.  Apply Lemma~\ref{lem:q5-radial} with
        $(O,A,P,B)=(3,1,4,2)$; here \eqref{eq:q5-radial-ptolemy} is
        $\delta_{12}\rho_3\le\delta_{13}\rho_2+\rho_1\delta_{23}$.
\item $\delta_{12}=\delta_{13}+\delta_{23}$: then
        $d(1,2)=d(1,3)+d(3,2)$, so $3$ lies on a geodesic from $1$ to $2$.
        Apply Lemma~\ref{lem:q5-radial} with $(O,A,P,B)=(4,1,3,2)$; again
        \eqref{eq:q5-radial-ptolemy} is
        $\delta_{12}\rho_3\le\delta_{13}\rho_2+\rho_1\delta_{23}$.
\item $\delta_{12}=|\rho_1-\rho_2|$: after relabelling so that
        $\rho_1\ge\rho_2$, $d(1,4)=d(1,2)+d(2,4)$, so $2$ lies on a geodesic
        from $1$ to $4$.  Apply Lemma~\ref{lem:q5-radial} with
        $(O,A,P,B)=(3,1,2,4)$; here \eqref{eq:q5-radial-ptolemy} is
        $\rho_1\delta_{23}\le\delta_{12}\rho_3+\delta_{13}\rho_2$.
\item $\delta_{12}=|\delta_{13}-\delta_{23}|$: after relabelling so that
        $\delta_{13}\ge\delta_{23}$, $d(1,3)=d(1,2)+d(2,3)$, so $2$ lies on a
        geodesic from $1$ to $3$.  Apply Lemma~\ref{lem:q5-radial} with
        $(O,A,P,B)=(4,1,2,3)$; here \eqref{eq:q5-radial-ptolemy} is
        $\delta_{13}\rho_2\le\delta_{12}\rho_3+\delta_{23}\rho_1$.
\end{itemize}
In each case $\det K\ge0$ at the endpoint.

\smallskip
\noindent\emph{Case 2: No triangle inequality is tight.}
Then a Ptolemy inequality is tight, so one of the three products
\[
        \delta_{12}\rho_3,\qquad
        \delta_{13}\rho_2,\qquad
        \delta_{23}\rho_1
\]
is the sum of the other two.  Since $q$-negative type is invariant under
relabeling, we may assume that
\[
        \delta_{12}\rho_3
        =
        \delta_{13}\rho_2+\delta_{23}\rho_1.
\]
Invert the metric with respect to the point $3$.  By
Lemma~\ref{lem:q5-ptolemy-inversion}, the inverted distance $\widehat d$
is again a metric.  Moreover,
\[
\begin{aligned}
        \widehat d(1,2)
        &=
        \frac{\delta_{12}}{\delta_{13}\delta_{23}} \\
        &=
        \frac{\rho_2}{\rho_3\delta_{23}}
        +
        \frac{\rho_1}{\rho_3\delta_{13}} \\
        &=
        \widehat d(4,2)+\widehat d(1,4).
\end{aligned}
\]
Thus $4$ lies on a geodesic from $1$ to $2$ in the inverted metric.
We now apply Lemma~\ref{lem:q5-radial} to $\widehat d$ with $(O,A,P,B)=(3,1,4,2)$.
The required inequality is
\[
        \widehat d(1,2)\widehat d(3,4)
        \le
        \widehat d(4,2)\widehat d(3,1)
        +
        \widehat d(1,4)\widehat d(3,2).
\]
After multiplying by
$\rho_3\delta_{13}\delta_{23}$, this becomes $\delta_{12}\le\rho_1+\rho_2$,
which is the triangle inequality for the triple $\{1,2,4\}$.
Therefore Lemma~\ref{lem:q5-radial} applies to $\widehat d$.
By Lemmas~\ref{lem:q5-inversion-euclidean}
and~\ref{lem:q5-schoenberg}, the original endpoint metric has
$q$-negative type, and hence $\det K\ge0$.

\smallskip
We have shown that $\det K\ge0$ at each positive endpoint; the possible
zero endpoint was handled separately above.  Hence $\det K\ge0$ at both
$\delta_{12}=\delta^-$ and $\delta_{12}=\delta^+$.  By concavity in
$W=\delta_{12}^q$, it follows that $\det K\ge0$ for every
$\delta_{12}\in[\delta^-,\delta^+]$, and in particular at the original value
of $\delta_{12}$.  Together with the nonnegativity of the $2\times2$
principal minors, this shows that $K\succeq0$, and
Lemma~\ref{lem:q5-schoenberg} completes the proof.

\end{proof}

\begin{corollary}
\label{cor:q5-sharp}
We have
\[
        q(5)=\log_2 3.
\]
\end{corollary}

\begin{proof}
By Proposition~\ref{prop:qn-Pn} and Theorem~\ref{thm:q5-four-point-ptolemaic},
\[
        q(5)=P(4)\ge \log_2 3.
\]
The opposite inequality is Proposition~\ref{prop:qn-upper} in the case $n=5$.
\end{proof}

\section{Strong representations and Lorentzian polynomials}
\label{sec:strong-representations}

Up to this point we have worked exclusively with \emph{weak} representations of matroids
and, more generally, $M$-convex sets. The corresponding theory of \emph{strong}
representations appears to be substantially subtler.

We will not need the full formalism here, so we refer to \cite{BHKL0} for the definition of
a strong representation over a tract. We write
\[
\upR_J^{\rm str}(F)\subseteq \upR_J^{\rm w}(F)
\]
for the space of strong $F$-representations of an $M$-convex set $J$. Thus every
strong representation is weak, but not conversely.

Our purpose in this section is not to give a definitive account of the strong case, but
rather to explain why it is nontrivial and why it does not seem to reduce easily to the weak
theory developed in the rest of the paper. On the weak side, the main theorems are driven by
local quadratic tests and reductions to rank-$2$ uniform matroids. The examples below show
that strong $\T_q$-representability involves genuinely global constraints, and that neither
a uniform passage from weak to strong nor a naive downward-induction argument can hold in
general.

We begin with one statement that does survive formally from the weak theory.

\begin{proposition}
\label{prop:strong-lower-trivial}
Let $M$ be a matroid, and let $q_M>0$ be as in 
Theorem~\ref{thm:matroid-sandwich}. Then for every $0<q\le q_M$,
\[
\upR_M^{\rm str}(\T_q)\subseteq \upR_M^{\rm w}(\T_q)\subseteq L_M.
\]
In particular,
\[
\Gr_M^{\rm str}(\T_q)\subseteq \P\upL_M.
\]
\end{proposition}

\begin{proof}
This is immediate from the inclusion
\[
\upR_M^{\rm str}(\T_q)\subseteq \upR_M^{\rm w}(\T_q)
\]
and Theorem~\ref{thm:matroid-sandwich}.
\end{proof}

Thus the lower-bound problem becomes no harder in the strong setting. The real difficulty
is the reverse direction: unlike weak $\T_q$-representability, strong
$\T_q$-representability imposes higher Pl\"ucker-type relations, and these are not
controlled by the quadratic local tests that characterize Lorentzianity.

The first example shows that there is no uniform way to promote weak $\T_1$-representations
to strong $\T_c$-representations by changing the parameter.

\begin{proposition}
\label{prop:no-uniform-weak-to-strong}
There is no constant $c>0$ such that every weak $\T_1$-representation of every
$M$-convex set is a strong $\T_c$-representation.
\end{proposition}

\begin{proof}
For each integer $d\ge 2$, let $J=\Delta_2^d$
and define a coefficient function $\rho_d:J\to \R_{\ge 0}$ by
\[
\rho_d(d-k,k)=a_k:=2^{\binom{k}{2}}
\qquad (0\le k\le d).
\]
Let
\[
f_d(x,y):=\sum_{k=0}^d a_k \frac{x^{d-k}}{(d-k)!}\frac{y^k}{k!}.
\]

By Lemma~\ref{lem:bivariate-Tq-ineq}, the polynomial $f_d$ is a weak $\T_1$-representation if
and only if
\[
a_{k-1}a_{k+1}\le 2a_k^2
\qquad (1\le k\le d-1).
\]
But
\[
\binom{k-1}{2}+\binom{k+1}{2}=2\binom{k}{2}+1,
\]
so
\[
a_{k-1}a_{k+1}
=
2^{\binom{k-1}{2}+\binom{k+1}{2}}
=
2^{2\binom{k}{2}+1}
=
2a_k^2.
\]
Hence $f_d$ is weak $\T_1$ for every $d\ge 2$.

Now suppose for the sake of contradiction that there exists a constant $c>0$
such that every weak $\T_1$-representation of every $M$-convex set is a strong
$\T_c$-representation. Then in particular $\rho_d$ is a strong
$\T_c$-representation for every $d$.
Apply the strong Pl\"ucker relation with
\[
s=d,\qquad \alpha=(0,0),\qquad i_0=1,\qquad i_1=\cdots=i_d=2,\qquad
j_2=\cdots=j_d=1.
\]
The condition
\[
\alpha+e_{i_0}+\cdots+e_{i_d}+e_{j_2}+\cdots+e_{j_d}\le (d,d)=\delta_J^+
\]
is satisfied, since the left-hand side is exactly $d e_1+d e_2$.

The summation index in this strong Pl\"ucker relation runs from $0$ to $d$.
The $k=0$ term is
\[
\rho_d(de_2)\rho_d(de_1)=a_d a_0.
\]
For each summation index $k=1,\dots,d$, the corresponding term is
\[
\rho_d(e_1+(d-1)e_2)\rho_d((d-1)e_1+e_2)=a_{d-1}a_1.
\]
Thus the strong Pl\"ucker
relation becomes
\[
a_0a_d+\underbrace{a_1a_{d-1}+\cdots+a_1a_{d-1}}_{d\text{ times}}\in N_{\T_c}.
\]

By definition of $\T_c$, this means that the $d+1$ nonnegative numbers
\[
(a_0a_d)^{1/c},\underbrace{(a_1a_{d-1})^{1/c},\dots,(a_1a_{d-1})^{1/c}}_{d\text{ times}}
\]
are the side lengths of a (possibly degenerate) Euclidean $(d+1)$-gon. In particular, each
side length is at most the sum of the others, so
\[
(a_0a_d)^{1/c}\le d\,(a_1a_{d-1})^{1/c}.
\]
Raising both sides to the $c$-th power yields the endpoint inequality
\[
a_0a_d\le d^c a_1a_{d-1}.
\]

Substituting $a_k=2^{\binom{k}{2}}$, and noting that $a_0=a_1=1$, we obtain
\[
2^{\binom{d}{2}}
\le
d^c\,2^{\binom{d-1}{2}}.
\]
Equivalently,
\[
2^{d-1}\le d^c.
\]
For fixed $c>0$, this fails for all sufficiently large $d$, a contradiction.

Therefore there is no constant $c>0$ such that every weak $\T_1$-representation of every
$M$-convex set is a strong $\T_c$-representation.
\end{proof}

The obstruction already appears in the matroidal setting.

\begin{corollary}
\label{cor:no-uniform-weak-to-strong-matroid}
There is no constant $c>0$ such that every weak $\T_1$-representation of every matroid
is a strong $\T_c$-representation.

More concretely, for every $c>0$ and all sufficiently large $d$, there exists a weak
$\T_1$-representation of the uniform matroid $U_{d,2d}$ which is not a strong
$\T_c$-representation.
\end{corollary}

\begin{proof}
For each $d\ge 2$, let
\[
J_d=\Delta_2^d,
\]
and let $\rho_d$ be the weak $\T_1$-representation of $J_d$ constructed in
Proposition~\ref{prop:no-uniform-weak-to-strong}. By that proposition, for every fixed
$c>0$ and all sufficiently large $d$, the representation $\rho_d$ is not a strong
$\T_c$-representation.

Let $N_d$ be the natural matroid of $J_d$. Since $J_d=\Delta_2^d$, the ground set of
$N_d$ is the disjoint union of two blocks $E_1$ and $E_2$, each of size $d$, and a
squarefree vector $\beta\in \{0,1\}^{E_1\sqcup E_2}$ is a basis of $N_d$ exactly when
$\theta(\beta)\in \Delta_2^d$. Equivalently, $\beta$ is a basis exactly when it has total
size $d$. Therefore
\[
N_d \cong U_{d,2d}.
\]

Now apply the up operator $\Pi^\uparrow$. By Proposition~4.8 of \cite{BHKL0},
$\Pi^\uparrow\rho_d$ is a weak $\T_1$-representation of $N_d$, and
$\Pi^\uparrow\rho_d$ is a strong $\T_c$-representation if and only if $\rho_d$ is a
strong $\T_c$-representation. Hence, for every fixed $c>0$ and all sufficiently large
$d$, the matroid representation $\Pi^\uparrow\rho_d$ is weak $\T_1$ but not strong
$\T_c$.

Since $N_d\cong U_{d,2d}$, this proves the claim.
\end{proof}

Proposition~\ref{prop:no-uniform-weak-to-strong} and
Corollary~\ref{cor:no-uniform-weak-to-strong-matroid} show that strong representability
cannot be recovered from weak representability by any universal reparametrization, even on
the very concrete family of uniform matroids $U_{d,2d}$.

A second obstruction comes from the downward operator. On the Lorentzian side, downward
operators are central. For $\T_q$-representations, however, downward operators are already delicate in the
weak setting, and they behave even worse in the strong setting.

For a polynomial $f(x_1,\dots,x_6)$, write
\[
D_{5,6}(f):=f(x_1,x_2,x_3,x_4,x_5,x_5).
\]
Thus $D_{5,6}$ identifies the variables $x_5$ and $x_6$.

\begin{proposition}
\label{prop:down-not-weak}
For every $q>0$, there exists a weak $\T_q$-representation of the uniform matroid
$U_{3,6}$ whose image under $D_{5,6}$ is not a weak $\T_q$-representation.
\end{proposition}

\begin{proof}
Fix $q>0$, and define a multi-affine cubic polynomial
\[
f_q(x_1,\dots,x_6)=\sum_{S\in \binom{[6]}{3}} c_S x^S
\]
by setting
\[
c_{135}=c_{246}=2^q,
\qquad
c_S=1 \text{ for all other } S\in \binom{[6]}{3}.
\]
Since the support is all of $\binom{[6]}{3}$, this is a coefficient function on the uniform
matroid $U_{3,6}$.

We claim that $f_q$ is a weak $\T_q$-representation. Since $f_q$ is multi-affine of degree
$3$, the only relations to check are the $3$-term Pl\"ucker relations. Each such relation is
indexed by a choice of $\alpha\in \Delta_6^1$ and distinct $i,j,k,\ell\in [6]$, hence by a
$5$-element subset of $[6]$. The only exceptional coefficients are $c_{135}$ and $c_{246}$,
and these correspond to disjoint $3$-subsets. Therefore any fixed $5$-element subset of
$[6]$ contains at most one of them. It follows that in every $3$-term Pl\"ucker relation,
the three terms are either
\[
1,\ 1,\ 1
\qquad\text{or}\qquad
1,\ 1,\ 2^q.
\]
For $\T_q$, these correspond to side lengths
\[
1,\ 1,\ 1
\qquad\text{or}\qquad
1,\ 1,\ 2,
\]
and both triples satisfy the Euclidean triangle inequality. Hence $f_q$ is weak $\T_q$.

Now set
\[
g_q:=D_{5,6}(f_q)=f_q(x_1,x_2,x_3,x_4,x_5,x_5).
\]
The coefficient of $x_i x_j x_5$ in $g_q$ is $c_{ij5}+c_{ij6}$. In particular,
\[
[g_q]_{125}=[g_q]_{145}=[g_q]_{235}=[g_q]_{345}=2,
\]
while
\[
[g_q]_{135}=[g_q]_{245}=1+2^q.
\]
Consider the $3$-term Pl\"ucker relation in $\partial_5 g_q$ corresponding to the quartet
$\{1,2,3,4\}$. Its three terms are
\[
[g_q]_{235}[g_q]_{145}=4,
\qquad
[g_q]_{135}[g_q]_{245}=(1+2^q)^2,
\qquad
[g_q]_{125}[g_q]_{345}=4.
\]
If this relation belonged to $N_{\T_q}$, then the three side lengths
\[
4^{1/q},\qquad (1+2^q)^{2/q},\qquad 4^{1/q}
\]
would form a Euclidean triangle, so we would have
\[
(1+2^q)^{2/q}\le 2\cdot 4^{1/q}.
\]
Raising both sides to the $q/2$-th power gives
\[
1+2^q\le 2^{1+q/2}.
\]
But
\[
1+2^q-2^{1+q/2}=(2^{q/2}-1)^2>0
\]
for every $q>0$. This is a contradiction. Hence $g_q$ is not a weak
$\T_q$-representation.
\end{proof}

Thus, even weak $\T_q$-representability is not stable under the basic downward operation
$x_6\mapsto x_5$. The next example shows that the strong theory can fail even more
dramatically: a strong $\T_q$-representation may be sent to a polynomial that is not even
weakly $\T_q$.

\begin{proposition}
\label{prop:down-not-strong}
For every $q\in[1,2]$, there exists a strong $\T_q$-representation $F_q$ of a
rank-$3$ matroid on $[6]$ such that $D_{5,6}(F_q)$ is not even a weak
$\T_q$-representation.
\end{proposition}

\begin{proof}
Fix $q\in[1,2]$, and define a multi-affine cubic polynomial
\[
F_q(x_1,\dots,x_6)=\sum_{S\in \binom{[6]}{3}} c_S x^S
\]
by
\[
c_{126}=0,\qquad c_{346}=0,\qquad c_{135}=2^q,
\]
and
\[
c_S=1 \text{ for all other } S\in \binom{[6]}{3}.
\]
Its support is
\[
\binom{[6]}{3}\setminus\{\{1,2,6\},\{3,4,6\}\}.
\]
This is the set of bases of a rank-$3$ matroid on $[6]$: indeed, the only nonbases are
$126$ and $346$, and a direct basis-exchange check is immediate since these two forbidden
triples intersect in only one element.

We claim that $F_q$ is a strong $\T_q$-representation. Since the support is multi-affine of
rank $3$, the only nontrivial strong Pl\"ucker relations are the $3$-term and $4$-term
relations.

For the $3$-term relations, a direct inspection of the finitely many possibilities shows that
the three terms are always of one of the forms
\[
1,\ 1,\ 1,
\qquad
0,\ 1,\ 1,
\qquad
1,\ 1,\ 2^q.
\]
Each of these lies in $N_{\T_q}$: after taking $q$-th roots, the corresponding side lengths
are
\[
1,\ 1,\ 1,
\qquad
0,\ 1,\ 1,
\qquad
1,\ 1,\ 2,
\]
all of which satisfy the Euclidean triangle inequality.

For the $4$-term relations, fix a $2$-subset $J\subseteq [6]$, and write
\[
[6]\setminus J=\{a_1,a_2,a_3,a_4\}.
\]
The corresponding strong relation has the form
\[
c_{J\cup\{a_1\}}c_{[6]\setminus(J\cup\{a_1\})}
+
c_{J\cup\{a_2\}}c_{[6]\setminus(J\cup\{a_2\})}
+
c_{J\cup\{a_3\}}c_{[6]\setminus(J\cup\{a_3\})}
+
c_{J\cup\{a_4\}}c_{[6]\setminus(J\cup\{a_4\})}\in N_{\T_q}.
\]
The only non-unit coefficients are $c_{126}=0$, $c_{346}=0$, and $c_{135}=2^q$, while
their complementary triples have coefficients
\[
c_{345}=1,\qquad c_{125}=1,\qquad c_{246}=1.
\]
It follows that each of the four terms in the above relation is either $0$, $1$, or
$2^q$. A direct inspection of the $15$ choices of $J$ shows that the resulting $4$-tuple
of terms is always one of
\[
1,\ 1,\ 1,\ 1;
\qquad
0,\ 1,\ 1,\ 1;
\qquad
0,\ 0,\ 1,\ 1;
\qquad
1,\ 1,\ 1,\ 2^q;
\qquad
0,\ 1,\ 1,\ 2^q.
\]
After taking $q$-th roots, these become side-length multisets
\[
(1,1,1,1),\qquad (0,1,1,1),\qquad (0,0,1,1),\qquad (1,1,1,2),\qquad (0,1,1,2),
\]
and in each case the largest side is at most the sum of the others. Hence every
$4$-term relation belongs to $N_{\T_q}$. Therefore $F_q$ is a strong
$\T_q$-representation.

Now set
\[
G_q:=D_{5,6}(F_q)=F_q(x_1,x_2,x_3,x_4,x_5,x_5).
\]
The coefficients of the monomials involving $x_5$ are
\[
[G_q]_{125}=1,\qquad [G_q]_{145}=2,\qquad [G_q]_{235}=2,\qquad [G_q]_{345}=1,
\]
and
\[
[G_q]_{135}=1+2^q,\qquad [G_q]_{245}=2.
\]
Consider again the $3$-term Pl\"ucker relation in $\partial_5 G_q$ corresponding to the
quartet $\{1,2,3,4\}$. Its three terms are
\[
[G_q]_{235}[G_q]_{145}=4,
\qquad
[G_q]_{135}[G_q]_{245}=2(1+2^q),
\qquad
[G_q]_{125}[G_q]_{345}=1.
\]
If this belonged to $N_{\T_q}$, then the three side lengths
\[
4^{1/q},\qquad \bigl(2(1+2^q)\bigr)^{1/q},\qquad 1
\]
would form a Euclidean triangle. 
The middle term is the largest of the corresponding three side lengths, so the triangle
inequality would imply
\[
\bigl(2(1+2^q)\bigr)^{1/q}\le 1+4^{1/q}.
\]
For $1\le q\le 2$, raising to the $q$-th power and using
\[
(a+b)^q\le 2^{q-1}(a^q+b^q)
\]
gives
\[
2(1+2^q)\le (1+4^{1/q})^q\le 5\cdot 2^{q-1}.
\]
But for $1\le q<2$,
\[
2(1+2^q)-5\cdot 2^{q-1}=2-2^{q-1}>0,
\]
a contradiction. For $q=2$, we have directly
\[
2(1+2^2)=10>(1+4^{1/2})^2=(1+2)^2=9,
\]
again a contradiction. 
Thus the displayed three-term relation does not hold
over $\T_q$, and hence $G_q$ is not a weak $\T_q$-representation.
\end{proof}

These examples naturally suggest the following questions.

\begin{question}
\label{question:downward-threshold}
For each degree $d$, does there exist a constant $q_d>0$ such that the downward
operator preserves strong $\T_q$-representability for all $q\ge q_d$ on multi-affine
degree-$d$ polynomials?
\end{question}

\begin{question}
\label{question:strong-upper}
Let $J\subseteq \Delta_n^d$ be an $M$-convex set, or let $M$ be a matroid.
Does there exist a constant $p^{\rm str}(J)\in (0,\infty]$ such that
\[
\upL_J \subseteq \upN \upR_J^{\rm str}(\T_{p^{\rm str}(J)})?
\]
In the matroidal case, does there exist $p^{\rm str}(M)$ such that
\[
\P\upL_M \subseteq \Gr_M^{\rm str}(\T_{p^{\rm str}(M)})?
\]
\end{question}

\appendix

\section{Six proofs that tree distance matrices are strictly Lorentzian}
\label{app:tree-distance-matrices}

One of the foundational results in \cite{Branden-Huh20} is that every weak
$\T_0$-representation determines a Lorentzian polynomial.  In the rank-$2$
uniform case, a weak $\T_0$-representation is encoded by a positive coefficient
vector $(\rho_{ij})_{1\le i<j\le n}$ such that $(\log \rho_{ij})$ is a rank-$2$
tropical Pl\"ucker vector, or equivalently, a tree metric up to the usual lineality
$c_i+c_j$.  Thus the final linear-algebraic input needed in rank $2$ is the
following theorem.

\begin{theorem}[Exponential tree-metric theorem]
\label{thm:exp-tree-Lorentzian}
Let $n\ge 2$, and let $d=d_T$ be a tree distance on $X=\{1,\dots,n\}$, arising from a weighted
tree $T$ with positive edge lengths whose leaves are labeled by $X$.  Let
$A\in \R^{n\times n}$ be the symmetric zero-diagonal matrix
\[
A_{ij}=e^{d(i,j)}\quad (i\neq j),
\qquad
A_{ii}=0.
\]
Then $A$ has strictly Lorentzian signature $(1,n-1)$.
\end{theorem}

It is useful to separate this statement from the following closely related ordinary
tree-distance theorem.

\begin{theorem}[Ordinary tree-distance theorem]
\label{thm:tree-distance-Lorentzian}
Let $n\ge 2$, and let $d=d_T$ be a tree distance on $X=\{1,\dots,n\}$, arising from a weighted
tree $T$ with positive edge lengths whose leaves are labeled by $X$.  Let
$D\in \R^{n\times n}$ be the symmetric zero-diagonal matrix
\[
D_{ij}=d(i,j)\quad (i\neq j),
\qquad
D_{ii}=0.
\]
Then $D$ has strictly Lorentzian signature $(1,n-1)$.  In fact, $D$ is
conditionally strictly negative definite on
\[
H_0:=\left\{w\in \R^n: \sum_{i=1}^n w_i=0\right\}.
\]
\end{theorem}

Several of the proofs below first establish
Theorem~\ref{thm:tree-distance-Lorentzian} and then deduce
Theorem~\ref{thm:exp-tree-Lorentzian}.  The reduction is short.  By
Lemma~\ref{lem:tree-to-ultrametric} below, every finite tree metric can be written in
the form
\[
d(i,j)=u(i,j)+c_i+c_j,
\]
where $u$ is an ultrametric on $X$ and $c_1,\dots,c_n\in\R$.  Hence
\[
e^{d(i,j)}=e^{c_i}e^{c_j}e^{u(i,j)}.
\]
Thus the exponential distance matrix for $d$ is positive diagonally congruent to
the exponential distance matrix for $u$.  Since increasing functions preserve the
ultrametric property, the function
\[
v(i,j):=e^{u(i,j)}\quad (i\neq j), \qquad v(i,i):=0,
\]
is again an ultrametric.  Every finite ultrametric is an equidistant tree metric, so
Theorem~\ref{thm:tree-distance-Lorentzian} applied to the tree metric $v$ gives
Theorem~\ref{thm:exp-tree-Lorentzian}.

The purpose of this appendix is to collect several proofs of these tree-metric inputs.
Proofs \#2--\#5 establish the ordinary tree-distance theorem, and therefore also the
exponential theorem by the preceding reduction.  Proofs \#1 and \#6 prove the
exponential theorem directly.

\subsection{How the exponential tree-metric theorem enters the proof}
\label{app:reductions}

The following result is proved in \cite{Branden-Huh20}; we recall a self-contained
variant of the argument here in order to clarify where the exponential theorem enters.

\begin{proposition}
\label{prop:appendix-matroid-reduction}
Assume Theorem~\ref{thm:exp-tree-Lorentzian}.  Then for every matroid $M$, every
weak $\T_0$-representation of $M$ has Lorentzian generating polynomial.
\end{proposition}

\begin{proof}
Let $M$ be a matroid of rank $r$ on the ground set $E$, and let
\[
f_\rho(x):=\sum_{B\in \mathcal B(M)} \rho(B)x^B
\]
be the multi-affine generating polynomial attached to a weak $\T_0$-representation
$\rho$ of $M$.

\emph{Step 1: reduction to rank $2$ by contraction.}
For an independent set $S\subseteq E$, the squarefree derivative $\partial^S f_\rho$
is the generating polynomial of the induced weak $\T_0$-representation of the
contraction $M/S$.  By Theorem~\ref{thm:BH-limit-free}, it is therefore enough to
prove the claim for rank-$2$ contractions of $M$.

\emph{Step 2: reduction from loopless rank-$2$ matroids to uniform ones.}
Let $N$ be a loopless rank-$2$ matroid, with parallel classes
\[
V_1,\dots,V_m.
\]
Exactly as in the proof of Theorem~\ref{thm:matroid-sandwich}, the weak
$\T_0$-relations imply a factorization
\[
\rho_{ij}=t_i t_j \sigma_{ab}
\qquad (i\in V_a,\ j\in V_b,
\ a\neq b),
\]
for suitable $t_i>0$, where $\sigma$ is a weak $\T_0$-representation of the uniform
matroid $U_{2,m}$.  The corresponding quadratic polynomial satisfies
\[
f_\rho(x)=g_\sigma\!\left(\sum_{i\in V_1} t_i x_i,\dots,
\sum_{i\in V_m} t_i x_i\right).
\]
As in Section~\ref{subsec:matroid-reduction}, this linear substitution preserves the
property of having at most one positive eigenvalue.  Thus it remains to prove the
claim for $U_{2,m}$.

\emph{Step 3: passage from $U_{2,m}$ to tree metrics.}
For $U_{2,m}$, the weak $\T_0$-relations say exactly that the maximum among
\[
\rho_{ij}\rho_{k\ell},
\qquad
\rho_{ik}\rho_{j\ell},
\qquad
\rho_{i\ell}\rho_{jk}
\]
is attained at least twice for every quartet.  Equivalently, after setting
\[
w_{ij}:=\log \rho_{ij},
\]
the vector $w=(w_{ij})$ is a rank-$2$ tropical Pl\"ucker vector.  Such vectors are
precisely tree metrics up to the lineality space: there are real numbers
$c_1,\dots,c_m$ and a tree distance $d$ on $[m]$ such that
\[
w_{ij}=d(i,j)+c_i+c_j.
\]
Consequently,
\[
\rho_{ij}=e^{w_{ij}}=e^{c_i}e^{c_j}e^{d(i,j)}.
\]
The Hessian of $g_\sigma$ is therefore positive diagonally congruent to the exponential
tree-distance matrix $(e^{d(i,j)})$.  By
Theorem~\ref{thm:exp-tree-Lorentzian}, this matrix has strictly Lorentzian signature.
Hence $g_\sigma$ is Lorentzian.  This proves the rank-$2$ uniform case, and hence
the proposition.
\end{proof}

\begin{remark}
\label{rem:appendix-reduction}
Combined with the quadratic reduction argument of
Section~\ref{subsec:general-from-matroids}, Proposition~\ref{prop:appendix-matroid-reduction} yields the full containment
\[
\upR_J^{\rm w}(\T_0)\subseteq \upL_J
\]
for arbitrary $M$-convex sets $J$.
\end{remark}

\subsection{Background and terminology}
\label{app:background}

\begin{definition}
A symmetric matrix $H$ is \emph{conditionally negative semidefinite} if
\[
w^\top H w \le 0
\qquad\text{for all } w\in H_0.
\]
It is \emph{conditionally strictly negative definite} if in addition
\[
w\in H_0,
\quad
w^\top H w=0
\qquad\Longrightarrow\qquad
w=0.
\]
\end{definition}

\begin{definition}
A metric $u$ on a finite set $X$ is an \emph{ultrametric} if
\[
u(i,j)\le \max\{u(i,k),u(j,k)\}
\qquad\text{for all } i,j,k\in X.
\]
\end{definition}

We record two useful elementary facts about the relationship between weighted trees and ultrametrics.

\begin{lemma}
\label{lem:equidistant-tree-ultrametric}
Let $T$ be a rooted weighted tree with leaf set $X$, and suppose that every leaf lies
at the same distance $h$ from the root.  Then the leaf-to-leaf distance function
$u_T$ on $X$ is an ultrametric.
\end{lemma}

\begin{proof}
For leaves $i,j\in X$, let $\operatorname{lca}(i,j)$ denote their lowest common
ancestor.  Then
\[
u_T(i,j)=2\bigl(h-h(\operatorname{lca}(i,j))\bigr),
\]
where $h(v)$ denotes the distance from the root to the vertex $v$.  Given three
leaves $i,j,k$, at least two of the three lowest common ancestors
\[
\operatorname{lca}(i,j),\quad
\operatorname{lca}(i,k),\quad
\operatorname{lca}(j,k)
\]
coincide and lie no deeper than the third.  It follows that two of the three distances
\[
u_T(i,j),\quad u_T(i,k),\quad u_T(j,k)
\]
are equal and dominate the third.  This is exactly the ultrametric inequality.
\end{proof}

\begin{lemma}
\label{lem:tree-to-ultrametric}
Let $d$ be a tree distance on the leaf set $X=\{1,\dots,n\}$.  Then there exist real
numbers $c_1,\dots,c_n$ and an ultrametric $u$ on $X$ such that
\[
d(i,j)=u(i,j)+c_i+c_j
\qquad (i\neq j).
\]
Moreover, if $E_d$ and $E_u$ are the associated zero-diagonal exponential matrices,
defined by
\[
(E_d)_{ij}=e^{d(i,j)},\qquad (E_u)_{ij}=e^{u(i,j)}
\qquad (i\neq j),
\]
and with diagonal entries equal to $0$, then
\[
E_d=C E_u C,
\qquad
C=\operatorname{diag}(e^{c_1},\dots,e^{c_n}).
\]
In particular, $E_d$ and $E_u$ have the same inertia.
\end{lemma}

\begin{proof}
Choose a root $r$ of the given tree $T$, and for each leaf $i\in X$ let
\[
h_i:=d_T(r,i).
\]
Let
\[
H:=\max_{i\in X} h_i,
\qquad
c_i:=h_i-H\le 0.
\]
Now extend the tree by attaching to each leaf $i$ an extra segment of length
$-c_i=H-h_i$.  In the resulting rooted tree, every leaf lies at distance $H$ from
the root.  Let $u$ denote the leaf-to-leaf distance function in the extended tree.
Then, for $i\neq j$,
\[
u(i,j)=d(i,j)-c_i-c_j,
\]
or equivalently
\[
d(i,j)=u(i,j)+c_i+c_j.
\]
By Lemma~\ref{lem:equidistant-tree-ultrametric}, the metric $u$ is an ultrametric.

For $i\neq j$, the preceding identity gives
\[
(E_d)_{ij}=e^{d(i,j)}
=e^{c_i}e^{u(i,j)}e^{c_j}.
\]
Both $E_d$ and $E_u$ have zero diagonal entries, so this is precisely the matrix
identity
\[
E_d=C E_u C,
\qquad
C=\operatorname{diag}(e^{c_1},\dots,e^{c_n}).
\]
Since $C$ is invertible, Sylvester's law of inertia shows that $E_d$ and $E_u$
have the same inertia.
\end{proof}

\subsection{Proof \#1: Schoenberg plus ultrametric embedding}
\label{app:schoenberg-proof}

This is the proof employed by Br\"and\'en and Huh in \cite{Branden-Huh20}.
It relies on the following classical fact, cf.~\cite{TimanVestfrid}.

\begin{theorem}[Timan--Vestfrid]
\label{thm:Timan-Vestfrid}
Every finite ultrametric space $(X,u)$ admits an isometric embedding of
$(X,\sqrt u)$ into Euclidean space.
\end{theorem}

\begin{proof}[Proof of Theorem~\ref{thm:exp-tree-Lorentzian}, first proof]
Let $d$ be a tree distance on $X$, and let
\[
A_{ij}=e^{d(i,j)}\quad (i\neq j),\qquad A_{ii}=0.
\]
By Lemma~\ref{lem:tree-to-ultrametric}, there exist real numbers $c_i$ and an
ultrametric $u$ on $X$ such that
\[
d(i,j)=u(i,j)+c_i+c_j
\qquad (i\neq j).
\]
Let $C=\operatorname{diag}(e^{c_1},\dots,e^{c_n})$, and let $B$ be the
zero-diagonal matrix with off-diagonal entries
\[
B_{ij}=e^{u(i,j)}\qquad (i\neq j).
\]
Then $A=CBC$, so it suffices to prove that $B$ has strictly Lorentzian signature.

Since the exponential function is increasing, the function
\[
v(i,j):=e^{u(i,j)}\quad (i\neq j),\qquad v(i,i):=0
\]
is again an ultrametric.  By Theorem~\ref{thm:Timan-Vestfrid}, the finite metric
space $(X,\sqrt v)$ embeds isometrically into Euclidean space.  For finite
ultrametrics, one may moreover choose the embedding to be affinely independent;
this is made completely explicit in Section~\ref{app:pythagorean-proof} below.
Hence Schoenberg's criterion, Lemma~\ref{lem:q5-schoenberg}, shows that $B$ is
conditionally strictly negative definite on $H_0$.

Finally,
\[
\mathbf 1^\top B\mathbf 1=2\sum_{i<j}v(i,j)>0,
\]
so $B$ has exactly one positive eigenvalue and is negative definite on $H_0$.
Thus $B$ has inertia $(1,n-1)$.  Since $A=CBC$ with $C$ positive diagonal,
Sylvester's law of inertia implies that $A$ also has inertia $(1,n-1)$.
\end{proof}

\subsection{Proof \#2: Explicit Pythagorean embedding of rooted trees}
\label{app:pythagorean-proof}

The Timan--Vestfrid embeddability proof in \cite{TimanVestfrid} is not constructive.
We now give a completely explicit Euclidean embedding; our construction is a metric enhancement of \cite[Theorem 3.6]{MaeharaSurvey}.

\begin{proof}[Proof of Theorem~\ref{thm:tree-distance-Lorentzian}, second proof]
Let $T=(V,E,w)$ be a weighted tree whose leaves are $X=\{1,\dots,n\}$.  Choose a
root $r$ which is not a leaf; if necessary, subdivide an edge and take the new vertex as
root.  For every edge $e\in E$, let $e^\ast$ denote the corresponding standard basis
vector of $\R^E$.  Define
\[
\Phi(v):=\sum_{e\in [r,v]} \sqrt{w(e)}\,e^\ast
\qquad (v\in V),
\]
where $[r,v]$ denotes the unique path from $r$ to $v$.

If $x,y\in V$, then the vectors $\Phi(x)$ and $\Phi(y)$ agree on the common initial
segment of the two root paths and differ exactly on the edges of the path $[x,y]$.
Therefore
\[
\|\Phi(x)-\Phi(y)\|^2
=
\sum_{e\in [x,y]} w(e)
=
d_T(x,y).
\]
In particular, for leaves $i,j\in X$,
\[
\|\Phi(i)-\Phi(j)\|^2=d(i,j).
\]
Thus $(X,\sqrt d)$ embeds isometrically into the Euclidean space $\R^E$.

We claim that the leaf images $\Phi(1),\dots,\Phi(n)$ are linearly independent.
Indeed, for each leaf $i$, let $e_i$ be the pendant edge incident to $i$.  Because the
root is not a leaf, the edge $e_i$ lies on the path from $r$ to $i$.  The coordinate of
$\Phi(i)$ in the $e_i$-direction is $\sqrt{w(e_i)}$, whereas the same coordinate
vanishes for $\Phi(j)$ when $j\neq i$.  Hence any linear relation among the
$\Phi(i)$ forces all coefficients to vanish.

In particular, the points $\Phi(1),\dots,\Phi(n)$ are affinely independent.  By
Schoenberg's criterion, the zero-diagonal matrix $D=(d(i,j))$ is
conditionally strictly negative definite on $H_0$.  As in the first proof, this implies
that $D$ has inertia $(1,n-1)$.
\end{proof}

\subsection{Proof \#3: Split decomposition}
\label{app:split-proof}

This proof utilizes the canonical split decomposition of a tree metric, which is a standard and imortant tool in phylogenetics.

Let $X=\{1,\dots,n\}$.  A \emph{split} of $X$ is a bipartition
\[
\sigma=\{A_\sigma,B_\sigma\}
\]
with $A_\sigma,B_\sigma\neq\varnothing$.  The associated split metric is
\[
\delta_\sigma(i,j)
=
\begin{cases}
1, & \text{if } i\in A_\sigma,\ j\in B_\sigma \text{ or vice versa},\\
0, & \text{otherwise}.
\end{cases}
\]

If $T$ is a weighted tree with leaf set $X$, then every edge $e$ determines a split
$\sigma_e$ by deleting $e$, and the tree distance admits the canonical decomposition
\[
d=
\sum_{e\in E(T)} w(e)\,\delta_{\sigma_e},
\]
see for example \cite{SempleSteel2003}.
Equivalently, one may group together equal splits and write
\[
d=\sum_{\sigma\in \mathcal S(X)} \lambda_\sigma\,\delta_\sigma,
\qquad
\lambda_\sigma\ge 0,
\]
where $\mathcal S(X)$ denotes the set of splits of $X$.

\begin{proposition}
\label{prop:split-CND}
Let $d$ be a tree distance on $X$, written in split form as above.  Then for every
$c=(c_i)_{i\in X}\in\R^X$ satisfying $\sum_i c_i=0$, one has
\begin{equation}
\label{eq:split-identity}
\sum_{i,j\in X} c_i c_j d(i,j)
=
-2\sum_{\sigma\in\mathcal S(X)} \lambda_\sigma
\left(\sum_{i\in A_\sigma} c_i\right)^2.
\end{equation}
In particular, $d$ is conditionally negative semidefinite.  Moreover, the right-hand
side vanishes only when $c=0$, so $d$ is conditionally strictly negative definite.
\end{proposition}

\begin{proof}
Fix a split $\sigma=\{A,B\}$.  Since $\delta_\sigma(i,j)=1$ exactly on
$A\times B$ and $B\times A$, we have
\[
\sum_{i,j} c_i c_j\delta_\sigma(i,j)
=
2\sum_{i\in A,\ j\in B} c_i c_j
=
2\left(\sum_{i\in A}c_i\right)\left(\sum_{j\in B}c_j\right).
\]
Because $\sum_i c_i=0$, we have $\sum_{j\in B}c_j=-\sum_{i\in A}c_i$, so
\[
\sum_{i,j} c_i c_j\delta_\sigma(i,j)
=
-2\left(\sum_{i\in A}c_i\right)^2.
\]
Multiplying by $\lambda_\sigma$ and summing over all splits gives
\eqref{eq:split-identity}.  This proves conditional negative semidefiniteness.

Now assume $c\in H_0$ and the right-hand side of \eqref{eq:split-identity} vanishes.
Then
\[
\sum_{i\in A_\sigma} c_i=0
\qquad\text{for every edge-split }\sigma.
\]
Choose a leaf $i$.  The split corresponding to the pendant edge at $i$ is
$\{\{i\},X\setminus\{i\}\}$, so the displayed relation yields $c_i=0$.  Since every
leaf appears in such a split, all coordinates of $c$ vanish.  Thus the form is strictly
negative definite on $H_0$.
\end{proof}

\begin{proof}[Proof of Theorem~\ref{thm:tree-distance-Lorentzian}, third proof]
Proposition~\ref{prop:split-CND} shows that the zero-diagonal distance matrix $D$ is
conditionally strictly negative definite on $H_0$.  As before, this implies that $D$
has at most one positive eigenvalue, while $\mathbf 1^\top D\mathbf 1>0$
shows that it has at least one positive eigenvalue.  
As before, we conclude that the inertia is $(1,n-1)$.
\end{proof}

\subsection{Proof \#4: Determinant formulas for tree distance matrices}
\label{app:determinant-proof}

We next give a proof based on determinant formulas for tree distance matrices.
The motivating result is:

\begin{theorem}[Bapat--Kirkland--Neumann]
\label{thm:BKN-distance-determinant}
Let $T=(V,E,w)$ be a weighted tree with $|V|=m\ge 2$, and let $D_T$ be its full
vertex--vertex distance matrix:
\[
(D_T)_{uv}=d_T(u,v),
\qquad
(D_T)_{uu}=0.
\]
Then
\[
\det(D_T)
=
(-1)^{m-1}2^{m-2}
\left(\sum_{e\in E}w(e)\right)\prod_{e\in E}w(e).
\]
In particular, $D_T$ is nonsingular and
\[
\operatorname{sign}\det(D_T)=(-1)^{m-1}.
\]
\end{theorem}

\begin{proof}
This is \cite[Corollary 2.5]{BKN05}; see also the earlier unweighted result of
Graham--Pollak~\cite{GP71}.
\end{proof}

The Bapat--Kirkland--Neumann formula applies to the full vertex--vertex distance
matrix of a tree.  In the present application we need signs of principal minors of a
leaf-to-leaf distance matrix.  For this we use the principal-minor formula of
Richman--Shokrieh--Wang~\cite{RSW24}, which directly generalizes the
Graham--Pollak and Bapat--Kirkland--Neumann determinant formulas in the direction
needed here. The following is a particular consequence of their general formula:

\begin{theorem}
\label{thm:tree-leaf-principal-minors}
Let $T$ be a weighted tree whose leaf set is $X=\{1,\dots,n\}$, and let
\[
D=(d_T(i,j))_{1\le i,j\le n},
\qquad
D_{ii}=0,
\]
be the leaf-to-leaf distance matrix.  Then for every subset $S\subseteq X$ with
$|S|=k\ge 2$,
\[
\operatorname{sign}\det(D[S])=(-1)^{k-1}.
\]
In particular, $D[S]$ is nonsingular.
\end{theorem}

\begin{proof}
This is a special case of the principal-minor formula of
Richman--Shokrieh--Wang~\cite{RSW24} for tree distance matrices.  Applied to the
principal submatrix indexed by $S$, their formula expresses $\det(D[S])$ as
$(-1)^{k-1}$ times a strictly positive sum of monomials in the edge lengths of the minimal subtree
spanned by $S$.  Since all edge lengths are positive, the asserted sign and
nonsingularity follow.
\end{proof}

\begin{proof}[Proof of Theorem~\ref{thm:tree-distance-Lorentzian}, fourth proof]
Let
\[
D=(d_T(i,j))_{1\le i,j\le n},
\qquad
D_{ii}=0,
\]
be the leaf-to-leaf distance matrix of $T$.

For every principal submatrix $N=D[S]$ of size $k\ge 2$,
Theorem~\ref{thm:tree-leaf-principal-minors} gives
\[
\operatorname{sign}\det(N)=(-1)^{k-1},
\]
and hence
\[
(-1)^k\det(N)<0.
\]
For $k=1$, the principal minors are zero.  Therefore all principal minors of $D$
satisfy the hypothesis of Theorem~\ref{thm:principal-minor-criterion}, so $D$ is
Lorentzian.
Moreover, Theorem~\ref{thm:tree-leaf-principal-minors} applied to $S=X$ shows that
$\det(D)\neq 0$, so $D$ has no zero eigenvalues.  
As before, it follows that the inertia of $D$ is $(1,n-1)$.
\end{proof}

\subsection{Proof \#5: Potential theory on metric graphs}
\label{app:potential-theory-proof}

We next give a proof using potential theory on metric graphs (see \cite{Baker-Faber06,Baker-Rumely07} for overviews, including all the background we need here).
Let $\Gamma$ be a metric graph, and let $r(x,y)$
denote the effective resistance between $x,y\in\Gamma$.

\begin{proposition}
\label{prop:effective-resistance-CND}
For any distinct points $x_1,\dots,x_n\in\Gamma$, the matrix
\[
H=\bigl(r(x_i,x_j)\bigr)_{1\le i,j\le n}
\]
is conditionally strictly negative definite on
\[
H_0=\left\{w\in\R^n:\sum_{i=1}^n w_i=0\right\}.
\]
\end{proposition}

\begin{proof}
Let $j_z(x,y)$ be the $j$-function on $\Gamma$; see
\cite[Corollary 3]{Baker-Faber06}.  For a probability measure $\mu$ on $\Gamma$,
define
\[
j_\mu(x,y):=\int_\Gamma j_z(x,y)\,d\mu(z).
\]
The energy pairing on the space $\operatorname{Meas}_0(\Gamma)$ of signed measures of total mass zero is defined by
\[
\langle \nu,\omega\rangle
:=
\iint_{\Gamma\times\Gamma} j_\mu(x,y)\,d\nu(x)\,d\omega(y),
\]
and is independent of the choice of $\mu$; see \cite[Section 10]{Baker-Rumely07}.
By \cite[Theorem 10.4]{Baker-Rumely07}, this pairing is positive definite on
$\operatorname{Meas}_0(\Gamma)$.

Moreover, there are a canonical probability measure $\mu_{\mathrm{can}}$ on $\Gamma$
and a constant $c(\Gamma)$ such that
\[
j_{\mu_{\mathrm{can}}}(x,y)
=
-\frac12 r(x,y)+c(\Gamma);
\]
see \cite[Theorem 14.1]{Baker-Rumely07}.  Therefore, for every
$\nu\in\operatorname{Meas}_0(\Gamma)$,
\[
\iint_{\Gamma\times\Gamma} r(x,y)\,d\nu(x)\,d\nu(y)
=
-2\langle \nu,\nu\rangle
\le 0,
\]
with equality if and only if $\nu=0$.

Now let $w=(w_1,\dots,w_n)\in H_0$, and define
\[
\nu:=\sum_{i=1}^n w_i\delta_{x_i}\in\operatorname{Meas}_0(\Gamma).
\]
Then
\[
w^\top H w
=
\sum_{i,j=1}^n r(x_i,x_j)w_iw_j
=
\iint_{\Gamma\times\Gamma} r(x,y)\,d\nu(x)\,d\nu(y)
=
-2\langle \nu,\nu\rangle.
\]
Thus $w^\top H w\le 0$.  If equality holds, then $\nu=0$.  Since the points
$x_1,\dots,x_n$ are distinct, this implies $w=0$.  Hence $H$ is conditionally
strictly negative definite on $H_0$.
\end{proof}

\begin{proof}[Proof of Theorem~\ref{thm:tree-distance-Lorentzian}, fifth proof]
Take $\Gamma=T$, viewed as a metric graph.  On a tree, effective resistance agrees
with path distance, so for leaves $i,j\in X$,
\[
r(i,j)=d_T(i,j).
\]
Hence the leaf distance matrix $D$ is conditionally strictly negative definite on
$H_0$ by Proposition~\ref{prop:effective-resistance-CND}.
As before, it follows that the inertia of $D$ is $(1,n-1)$.
\end{proof}

\subsection{Proof \#6: Tropicalization and logarithmic limit sets of \texorpdfstring{$\Gr(2,n)$}{Gr(2,n)}}
\label{app:tropical-proof}

We conclude with a direct proof of the exponential tree-metric theorem using ideas from tropical
geometry. This is the most novel of our six proofs.

Let $N:=\binom n2$ and let $\Gr(2,n)^\circ\subset (\C^\times)^N$
denote the intersection of the Pl\"ucker embedding of $\Gr(2,n)$ with the dense torus.
We write
\[
\upL_{U_{2,n}}^\times\subset (\R_{>0})^N
\]
for the strictly Lorentzian stratum, i.e., the set of positive coefficient vectors
\[
(p_{ij})_{1\le i<j\le n}
\]
such that the quadratic form
\[
\sum_{1\le i<j\le n} p_{ij}x_i x_j
\]
is strictly Lorentzian.

For a subset $S\subset (\R_{>0})^N$, define its \emph{logarithmic limit set} by
\[
\operatorname{LLS}(S)
:=
\lim_{t\to 0^+}\operatorname{Log}_{1/t}(S),
\]
where
\[
\operatorname{Log}_a(x_1,\dots,x_N)
:=
(\log_a x_1,\dots,
\log_a x_N).
\]

\begin{lemma}
\label{lem:LLS-star-shaped}
If $\operatorname{Log}(S)$ is closed and star-shaped with respect to the origin, then
\[
\operatorname{LLS}(S)\subseteq \operatorname{Log}(S).
\]
\end{lemma}

\begin{proof}
Let $y\in\operatorname{LLS}(S)$.  By definition, there exist $x_m\in S$ and
$t_m\to 0^+$ such that
\[
y_m:=\operatorname{Log}_{1/t_m}(x_m)\to y.
\]
But
\[
\operatorname{Log}_{1/t}(x)=\frac{\operatorname{Log}(x)}{\log(1/t)}.
\]
Since $\log(1/t_m)\ge 1$ for $m$ large, star-shapedness of $\operatorname{Log}(S)$
implies that each $y_m$ lies in $\operatorname{Log}(S)$.  Since
$\operatorname{Log}(S)$ is closed, we obtain $y\in\operatorname{Log}(S)$.
\end{proof}

For a complex algebraic variety $V\subset (\C^\times)^N$, define its amoeba by
\[
\mathcal A(V):=\operatorname{Log}\bigl(|V(\C)|\bigr)\subset \R^N,
\]
and its large-scale limit by
\[
\mathcal A_0(V):=\operatorname{LLS}\bigl(|V(\C)|\bigr).
\]

We will use the following facts:
\begin{enumerate}
\item By \cite[Theorem 5.1]{BHKL1}, the map
\[
(z_{ij})\longmapsto (|z_{ij}|^2)
\]
sends $\Gr(2,n)^\circ$ into $\upL_{U_{2,n}}^\times$.  Equivalently,
\[
2\mathcal A(\Gr(2,n)^\circ)
\subseteq
\operatorname{Log}(\upL_{U_{2,n}}^\times).
\]
\item By \cite[Theorem 4.4]{BHKL1}, the set
\[
\operatorname{Log}(\upL_{U_{2,n}}^\times)
\]
is closed and star-shaped with respect to the origin.
\item By \cite[Theorem A]{Jonsson16},
\[
\mathcal A_0(V)=\operatorname{trop}(V)
\]
for every $V\subset (\C^\times)^N$, where the tropicalization is computed with respect to the trivial valuation on $\C$.
\item By \cite[Theorem 3.4]{Speyer-Sturmfels04},
\[
\operatorname{trop}(\Gr(2,n)^\circ)=\operatorname{Dr}_{U_{2,n}},
\]
the Dressian of $U_{2,n}$, i.e., the collection of all rank-$2$ tropical Pl\"ucker vectors.
\item By \cite[Theorem 4.3.5]{Maclagan-Sturmfels15}, the Dressian $\operatorname{Dr}_{U_{2,n}}$ coincides with the space of tree metrics on $[n]$.
\end{enumerate}

\begin{proposition}
\label{prop:tropical-Dressian-inclusion}
We have
\[
\operatorname{Dr}_{U_{2,n}}
\subseteq
\operatorname{Log}(\upL_{U_{2,n}}^\times).
\]
\end{proposition}

\begin{proof}
Let
\[
A:=\mathcal A(\Gr(2,n)^\circ),
\qquad
A_0:=\mathcal A_0(\Gr(2,n)^\circ).
\]
By (1),
\[
2A\subseteq \operatorname{Log}(\upL_{U_{2,n}}^\times).
\]
Taking logarithmic limit sets and using (2) together with
Lemma~\ref{lem:LLS-star-shaped}, we obtain
\[
2A_0\subseteq \operatorname{Log}(\upL_{U_{2,n}}^\times).
\]
Since $\operatorname{Log}(\upL_{U_{2,n}}^\times)$ is star-shaped with respect to the
origin, it follows that
\[
A_0\subseteq \operatorname{Log}(\upL_{U_{2,n}}^\times).
\]
Now (3) and (4) give
\[
A_0
=
\operatorname{trop}(\Gr(2,n)^\circ)
=
\operatorname{Dr}_{U_{2,n}}.
\]
Therefore
\[
\operatorname{Dr}_{U_{2,n}}
\subseteq
\operatorname{Log}(\upL_{U_{2,n}}^\times),
\]
as claimed.
\end{proof}

\begin{proof}[Proof of Theorem~\ref{thm:exp-tree-Lorentzian}, sixth proof]
If $d=(d_{ij})_{1\le i<j\le n}$ is a tree metric then by (5) we have
$d\in \operatorname{Dr}_{U_{2,n}}$.
By Proposition~\ref{prop:tropical-Dressian-inclusion}, we have
\[
d\in \operatorname{Log}(\upL_{U_{2,n}}^\times).
\]
Therefore
\[
(e^{d_{ij}})_{i<j}\in \upL_{U_{2,n}}^\times,
\]
which is equivalent to the assertion that the symmetric
zero-diagonal matrix
\[
A_{ij}=e^{d_{ij}}\quad(i\neq j),
\qquad
A_{ii}=0
\]
has inertia $(1,n-1)$.
\end{proof}

\begin{small}
 \bibliographystyle{plain}
 \bibliography{lorentzian}
\end{small}

\end{document}